\documentclass[reqno,letterpaper]{amsart}
\usepackage{amsmath,amssymb,amsthm,graphicx,mathrsfs,url}
\usepackage[usenames,dvipsnames]{color}
\usepackage[colorlinks=true,linkcolor=Red,citecolor=Green]{hyperref}
\usepackage{amsxtra}
\usepackage[utf8]{inputenc}
\usepackage{soul}
\usepackage[normalem]{ulem}
	\usepackage{wasysym}

\usepackage{multirow}

\def\?[#1]{\textbf{[#1]}\marginpar{\Large{\textbf{??}}}}

\def\smallsection#1{\smallskip\noindent\textbf{#1}.}
\let\epsilon=\varepsilon 

\newcommand{\RR}{{\mathbb R}}
\newcommand{\NN}{{\mathbb N}}
\newcommand{\CC}{{\mathbb C}}

\newcommand{\ZZ}{{\mathbb Z}}

\def\me{\mathsf{e}}

\def\mv{\mathsf{v}}

\newtheorem{theo}{Theorem}
\newtheorem{prop}{Proposition}[section]	
\newtheorem{ex}[prop]{Example}
\newtheorem{defi}[prop]{Definition}

\newtheorem{Assumption}{Assumption}
\newtheorem{lemm}[prop]{Lemma}
\newtheorem{corr}[prop]{Corollary}
\newtheorem{rem}[prop]{Remark}
\numberwithin{equation}{section}

\let\Im=\Imag

\let\Re=\Real
\DeclareMathOperator{\sgn}{sgn}

\DeclareMathOperator{\supp}{supp}

\def\indic{\operatorname{1\hskip-2.75pt\relax l}}

\definecolor{purple}{rgb}{0.6, 0.4, 0.8}

\title[Schrödinger and polyharmonic operators on infinite graphs]{Schrödinger and polyharmonic operators on infinite graphs: Parabolic well-posedness and $p$-independence of spectra}
\author{Simon Becker}
\email{simon.becker@damtp.cam.ac.uk}
\address{University of Cambridge,
DAMTP, Wilberforce Rd, Cambridge CB3 0WA, UK}
\author{Federica Gregorio}
\email{fgregorio@unisa.it}
\address{Universita' degli Studi di Salerno, Via Giovanni Paolo II, 132 Fisciano 84084 Sa, Italy}
\author{Delio Mugnolo}
\email{delio.mugnolo@fernuni-hagen.de}
\address{FernUniversit\"at in Hagen, Lehrgebiet Analysis, 58084 Hagen, Germany}


\thanks{S.B.\ gratefully acknowledges support by the UK Engineering and Physical Sciences Research Council (EPSRC) grant EP/L016516/1 for the University of Cambridge Centre for Doctoral Training.
F.G.\ is member of the Gruppo Nazionale per l'Analisi Matematica, la Probabilità e le loro Applicazioni (GNAMPA) of the Istituto Nazionale di Alta Matematica (INdAM).
The work of D.M.\ was supported by the Deutsche Forschungsgemeinschaft (Grant 397230547). F.G.\ and D.M.\ would like to acknowledge networking support by the COST Action CA18232.
}

\begin{document}

\begin{abstract}
We analyze properties of semigroups generated by Schrödinger operators $-\Delta+V$ or polyharmonic operators $-(-\Delta)^m$, on metric graphs both on $L^p$-spaces and spaces of continuous functions. In the case of spatially constant potentials, we provide a semi-explicit formula for their kernel. Under an additional sub-exponential growth condition on the graph, we prove analyticity, ultracontractivity, and pointwise kernel estimates for these semigroups; we also show that their generators' spectra coincide on all relevant function spaces and present a Kre\u{\i}n-type dimension reduction, showing that their spectral values are determined by the spectra of generalized discrete Laplacians acting on various spaces of functions supported on combinatorial graphs.
\end{abstract}

\maketitle

\section{Introduction}

Differential operators on so-called \textit{metric graphs} have been introduced in the mathematical literature by Lumer in~\cite{Lum80}, cf.~\cite[\S~2.5]{Mug14} for references to earlier investigations in the applied sciences.
Most of ongoing research is devoted to the case of \textit{finite graphs}, i.e., of metric graphs whose underlying discrete structure consists of finitely many vertices and edges.

Restricting to the case of finite graphs largely simplified earlier investigations: it is nowadays known 
 that the theory of Laplace-type operators becomes notably subtler in the infinite case, as recently shown by e.g.~\cite{HaeKelLen12,KelLen12} for the case of \textit{difference operators} on combinatorial graphs and \cite{KosMugNic19} for Schrödinger operators on metric graph. {The first and most obvious side effect is that} on general infinite graphs the compactness of the resolvent may be lost and the spectral properties of the Laplacian on different $L^p$-spaces and Banach spaces of continuous functions may no longer coincide. {Most importantly, one may have to take care of boundary growing at infinity and, in turn, of non-uniqueness of solutions of diffusion equations.}


However, most existing results heavily rely on Hilbert space machinery; so far, to the best of our knowledge not much attention has been devoted to the case of linear differential operators acting on spaces of continuous functions over an infinite metric graph $\mathcal G$. For results on finite graphs, see \cite{Lum80,Bel88,Mug07}. The case of infinite graphs that we address in this paper is subtler, as spaces $C(\mathcal G)$ are no longer contained in $L^2(\mathcal G).$
To mention just one important difference: Unlike in the case of the Hilbert space $L^2(\mathcal G)$ there is in general no strongly continuous heat semigroup acting on $C(\mathcal G)$, { as already the trivial case of $\mathcal G=\mathbb R$ shows}. 
{However, among other things we are going to show in this article} that both the space of continuous functions vanishing at infinity and the space of bounded uniformly continuous functions do yield strongly continuous heat semigroups.

Indeed, we are going to show that polyharmonic operators and a broad class of Schrödinger operators drive well-posed parabolic equations on suitable spaces of continuous functions acting on general, possibly non-equilateral infinite graphs, as well as on more usual $L^p$-spaces. Our way to prove this is unusual and, we believe, among the main points of interest of our article: instead of checking that the assumptions of Hille--Yosida-type results are satisfied, we are able to provide a semi-explicit formula for the integral kernel of the relevant semigroups, in the case of spatially constant potentials. We do so by developing a convolution calculus for functions supported on graphs and showing that semigroup kernels on $\RR$ extend to semigroup kernels on infinite graphs under rather mild summability assumptions.
An analogous (semi-)explicit formula for the solution to the \textit{heat} equation on \textit{finite} metric graphs has been found by Roth in~\cite{Rot83} and then extended by Cattaneo to \textit{infinite equilateral trees} and finally to general \textit{equilateral} infinite graphs in~\cite{Cat98,Cat99}, respectively. 

The fundamental object of our theory is the so-called Kirchhoff Laplacian $\Delta$, i.e., the operator acting as second derivative on the space of edgewise smooth functions that satisfy continuity condition across each vertex while their normal derivatives sum up to zero about each vertex (Kirchhoff condition). We refer to Section~\ref{sec:DefNL} for a precise definition.

The explicit construction of semigroups generated by Schrödinger and polyharmonic operators is stated in Theorem \ref{theo:theo1} and presented in Section \ref{sec:convsem}. We then show in Theorem \ref{theo:hypercontr} that the above semigroups are also ultracontractive from which we conclude that the spectrum of Schrödinger operators is independent of $p \in [1,\infty)$ and coincides with the spectrum on $C_0(\mathcal G)$ and $BUC(\mathcal G)$, cf. Theorem \ref{theo:specind} and Corollary \ref{corr:specequiv}. 

In this article, we connect this explicit construction with the $L^p$-invariance properties of the spectrum of the generator investigated by Voigt and his co-authors in \cite{Stollmann1996,hempel1986,H,Hempel1994}. We extend the results for Schrödinger semigroups on $\RR^d$ to metric graphs and are able to conclude essentially from the spectral invariance a variety of semigroup properties -- including contractivity (Proposition \ref{prop:genercontr}), positivity (Proposition \ref{prop:pos}), analyticity (Theorem \ref{theo:inc}), the asymptotic behavior (Corollary \ref{corr:asymptotics}) and the (strong) Feller property (Corollary \ref{corr:Feller}) -- for the semigroup.

A reduction principle connecting the eigenvalues of the Kirchhoff Laplacian on a \textit{finite} metric graph to those of the normalized discrete Laplacian on the underlying combinatorial graph was first observed by von Below in~\cite{Bel85}. His results were extended to the different parts of the spectrum of more general operators on possibly infinite graphs, beginning with~\cite{Cat97}; we refer to~\cite{Pan12,LenPan16} and references therein for later refinements of Cattaneo's results. In Section \ref{sec:PS} we prove that this spectral correspondence between combinatorial and metric graphs carries over to general Schrödinger operators and can be extended to general $L^p$-spaces and certain spaces of continuous functions on the one hand and specific discrete spaces. These results are summarized in Theorem \ref{theo:spec}.

In Section \ref{sec:CS} we then show that we are also able, at least in the case of the Laplacian on equilateral metric graph, to identify the continuous spectrum with the continuous spectrum of the discrete Laplacian on the vertices. 

It should be mentioned that Exner, Kostenko, Malamud, and their co-authors have proved interesting connections between the Laplacian with $\delta$-interactions on possibly non-equilateral metric graphs and a certain generalized discrete Laplacian on the underlying combinatorial graph~\cite{KosMalNei17,ExnKosMal18,KosNic19} -- remarkably, not only at a spectral level, but also concerning parabolic properties: for instance, ultracontractivity of the semigroup generated by the former operator is equivalent to ultracontractivity of the semigroup generated by the latter, with equal dimension.

Sections \ref{sec:Markov} and \ref{sec:PAM} contain applications of our results. In Section \ref{sec:Markov} we construct a Feller process on metric graphs associated with the Laplacian as its generator, which can be interpreted as a Brownian motion on metric graphs, thus extending the construction carried out on finite graphs in~\cite{KosPotSch12}. We also derive a Feynman--Kac formula in Prop. \ref{prop:FK} and obtain pointwise heat kernel estimates in Theorem \ref{theo:kernel}. In Section \ref{sec:PAM}, we study the properties of a parabolic Anderson model (parabolic equation with random Schrödinger operator as its generator) on metric graphs.

\smallsection{Outline of article}
\begin{itemize}
\item Definition of Kirchhoff-Laplacian, Sec.~ \ref{sec:DefNL}.
\item Explicit construction of heat semigroup, Sec.~ \ref{sec:expconst}.
\item Spectral independence and Schrödinger semigroups, Sec.~ \ref{sec:Schrsemi}.
\item Properties of Schrödinger semigroups, Sec.~ \ref{sec:contraction}.
\item Point spectrum of Schrödinger operators, Sec.~ \ref{sec:PS}.
\item Continuous spectrum of the Kirchhoff-Laplacian, Sec.~ \ref{sec:CS}.
\item Markov semigroup and Brownian motion on metric graphs, Sec.~\ref{sec:Markov}.
\item Parabolic Anderson model on metric graphs, Sec.~\ref{sec:PAM}.
\end{itemize}

\smallsection{Acknowledgments}
The authors wish to thank Sebastian Mildner (TU Dresden) for fundamental contributions to this paper.

\smallsection{Assumptions \& Notation} 

\smallsection{The metric graph}
Let $\mathcal G=(\mathcal V, \mathcal E)$ be a countable oriented connected locally finite (i.e.~ the degree of every vertex is finite) graph without self-loops.
Here $\mathcal V$ is the vertex set and $\mathcal E$ the set of edges. The length of all edges $\me \in \mathcal E$ is constrained uniformly from above and below by constants $\ell_{\downarrow}$ and $\ell_{\uparrow}$
\[ 0<\ell_{\downarrow} \le \vert \me \vert \le \ell_{\uparrow} < \infty.\] 

We define $\mathcal E_\mv$ to be the set of edges adjacent to a vertex $\mv.$ Moreover, we assume that $\mathcal G$ is oriented such that every edge $\me$ gets assigned an initial $i(\me)$ and terminal vertex $t(\me)$. The cardinality of a vertex $\mv$ is defined as $d_{\mv} := \vert \mathcal E_{\mv} \vert$ and we assume $\mathcal G$ to be \textit{uniformly locally finite}, i.e., $d_{\mv}^*:=\sup_{\mv \in \mathcal V} d_{\mv}<\infty$.

We can, without loss of generality, identify all edges $\me$ with intervals $ (0,\vert \me \vert)$ such that $\mathcal G$ becomes naturally a metric space. For a graph with edges $\me \in \mathcal E$ we define the union of edges with either orientation, i.e., $\pm \me$ where initial and terminal vertex are interchanged, by $\widetilde{\mathcal G}.$ For $\me' \in \mathcal E(\widetilde{\mathcal G})$ we define $\me:=\vert \me' \vert$ with $\me \in \mathcal E$ to be the unique edge that either satisfies $\me = \me'$ or $\me=-\me'.$

\medskip

\medskip

A path $P$ along $m \in \mathbb N_0$ edges is a finite sequence of edges $(\me_0,...,\me_m) \subset \mathcal E(\widetilde{\mathcal G})^{m+1}$ such that $t(\me_j)=i(\me_{j+1})$ for each $0\le j \le m-1.$ We also define the set of paths along $m \in \mathbb N_0$-many edges between edges $\me$ and $\me'$ by $\mathcal P_{\me,\me'}(m)$. The length of such a path is defined as $\vert P \vert :=\sum_{i=0}^{m-1} \vert \me_i \vert.$ 

We assume that there is a non-zero conductivity $c(\me)$ associated with every edge $\me \in \mathcal E(\mathcal G).$ For the same edge with opposite orientation, $-\me$, we set $c(-\me):=c(\me).$
As the conductivity of a vertex $\mv$, we define $c(\mv):=\sum_{\me \in \mathcal E_\mv} c(\me)$ which is the sum of all conductivities of adjacent edges.

\medskip

When analyzing general graphs of the above type in this article satisfying the preceding assumptions, we sometimes have to impose the following assumption on the growth of the graph and conductivities in parts of Sections \ref{sec:Schrsemi} and \ref{sec:contraction} to conclude spectral equivalence.

\begin{Assumption}[Growth condition-General graphs]
\label{ass2}
We assume that there are two constants $k_{\downarrow},k_{\uparrow}$ such that all conductivities satisfy $k_{\downarrow} \le c(\me)\le k_{\uparrow}.$ Moreover, the number of edges is supposed to grow only sub-exponentially, i.e. let $k(r)$ be the number of edges contained in a ball or radius $r$ with respect to a fixed reference vertex of the graph, then for all $\varepsilon>0$ we have $k(r)/e^{\varepsilon r} = o(1)$ as $r \to \infty.$
\end{Assumption}

\smallsection{Mathematical notations.}
$B_r(x)$ is the ball of radius $r$ centered at $x.$

For $p \in [1,\infty),$ we consider the $L^p$-spaces with measure $dc:=\prod_{\me \in \mathcal E(\mathcal G)} c(\me) \ d\lambda_\me$ where $\lambda_\me$ is the standard Lebesgue measure on some edge. The dual of $L^p$ is then still $L^q$ with $p^{-1}+q^{-1}=1$ for $p \in [1,\infty) $ and $q \in (1,\infty]$ due to the dual pairing
\[(u,v)_{L^p \times L^q} = \int_{\mathcal G} u(x) \overline{v(x)} \ dc(x). \]
Instead of writing $dc(x)$, we will usually just write $dx$ in the sequel.

Analogously, the discrete $\ell^p$-spaces on the vertex set $\mathcal V(\mathcal G)$ are given by
\[ \ell^p(\mathcal V(\mathcal G)):=\left\{ z : \mathcal V(\mathcal G) \rightarrow \mathbb C; \Vert z \Vert_{\ell^p} := \left( \sum_{\mv \in \mathcal V(\mathcal G)} \vert z(\mv) \vert^p c(\mv) \right)^{1/p} < \infty \right\}\]
with dual pairing $(u,w)_{\ell^p\times \ell^q}:=\sum_{\mv \in \mathcal V(\mathcal G) } c(\mv) \left(u(\mv) \overline{w(\mv)}\right).$
We write $C(\mathcal G)$ for the space of continuous functions on the graph (requiring continuity also at the vertices).
The spaces of discrete and continuous functions 
\[c_0(\mathcal V(\mathcal G)),\ell^{\infty}(\mathcal V(\mathcal G))\text{ and }C_0(\mathcal G),BUC(\mathcal G), L^{\infty}(\mathcal G)\] 
are defined with the standard (i.e., unweighted) supremum norm using the metric structure of the graph and do not depend on conductivities.

As usual, $\ell^{\infty}$ can be identified with the double-dual of $c_0$ and the dual space of $\ell^1.$ Moreover, both $C_0$ and $BUC$ are closed subspaces of $L^{\infty}.$

Here, $C_0$ is the space of functions vanishing at \textit{infinity} and $BUC$ is the space of bounded uniformly continuous functions.

We denote the point spectrum of an operator $T$ by $\sigma_p(T):=\{\lambda \in \mathbb C: \operatorname{ker}(T-\lambda) \neq \{0\} \}$ and the continuous spectrum by 
\[\sigma_c(T):=\{ \lambda \in \mathbb C: T-\lambda \text{ is not bijective, but has dense range and is injective}\}.\]

We can view $\mathcal E(\mathcal G)$ as the collection of edges (without vertices), where $\mathcal E(\mathcal G)$ is naturally identified as a smooth disconnected Riemannian manifold.

The space of distributions $\mathscr D'(\mathcal E(\mathcal G)):=\bigoplus_{\me \in \mathcal E(\mathcal G)} \mathscr D'(\me)$ on $\mathcal E(\mathcal G)$ is the space of linear forms on $C_c^{\infty}(\mathcal E(\mathcal G))$ such that for every compact set $K \Subset \mathcal E(\mathcal G)$ there are $C>0$ and $k \in \mathbb N$ such that $\vert u(\varphi) \vert \le C \sum_{\vert \alpha \vert \le k} \Vert \varphi^{(\alpha)} \Vert_{\infty}$ for all $\varphi \in C_c^{\infty}(K).$
In particular, we can define the locally convex space $\mathscr E(\mathcal E(\mathcal G))$, the set of $C^{\infty}(\mathcal E(\mathcal G))$ functions with the topology defined by seminorms $\varphi \mapsto \sum_{\alpha \le k} \sup_{x \in K \Subset \mathcal E(\mathcal G)} \vert \varphi^{(\alpha)}(x) \vert,$
where $K$ ranges over all compact subsets of $\mathcal E(\mathcal G)$ and $k$ over all integers.

The dual space $\mathscr E'(\mathcal E(\mathcal G))$ is the space of compactly supported distributions on $\mathcal E(\mathcal G)$ that are of finite order. In fact, take $\nu \in \mathscr E'(\mathcal E(\mathcal G))$, then standard results on locally convex spaces imply that there are $C>0$, $k \in \mathbb N$ and $K \Subset \mathcal E(\mathcal G)$ such that $\vert \nu(\varphi) \vert \le C \sum_{\alpha \le k } \sup_{x \in K \Subset \mathcal E(\mathcal G)} \vert \varphi^{(\alpha)}(x) \vert \quad \forall \varphi \in C^{\infty}(\mathcal E(\mathcal G)).$
It is then obvious that if $\supp(\varphi) \subset \mathcal E(\mathcal G) \backslash K$ then $\nu(\varphi) = 0,$ i.e., the support of $\nu$ is contained in $K$ and the order of $\nu$ does not exceed $k.$

For $m \in \mathbb Z$, we define the Sobolev spaces $H^m(\mathcal G)$ as the Hilbert space direct sum $H^m(\mathcal G) := \bigoplus_{\me \in \mathcal E(\mathcal G)} H^m(\me).$
We define the function spaces $H^m_{\operatorname{comp}}(\mathcal G) := \bigcup_{A \subset \mathcal E(\mathcal G): \vert A \vert < \infty} H^m(A)$ equipped with the inductive limit topology and spaces 
\[ H^m_{\operatorname{loc}}(\mathcal G) := \left\{ u \in \mathscr D'(\mathcal E(\mathcal G)); \forall \varphi \in C_c^{\infty}(\mathcal E(\mathcal G)): \varphi u \in H^m(\mathcal G) \right\}.\]
The topology on space $H^m_{\operatorname{loc}}(\mathcal G)$ is generated by seminorms $p_{\varphi}(u):= \Vert \varphi u \Vert_{H^m(\mathcal G)}$ for $\varphi \in C_c^{\infty}(\mathcal E(\mathcal G)).$
We then have the continuous inclusions, if we again consider spaces with their inductive limit topologies, $\bigcap_{m \in \mathbb Z} H^m_{\operatorname{loc}}(\mathcal G) \hookrightarrow \mathscr E(\mathcal E(\mathcal G))$ and $\mathscr E'(\mathcal E(\mathcal G)) \hookrightarrow \bigcup_{m \in \mathbb Z} H^m_{\operatorname{comp}}(\mathcal G).$

\section{Definition of the Kirchhoff-Laplacian}
\label{sec:DefNL}

In this preliminary section, we introduce the central object of this article, the Kirchhoff-Laplacian $\Delta$ on various function spaces that we will analyze more thoroughly in this article.

We consider a collection $(c(\me))_{\me \in \mathcal E(\mathcal G)}$ of conductivities. The conductivities enter in the Kirchhoff boundary conditions\footnote{sometimes also called Neumann conditions},
\begin{equation}
\label{eq:kirchhoff}
 \sum_{ \me \in \mathcal{E}_{\mv}(\mathcal G) } \dfrac{\partial u_\me}{\partial
n_\me}(\mv) =0,
\end{equation}
 that we impose at every $\mv \in \mathcal{V}(\mathcal G)$. In this condition, both the orientation of the graph and the conductivities enter \begin{equation*}
\dfrac{\partial u_\me}{\partial n_\me}(\mv):=
\begin{cases}
 \lim_{h \rightarrow 0_+} c(\me)(u_\me(h)-u_\me(0))/h & \mbox{if} \ \mv=i(\me), \vspace{2mm}\\
-\mbox{lim}_{h \rightarrow 0_+} c(\me)(u_\me(1-h)-u_\me(1))/h & \mbox{if} \ \mv=t(\me).
\end{cases} 
\end{equation*}
We define the domain of the operator $\Delta$ on either $X_{\operatorname{c}}=C_0(\mathcal G)$, $X_{\operatorname{c}}=BUC(\mathcal G)$ or $X_{\operatorname{c}}=L^p(\mathcal G)$ with $p \in [1,\infty)$ by
\begin{equation}
\label{eq:domain}
\begin{split}
D(\Delta_X)&:= \Big \{ u \in X_{\operatorname{c}} : Hu \in X_{\operatorname{c}}, u \in C(\mathcal G), u \text{ satisfies }\eqref{eq:kirchhoff} \Big \}\text{ and }(\Delta_Xu)_{\me}:=u''_{\me};
\end{split}
\end{equation}
that is, at each vertex elements $u$ of $D(\Delta_X)$ satisfy so-called \textit{natural} (or \textit{standard}) conditions: continuity along the vertices for $u$ and Kirchhoff condition on the normal derivative.
Furthermore, powers $\Delta_X^m$ are then iteratively defined on $D(\Delta_X^{m+1}) :=\{ u \in D(\Delta_X^m); \Delta u \in D(\Delta_X)\}.$ This means that all \textit{odd derivatives are required to satisfy a Kirchhoff condition} and \textit{all even ones a continuity condition} at the vertices. 

The operator $\Delta_{L^2}$ is a self-adjoint operator on $L^2$ \cite{Cat97}. Moreover, for general spaces $X$ as above, the operator $\Delta_X$ is closed with the choice of domain \eqref{eq:domain} as the following Lemma shows: 
\begin{lemm}
\label{lemm:closed}
The domain $D(\Delta_X)$ is dense in $X_{\operatorname{c}}$ and the operator $-(-\Delta_X)^m$ is closed on $D((-\Delta_X)^m)$ with respect to the graph norm $\Vert f \Vert_{\Delta_p}:= \Vert (\Vert f \Vert_{L^{p}}, \Vert (-\Delta)^m f \Vert_{L^{p}} ) \Vert_{\ell^p}$ with the respective $p \in [1,\infty)$ for spaces $X_{\operatorname{c}}=L^p(\mathcal G)$ and $p =\infty$ for $X_c=C_0(\mathcal G)$ and $X_{\operatorname{c}}=BUC(\mathcal G).$
\end{lemm}
\begin{proof}
We start by defining Banach spaces $X_{\operatorname{c}}^{(2)}$
\begin{equation}
\begin{split}
\label{eq:spaces}
C_0^2(\mathcal G)&:=\left\{ (f_\me)_{\me \in \mathcal E(\mathcal G)}; f_{\me} \in C^2[0,\vert \me \vert]; (\Vert f_{\me} \Vert_{C^2[0,\vert \me \vert]})_{\me} \in c_0(\mathcal E(\mathcal G))\right\}, \\
W^{2,p}(\mathcal G)&:=\left\{ (f_\me)_{\me \in \mathcal E(\mathcal G)}; f_{\me} \in W^{2,p}[0,\vert \me \vert]; (\Vert f_{\me} \Vert_{W^{2,p}[0,\vert \me \vert]})_{\me} \in \ell^p(\mathcal E(\mathcal G))\right\} \\
BUC^2(\mathcal G)&:=\big\{ (f_\me)_{\me \in \mathcal E(\mathcal G)}; f_{\me} \in C^{2}[0,\vert \me \vert]; (\Vert f_{\me} \Vert_{C^{2}[0,\vert \me \vert]})_{\me} \in \ell^{\infty}(\mathcal E(\mathcal G)), \text{ and }\\
&\qquad\quad \forall \varepsilon>0 \ \exists \delta>0 \ \forall x,y \in (0,\vert \me \vert) \ \forall \me \in \mathcal E(\mathcal G): \\
& \quad\qquad \vert x-y \vert \le \delta \Rightarrow \vert f_\me(x)-f_\me(y) \vert ,\vert f''_\me(x)-f''_\me(y) \vert \le \varepsilon \big\}. \\
\end{split}
\end{equation}

That the above spaces are Banach spaces with respect to the graph norm 
\[ \Vert f \Vert_{\Delta_p}:= \Vert (\Vert f \Vert_{L^{p}}, \Vert \Delta f \Vert_{L^{p}} ) \Vert_{\ell^p} \] 
follows easily from the Cauchy--Schwarz and Gagliardo--Nirenberg interpolation inequality 
\begin{equation*}
\begin{split}
\Vert f' \Vert^p_{L^{p}} 
&= \sum_{e \in \mathcal E(\mathcal G)} c(\me) \Vert f' \Vert^p_{L^p( \me)} \lesssim \sum_{e \in \mathcal E(\mathcal G)} c(\me) \left(\Vert f \Vert_{L^p( \me)}^{\frac{p}{2}} \Vert \Delta f \Vert_{L^p( \me)}^{\frac{p}{2}}+ \Vert f \Vert^p_{L^p( \me)}\right) \\
& \lesssim \sqrt{ \sum_{e \in \mathcal E(\mathcal G)} c(\me) \Vert f \Vert^p_{L^{p}(\me)} } \sqrt{ \sum_{e \in \mathcal E(\mathcal G)} c(\me) \Vert \Delta f \Vert^p_{L^{p}(\me)} } +\sum_{e \in \mathcal E(\mathcal G)} c(\me) \Vert f \Vert^p_{L^{p}(\me)} \\
&=\Vert f \Vert^{p/2}_{L^{p}} \Vert \Delta f \Vert^{p/2}_{L^{p}} +\Vert f \Vert^{p}_{L^{p}} . 
\end{split}
\end{equation*}
We then introduce functionals $F_{\mv}$ for vertices $\mv \in \mathcal V(\mathcal G)$ and $G_{\me,\me'}$ for adjacent edges $\me,\me' \in \mathcal E_{\mv}(\mathcal G)$ by
\begin{equation}
\begin{split}
\label{eq:functionals}
F_{\mv}: X_{\operatorname{c}}^{(2)} \ni f&\mapsto \sum_{\me \in \mathcal E_\mv(\mathcal G)} \frac{\partial f_\me}{\partial n_\me}(\mv) \text{ and } G_{\me,\me'}: X_{\operatorname{c}}^{(2)} \ni f\mapsto f_\me(\mv)-f_{\me'}(\mv).
\end{split}
\end{equation}
Thus, let $X_{\operatorname{c}}^{(2)}$ be any of the spaces in \eqref{eq:spaces}, then we can define
\begin{equation*}
\begin{split}
D(\Delta_{X})&:=X_{\operatorname{c}}^{(2)} \cap \bigcap_{\mv \in \mathcal V(\mathcal G)} \operatorname{ker}(F_\mv) \cap \bigcap_{\me,\me' \in \mathcal E(\mathcal G); \me\cap \me' =\left\{\mv \right\}} \operatorname{ker}(G_{\me,\me'}).
\end{split}
\end{equation*}
The operators $\Delta_{X}$ are closed since functionals \eqref{eq:functionals} are continuous as maps from $X_c^{(2)}$ to $\RR$. The domain $D(\Delta_X)$ is dense in $X_c$ as it contains $C_c^{\infty}(\mathcal G)$. In the case of $X_c=BUC(\mathcal G) \ni f$, it suffices to subdivide for every $\varepsilon_n=1/n$ each edge into intervals $[x_i,y_i]$ of sizes between $\delta_n/2$ and $\delta_n$ such that $\vert x-y \vert \le \delta_n $ implies $\vert f(x)-f(y) \vert \le \varepsilon_n.$ On each of these intervals, we can then choose a function $g$ interpolating between $f(x_i)$ and $f(y_i)$ such that $g'(x_i)=g''(x_i)=g'(y_i)=g''(y_i)=0$ and $g \in D(\Delta_{BUC(\mathcal G)}).$
\end{proof}

\section{Convolutions on metric graphs and explicit construction of heat semigroup}
\label{sec:expconst}

In this section, we develop a general framework to naturally extend to metric graphs a class of convolution semigroups on the real line $\RR$, with sufficiently rapidly decaying integral kernels, to semigroups on $\mathcal G$, building on ideas from ~\cite{Rot83,Cat98}. To this aim, we start by defining transfer coefficients:

\begin{defi}[Transfer coefficients]
\label{def:TC}
For two edges $\me,\me' \in \mathcal E(\widetilde{\mathcal G})$ we define the transfer coefficient
\begin{equation*}
 \begin{split}
 \mathbb T_{\me,\me'}:&=\left(2 \tfrac{c(\me)}{c(t(\me))}-\delta_{\me,-\me'} \right)\delta_{t(\me),i(\me')}=\begin{cases}
  2 \frac{c(\me)}{c(t(\me))}, &\text{if }t(\me)=i(\me'),\ \me' \neq -\me,\\
  2 \frac{c(\me)}{c(t(\me))}-1, &\text{if }t(\me)=i(\me'),\ \me' =-\me, \text{ and}\\
  0, &\text{otherwise.} 
  \end{cases}
 \end{split}
\end{equation*}
For each path $P=(\me_0,..,\me_m)$ we denote by $\mathbb T_P$ the product of transfer coefficients along that path $\mathbb T_P:=\prod_{j=0}^{m-1} \mathbb T_{\me_j,\me_{j+1}}.$

\end{defi}

We now turn to the construction of the integral kernel on the graph and consider intervals $I_m:= [m \ell_{\downarrow} ,\infty)$ for integers $m \in \mathbb N_0$. Then we can define a weighted $L^1$-norm 
\begin{equation}
 \Vert f \Vert_{\mathcal L^1} :=\sum_{m \in \mathbb N_0} 3^{m} \Vert f \Vert_{L^{1}(I_m)} 
 \end{equation}
on the Banach space $\mathcal L^1$ of even functions for which the above sum is finite 
\begin{equation*}\label{A}\mathcal L^1:=\left\{ f \in L^{1}(\mathbb R); f(-\bullet)=f(\bullet) \text{ and } \Vert f \Vert_{\mathcal L^1}<\infty \right\}.\end{equation*}

\medskip

We consider exponential scaling in the definition of $\mathcal L^1$ due to the --in general-- exponential scaling in the number of paths a graph admits. This, we illustrate with the following simple example: 
\begin{ex}[Paths on $\ZZ$]
We consider a graph $\mathcal G$ with edges $\{(n,n+1);n \in \ZZ\}$ and vertices $\mathcal V(\mathcal G)=\ZZ$.
\begin{itemize}
\item If the conductivities are constant $c_{(n,n+1)}=c_{(n+1,n+2)}\text{ for all }n \in \ZZ,$ 
the transfer coefficients are $\mathbb T_{(n,n+1),(n+1,n+2)}=1$ and $\mathbb T_{(n,n+1),(n,n+1)}=0$ since there is full transmission and no back-scattering at the vertices. This implies that there is only a single path from the edge $\me:= (n,n+1)$ to $\me':=(n+k,n+(k+1)).$ 
\item If the conductivities are not all the same then there are in general paths of length $2m+k$ from $\me$ to $\me'$ for any integer $m \ge 0.$ Thus, the total number of such paths of length $2m+k$ is $\binom{2m+k}{m}$ and scales exponentially in $m$ (Stirling's formula).
\end{itemize}
\end{ex}

\begin{lemm}
The Banach space $\mathcal L^1$ is a Banach algebra with respect to the standard convolution on $\RR.$
\end{lemm}
\begin{proof}
We have from an immediate computation with $\phi_n:=\Vert f \Vert_{L^{1}[n \ell_{\downarrow},(n+1) \ell_{\downarrow}]}$ and $\gamma_n:=\Vert g \Vert_{L^{1}[n \ell_{\downarrow},(n+1) \ell_{\downarrow}]}$
\begin{equation*}
\begin{split}
\int_{m \ell_{\downarrow}}^{(m+1) \ell_{\downarrow}} \vert (f*g)(t) \vert \ dt \le &\int_{m \ell_{\downarrow}}^{(m+1) \ell_{\downarrow}} \sum_{n \in \ZZ} \int_{n \ell_{\downarrow}}^{(n+1) \ell_{\downarrow}} \vert f(s) \vert \vert g(t-s) \vert \ ds \ dt \\
& \le \sum_{n \in \ZZ} \phi_n\left(\gamma_{m-(n+1)} + \gamma_{m-n} \right).
\end{split}
\end{equation*}

This implies that 
\begin{equation*}
\begin{split}
\Vert f*g \Vert_{\mathcal L^1} 
&\le \sum_{m\in \NN_0,n\in \ZZ} 3^{ m } \phi_n\gamma_{m-(n+1)} + \sum_{m\in \NN_0,n\in \ZZ} 3^{ m } \phi_n\gamma_{m-n} \\
&\le 3 \sum_{m\in \NN_0,n\in \ZZ} 3^{\vert n \vert} \phi_n 3^{\vert m-(n+1) \vert}\gamma_{m-(n+1)} + \sum_{m\in \NN_0,n\in \ZZ} 3^{\vert n \vert} \phi_n 3^{\vert m-n \vert}\gamma_{m-n} \\
&\le 6 \Vert f \Vert_{\mathcal L^1} \Vert g \Vert_{\mathcal L^1} + 2\Vert f \Vert_{\mathcal L^1} \Vert g \Vert_{\mathcal L^1} = 2(3+1) \Vert f \Vert_{\mathcal L^1} \Vert g \Vert_{\mathcal L^1}.
\end{split}
\end{equation*}
This concludes the proof.
\end{proof}

In the following, we parametrize points of the metric graph $\mathcal G$ by tuples $(\me,\xi)$ where $\me \in \mathcal E({ \mathcal G})$ and $\xi \in [0,\vert \me \vert]$, with endpoints suitably identified to mirror the graph's connectivity, see~\cite{Mug19}. For $P \in \mathcal P_{\me,\me'}(m)$ with $P = (\me_0,..,\me_m)$ and both $\me_0=\me$ and $\me_m=\me'$, we then have that $c(\me_0)^{-1} \mathbb T_P = c(\me_m)^{-1} \mathbb T_{-P}$ where $-P=(-\me_m,-\me_{m-1},...,-\me_0)$ is the inverted path. 

We can now extend integral kernels $f \in \mathcal L^1$ -of convolution type- on $\RR$ to integral kernels $K_f$ on products of graphs $\mathcal G:$ 

\begin{defi}[Integral kernel on $\mathcal G$]
For functions $f \in \mathcal L^1$ we then define an integral kernel $K_f$ on $ \mathcal G \times \mathcal G $ by the sum $K_f :=K^{(1)}_f+ K^{(2)}_f$, where $K^{(1)}_f:=K^{(1a)}_f+K_f^{(1b)}$ and $K^{(2)}_f:=K^{(2a)}_f+K_f^{(2b)}$ are defined by
\begin{equation}
\begin{split}
\label{eq:kernel}
K^{(1a)}_f((\me,\xi),(\me',\xi'))&:=c(\me )^{-1} \Bigg( \delta_{\me, \me'} f(\xi'-\xi)+\sum_{m \in \mathbb N}\sum_{P \in \mathcal P_{\me,\me'}(m)} \mathbb T_P f(\xi'+\vert P \vert-\xi) \Bigg) \\
K^{(1b)}_f((\me,\xi),(\me',\xi'))&:=c(\me )^{-1} \sum_{m \in \mathbb N} \sum_{P \in \mathcal P_{\me,-\me'}(m)} \mathbb T_P f(\vert \me' \vert-\xi'+\vert P \vert-\xi) \\
K^{(2a)}_f((\me,\xi),(\me',\xi'))&:=c(\me )^{-1}\sum_{m \in \mathbb N} \sum_{P \in \mathcal P_{-\me,\me'}(m)} \mathbb T_P f(\xi'+\vert P \vert-(\vert \me \vert-\xi)) \\
K^{(2b)}_f((\me,\xi),(\me',\xi'))&:=c(\me )^{-1}\sum_{m \in \mathbb N} \sum_{P \in \mathcal P_{-\me,-\me'}(m)} \mathbb T_P f(\vert \me' \vert-\xi'+\vert P \vert-(\vert \me \vert-\xi))
\end{split}
\end{equation}
where $\vert P \vert$ is the length of path $P.$
\end{defi}
\begin{rem}
The four sums above appear by considering paths between points $(\me,\xi),(\me',\xi')$, where each point on the graph can be represented using either possible orientation of edges $\pm \me, \pm \me'$ on the graph .
\end{rem}
Let us first note that the total kernel $K_f$ we defined is symmetric. This follows directly, since we can revert the orientation of all paths $P$ appearing in the above sums. This implies that local (edge-wise) coordinates change from $\xi$ to $\vert \me \vert- \xi$ and analogously for $\me'$. If we then use that $c(\me)^{-1} \mathbb T_P = c(\me')^{-1} \mathbb T_{-P}$ it follows immediately from \eqref{eq:kernel} that $K_f$ is symmetric when all these operations are jointly executed. 

The kernel $K_f$ on the graph is such that it coincides with the kernel on individual edges and takes scattering at the vertices into account when linking different edges together.

This allows us to define a convolution for $f \in \mathcal L^1$, suitable (to be specified later) functions $u$ on $ \mathcal G $, and $x \in {\mathcal G}$ as
\begin{equation}\label{eq:deficonv}
 (f*_{\mathcal G} u)(x):=\int_{ {\mathcal G}} K_f(x,y)u(y) \ dy 
 \end{equation}
satisfying the associative property, see Lemma \ref{lemm:convlemm}, for $f,g\in \mathcal L^1$
\[ (f*_{\mathcal G}(g*_{\mathcal G}u))(x) = ((f*g)*_{\mathcal G}u)(x). \]

\begin{defi}[Convolution semigroup on $\mathcal G$]
To extend the definition of a convolution semigroup on $\RR$ with kernel $k_t \in \mathcal L^1$, for $t>0$, to the graph ${\mathcal G}$, we make the \textit{Ansatz}
\begin{equation}
\label{eq:semgroup}
(T_tf)(x) \stackrel{!}{=} \int_{{\mathcal G}} K_{k_t}(x,y) f(y) \ dy,\qquad x\in {\mathcal G}.
\end{equation}
If the choice of kernel $k_t$ is clear, we may just write $K_t.$
\end{defi}

\subsection{Convergence properties of convolution semigroups}
To analyze the semigroup \eqref{eq:semgroup}, we now introduce operators 
\begin{equation*}
\begin{split}
 (E^{(\nu)}_{\me}f)(\me',\xi)&:=K_f^{(\nu)}((\me,0),(\me',\xi)),\quad \me\in \mathcal E(\mathcal G),\ (\me',\xi)\in \mathcal G, \nu \in \{1a,1b,1,2a,2b,2\}.
 \end{split}
 \end{equation*}
 
 To link the function $E_{\me}f(\me',\xi')=K_f((\me,0),(\me',\xi'))$ to the kernel $K_f((\me,\xi),(\me',\xi'))$ in the subsequent Lemma, we define the shift operator $(\tau_{\xi} f)(x):=f(x-\xi)$. Then the shift $\RR \ni s \mapsto \tau_s f=f(\bullet-s) \in \mathcal L^1$ is continuous with $\Vert \tau_s \Vert_{\mathcal L(\mathcal L^1)} \le \kappa$ for all $s \in [-\ell_{\uparrow},\ell_{\uparrow}]$ and some $\kappa>0$.

 \begin{lemm}[$L^1$-$L^{\infty}$ estimate]
 \label{lemmi}
 For $u \in L^{\infty}(\mathcal G)$ and $f \in \mathcal L^1$ we have
 \[ \sup_{x \in \mathcal G}\int_{{\mathcal G}} \left\lvert K_{f}(x,y) u(y)\right\rvert \ dy \ \lesssim \kappa \Vert f \Vert_{\mathcal L^1} \Vert u \Vert_{L^{\infty}}.\]
In particular, for all $\me \in \mathcal E({\mathcal G})$ the operator $E_{\me}$ defines a continuous operator $E_\me \in \mathcal L(\mathcal L^1, L^1(\mathcal G))$ with bound $\Vert E_\me f \Vert_{L^1} \lesssim \kappa \Vert f \Vert_{\mathcal L^1}$ independent of $\me$. 
 \end{lemm}
 \begin{proof}
The statement follows directly by applying \eqref{eq:expthree2} to \eqref{eq:kernel}, substituting the argument of $f$ in \eqref{eq:kernel} as the new integration variable, such that
\begin{equation*}
\begin{split}
\int_{{\mathcal G}} \left\lvert K_{f}(x,y) u(y)\right\rvert \ dy &\lesssim \frac{\kappa}{c(\me)}\sum_{\substack{m \in \mathbb N_0 \\ \me' \in \mathcal E(\mathcal G) \\ P \in \mathcal P_{\me,\me'}(m)}} c(\me') \vert \mathbb T_{P} \vert \Vert f \Vert_{L^{1}(I_{m})} \Vert u \Vert_{L^{\infty}} \\
&\lesssim \kappa \sum_{m \in \mathbb N_0} 3^m \Vert f \Vert_{L^{1}(I_m)} \Vert u \Vert_{L^{\infty}}= \kappa \Vert f \Vert_{\mathcal L^1}\Vert u \Vert_{L^{\infty}}.
\end{split}
\end{equation*}
This concludes the proof.
 \end{proof}

\begin{lemm}
\label{techlemma}
Let $f \in \mathcal L^1.$ The linear convolution maps, on spaces of continuous functions, $f*_{\mathcal G}: L^{\infty}(\mathcal G) \rightarrow BUC(\mathcal G)$ and $f*_{\mathcal G}:C_0(\mathcal G) \rightarrow C_0(\mathcal G)$
are bounded. In addition, the convolution $f*_{\mathcal G}: L^{p}(\mathcal G) \rightarrow L^p(\mathcal G)$ is a bounded map satisfying 
\[\Vert f*_{\mathcal G} u \Vert_{L^p(\mathcal G)} \le (\kappa \Vert f \Vert_{\mathcal L^1}) \Vert u \Vert_{L^p(\mathcal G)}\text{ for all }p\in [1,\infty].\]
\end{lemm}
\begin{proof}
The continuity of the shift is immediate if we assume that $f \in C_c(\RR)$. We therefore start by showing that the convolution is uniformly continuous: For $\eta,\eta' \in [0,\vert e \vert]$ with $\vert \eta-\eta' \vert \le \operatorname{min}\{\delta, \ell_{\downarrow}\}$, where $\delta$ is chosen such that $\Vert (\tau_{\eta}-\tau_{\eta'}) f \Vert_{\mathcal L^1}+\Vert (\tau_{-\eta}-\tau_{-\eta'}) f \Vert_{\mathcal L^1}\le \varepsilon$, and Lemma \ref{lemmi}
\begin{equation*}
\begin{split}
&\vert (f*_{\mathcal G}u)(\eta,\me)-(f*_{\mathcal G}u)(\eta',\me) \vert \\
&= \left\lvert \sum_{\me' \in \mathcal E(\widetilde{\mathcal G})} c(\me') \int_0^{\vert \me' \vert} \left(E^{(1)}_\me(\tau_{\eta} f)(\xi,\me')-E^{(1)}_\me(\tau_{\eta'} f)(\xi,\me')\right)u(\xi,\me') \ d\xi \right\rvert \\
&\quad + \left\lvert \sum_{\me' \in \mathcal E(\widetilde{\mathcal G})} c(\me') \int_0^{\vert \me' \vert} \left(E^{(2)}_\me(\tau_{-\eta} f)(\xi,\me')-E^{(2)}_\me(\tau_{-\eta'} f)(\xi,\me')\right)u(\xi,\me') \ d\xi \right\rvert \\
&\le \left\lVert E^{(1)}_\me(\tau_{\eta} f)-E^{(1)}_\me(\tau_{\eta'} f) \right\rVert_1 \left\lVert u \right\rVert_{\infty} + \left\lVert E^{(2)}_\me(\tau_{-\eta} f)-E^{(2)}_\me(\tau_{-\eta'} f) \right\rVert_1\Vert u \Vert_{\infty} \\
& \le\left( \Vert (\tau_{\eta}-\tau_{\eta'}) f \Vert_{\mathcal L^1}+\Vert (\tau_{-\eta}-\tau_{-\eta'}) f \Vert_{\mathcal L^1} \right)\Vert u \Vert_{\infty}\le \varepsilon \Vert u \Vert_{\infty}. 
\end{split}
\end{equation*}
The result then follows since $C_c(\RR)$ is dense in $\mathcal L^1.$
To study the map on $C_0(\mathcal G)$, it suffices to use the density of $C_c(\RR)$ in $\mathcal L^1$ and of $C_c(\mathcal G)$ in $C_0(\mathcal G).$ This immediately implies that $f*_{\mathcal G}$ maps $C_0(\mathcal G)$ into $C_0(\mathcal G)$ by Lemma \ref{lemmi}.

It remains therefore to show $L^p$-boundedness:

To show boundedness on $L^p(\mathcal G)$-spaces that satisfy $p<\infty$, we use that for $p^{-1}+q^{-1}=1$ by Hölder's inequality, Lemma \ref{lemmi}, and the symmetry of the integral kernel 
\begin{equation*}
\begin{split}
&\int_{\mathcal G} \vert K_f((\me',\xi'),(\me,\xi)) \vert \vert u(\me,\xi) \vert \ d\lambda_c(\me,\xi) \\
&=\int_{\mathcal G} \vert K_f((\me',\xi'),x) \vert^{1/q} (\vert K_f((\me',\xi'),x) \vert \vert u(x) \vert^p)^{1/p} \ dx \\
& \le \Vert K_f((\me',\xi'),\bullet) \Vert_1^{1/q} \left(\int_{\mathcal G} \vert K_f((\me',\xi'),x) \vert \vert u(x) \vert^p \ dx\right)^{1/p}\\
& \lesssim (\kappa \Vert f \Vert_{\mathcal L^1})^{1/q} \left(\int_{\mathcal G} \vert K_f((\me,\xi),x) \vert \vert u(x) \vert^p \ dx \right)^{1/p}.
\end{split}
\end{equation*}
Hence, we get, using again Lemma \ref{lemmi}
\begin{equation*}
\begin{split}
\Vert f*_{\mathcal G}u \Vert_{p}^p &\lesssim (\kappa \Vert f \Vert_{\mathcal L^1})^{p/q} \int_{\mathcal G} \int_{\mathcal G} \vert K_f(x,y) \vert \vert u(x) \vert^p \ dx \ dy \\
& \lesssim (\kappa \Vert f \Vert_{\mathcal L^1})^{p/q} \int_{\mathcal G} \Vert E_{\me}(\tau_{-\xi} f) \Vert_1 \vert u(e,\xi) \vert^p d\lambda_c(e,\xi) \\
&\lesssim (\kappa \Vert f \Vert_{\mathcal L^1})^{p/q} \kappa \Vert f \Vert_{\mathcal L^1}\Vert u \Vert_p^p\le (\kappa \Vert f \Vert_{\mathcal L^1})^{p} \Vert u \Vert_p^p.
\end{split}
\end{equation*}
This concludes the proof.
\end{proof}

\begin{lemm}
\label{lemm:unitint}
Let $t\mapsto k_t \in \mathcal L^1$ be a differentiable function with $\frac{\partial k_t}{\partial t} \in \mathcal L^1$, such that for all time $t>0$ we have $\int_{\RR} k_t=\operatorname{const}$ and $\int_{\mathcal G} K_{\frac{\partial k_t}{\partial t}}(x,y) \ dy =0$. Assume that for every $\varepsilon>0$ we have
$\lim_{t \downarrow 0} \sum_{m \in \mathbb N_0} 3^{m} \Vert k_t \Vert_{L^{1}(I_m \backslash [0,\varepsilon))} =0$ and there exists a function $g_{\varepsilon} \in \mathcal L^1$ with $\vert k_t(x) \vert \le g_{\varepsilon}(x)$ for all $x \in \RR \backslash(-\varepsilon,\varepsilon)$ and sufficiently small times. Then, it follows that for every $x \in \mathcal G$ on the interior of an edge 
\[\int_{\mathcal G} K_{k_t}(x,y) \ dy= \operatorname{const}.\]
\end{lemm}
\begin{proof}
Let $x=(\me, \xi)$ be a point in the interior of $\me$ then by assumption there is $\alpha_x=\tfrac{1}{2}\operatorname{min}\{d(x,i(\me)),d(x,t(\me))\}>0$ such that $G_x:=\mathcal G \backslash \{x-\alpha_x,x+\alpha_x\}$
\[ \lim_{t \downarrow 0} \int_{G_x} K_{k_t}(x,y) \ dy = 0 \text{ and }\lim_{t \downarrow 0} \int_{\{x-\alpha_x,x+\alpha_x\}} K_{k_t}(x,y) \ dy =\lim_{t \downarrow 0} \int_{\mathbb R} k_t(s) \ ds=\operatorname{const}.\]
We then observe that by the dominated convergence theorem 
\[ \frac{\partial}{\partial t} \int_{\mathcal G} K_{k_t}(x,y) \ dy = \int_{\mathcal G}K_{\frac{\partial k_t}{\partial t}}(x,y) \ dy =0.\]
Thus, it follows that $\int_{\mathcal G} K_t(x,y) \ dy= \int_{\mathbb R} k_t(s) \ ds= \operatorname{const}.$
\end{proof}

We proceed with a result on approximate identities on metric graphs $\mathcal G$ that will allow us to show strong continuity of $C_0$-semigroups.
\begin{lemm}
Consider a kernel $k_t$ as in Lemma \ref{lemm:unitint}. Then, we obtain the following strong limits for the semigroup defined by $T_tu:= k_t*_{\mathcal G} u$
\begin{equation*}
\begin{split}
&\lim_{t \downarrow 0} \Vert k_t*_{\mathcal G} u - u \Vert_{\infty}=0 \qquad\hbox{for all }u \in C_0(\mathcal G), BUC(\mathcal G) \text{ and } \\
&\lim_{t \downarrow 0} \Vert k_t*_{\mathcal G} u - u \Vert_{p}=0,\qquad\hbox{for all }u \in L^p(\mathcal G),\ p\in [1,\infty).
\end{split}
\end{equation*}
\end{lemm}
\begin{proof}
The dominated convergence theorem implies the existence of the following limit
\begin{equation}
\label{eq:uniflimit}
 \lim_{t \downarrow 0}\sup_{\xi' \in [0,\ell_{\downarrow}]} \sum_{m = 2}^{\infty} 3^{m} \int_0^{\ell_{\downarrow}} \vert k_t(\xi'+m\ell_{\downarrow}-\xi) \vert \ d\xi=0.
 \end{equation}

To simplify the notation, we write the kernel $K_{k_t}$ for $x \in \me $ and $y \in \me'$ as
\begin{equation*}
\begin{split}
K_{k_t}(x,y):=&c(\me)^{-1}\Bigg( k_t(x-y)\delta_{\me, \me'}+\sum_{\substack{ m \in \mathbb N \\ P \in \mathcal P_{\me,\me'}(m)}}\mathbb T_P\ k_t\left(\vert x-t(\pm \me) \vert+ \vert P \vert + \vert y-i(\pm \me')\vert\right) \Bigg)
\end{split}
\end{equation*}
where the summation is also over all four possible orientation of the initial and terminal edges. 

We then start by showing the uniform convergence for $u \in BUC(\mathcal G)$, which also implies the result for $C_0(\mathcal G)$. Using the uniform limit \eqref{eq:uniflimit}, we find that
\begin{equation*}
\begin{split}
\label{eq:difference}
&\vert (k_t*_{\mathcal G} u)(\xi,\me)- u(\xi,\me) \vert \le \int_0^{\vert \me \vert} \vert k_t(\xi-\xi') \vert \vert u(\xi',\me)-u(\xi,\me) \vert \ d\xi'+ \sum_{\substack{m \in \mathbb N \\ \me' \in \mathcal E(\mathcal G) \\ P \in \mathcal P_{\me,\me'}(m)} } \frac{c(\me')\vert \mathbb T_P \vert}{c(\me)} \times \\
& \int_0^{\vert \me' \vert} \vert k_t\left(\vert (\xi,\me)-t(\pm \me) \vert+ \vert P \vert -\vert \me\vert+ \vert (\xi',\me')-i(\pm \me')\vert \right) \vert \left\vert u(\xi',\me')-u(\xi,\me) \right\vert \ d\xi' \\
&\le \int_0^{\vert \me \vert} \vert k_t(\xi-\xi') \vert \vert u(\xi',\me)-u(\xi,\me) \vert d\xi' + 2\Vert u \Vert_{\infty} \varepsilon + \sum_{\substack{ \me' \in \mathcal E(\mathcal G) \\ P \in \mathcal P_{\me,\me'}(1)} } \frac{c(\me')\vert \mathbb T_P \vert}{c(\me)} \times \\
& \int_0^{\vert \me' \vert} \vert k_t\left(\vert (\xi,\me)-t(\pm \me) \vert+ \vert P \vert -\vert \me\vert+ \vert (\xi',\me')-i(\pm \me')\vert \right) \vert \left\vert u(\xi',\me')-u(\xi,\me) \right\vert \ d\xi' 
\end{split}
\end{equation*}
for all $t \in (0,t_0)$ sufficiently small and $(\xi,\me') \in \mathcal G.$

We then split the integral in the second-to-last line above into an integral over a set $I_{\delta}:=(\xi-\delta,\xi +\delta) \cap (0,\vert \me \vert)$ and its relative complement $I_{\delta}^C:=(\xi-\delta,\xi +\delta)^C \cap (0,\vert \me \vert)$. 
We see that for $\xi \in I_\delta^C,$ the integral in the second-to-last line, and the integral in the last line are small in virtue of the limit $\lim_{t \downarrow 0} \int_{\RR \backslash (-\delta,\delta)} \vert k_t(y)\vert \ dy =0.$
If $\xi \in I_\delta$, then $\vert u(\xi',\me)-u(\xi,\me) \vert$ is small by the uniform continuity of $u.$ Hence, we have shown that 
\begin{equation}
\label{eq:unifconv}
 \limsup_{t \downarrow 0} \Vert k_t*_{\mathcal G} u-u \Vert_{\infty} =0.
 \end{equation}

To show general $L^p$ convergence, we start by showing $L^1$ convergence, first. It suffices by Lemma \ref{techlemma} to show by density the convergence for $u \in C_c(\mathcal G)$ and thus we now fix $R>0$ large enough, such that $u(x):=0$ for $d(x,o) \ge R$ where $o$ is a fixed reference point of $\mathcal G.$ Then we first consider $x=(\xi,\me)$ with $d(x,o) > 2R.$ It follows that the expression for the convolution reduces to
\[ (k_t*_{\mathcal G} u)(x)=\sum_{\substack{m \ge 2 \\ \me' \in \mathcal E(B_{R+\ell_{\uparrow}}(o) ) \\ P \in \mathcal P_{\me,\me'}(m)}} \frac{c(\me')}{c(\me)} \mathbb T_P \int_0^{\vert \me' \vert} k_t(\xi'+\vert P \vert-\xi) u(\xi',\me') \ d\xi' \]
where $B_{R+\ell_{\uparrow}}(o)$ is the ball with respect to the graph metric.

This implies that for any small $\varepsilon>0$ we have for $t$ small enough that $\vert (k_t*_{\mathcal G} u)(x)- u(x) \vert = \vert(k_t*_{\mathcal G}u)(x) \vert \le \varepsilon $ by \eqref{eq:uniflimit} for $x$ satisfying $d(x,o) \ge 2R.$

On the other hand, for $x$ satisfying $d(x,o) \le 2R$ the convergence is uniform by \eqref{eq:unifconv} and thus also in $L^1.$ 

For functions $g \in L^1 \cap L^{\infty}$ we have $\Vert g \Vert_{p} \le \Vert g \Vert^{1/p}_{1} \Vert g \Vert_{\infty}^{(p-1)/p}.$
Applying this to $g=u-k_t*_{\mathcal G} u$, and using the convergence in both $L^1$-norm and $L^{\infty}$-norm, we conclude convergence in $L^p$-norm for any intermediate exponent $p \in (1,\infty),$ whence the claim. Thus, let $p < \infty$, $u \in L^p$ and $\varepsilon>0$ small. Choose $v \in C_c$ such that $\Vert u-v\Vert_{p}< \varepsilon$. Then by \eqref{eq:uniflimit} and by Lemma \ref{techlemma}
\begin{equation*}
\begin{split}
 \limsup_{t \downarrow 0} \Vert k_t*_{\mathcal G} u-u \Vert_p 
 &\le \limsup_{t \downarrow 0} \left[ \Vert k_t*_{\mathcal G} v-v\Vert_p + \Vert k_t*_{\mathcal G}(u-v) \Vert_p \right]+ \Vert u-v \Vert_p \\
 &\le (C+1) \Vert u-v \Vert_p \le (C+1)\varepsilon.
\end{split}
\end{equation*}
This concludes the proof.
\end{proof}
\subsection{Boundary conditions for convolutions on graphs}
\label{subsec:bc}
We now show that the convolution, introduced in \eqref{eq:deficonv}, is compatible with the boundary conditions introduced in Section \ref{sec:DefNL}, i.e., the convolution is continuous at the vertices and its derivative satisfies the natural condition \eqref{eq:kirchhoff} at the vertices. 
\begin{lemm}
\label{eq:lemmbc}
Let $f \in \mathcal L^1$ then it follows that for any $u \in \operatorname{BUC}(\mathcal G)$
\[ f*_{\mathcal G} u \text{ is continuous at the vertices }.\]

Let $f \in \mathcal L^1$ be differentiable such that $\vert f' \vert \in \mathcal L^1$ then it follows that for any $u \in \operatorname{BUC}(\mathcal G)$
\[ f'*_{\mathcal G} u \text{ satisfies a Kirchhoff condition at the vertices }.\]
\end{lemm}
\begin{proof}
To start with the continuity conditions, we fix an edge $\me$, then the kernel $K^{(1)}$, defined in \eqref{eq:kernel}, yields using Def. \ref{def:TC}
\begin{equation}
\begin{split}
\label{eq:k1}
&K^{(1)}_f((\me,\xi),(\me',0)):=c(\me )^{-1} \Bigg( \delta_{\me, \me'} f(\xi)+\sum_{m \in \mathbb N}\Bigg( \sum_{P \in \mathcal P_{\me,\me'}(m)} \mathbb T_P f(\vert P \vert-\xi) \\
&\qquad\qquad + \sum_{P \in \mathcal P_{\me,-\me'}(m)} \mathbb T_P f(\vert \me' \vert +\vert P \vert-\xi) \Bigg) \Bigg) \\
&=c(\me )^{-1} \Bigg( \delta_{\me, \me'} f(\xi)+\mathbb T_{\me,\me'} f(\vert \me \vert-\xi) +\sum_{m \in \mathbb N}\Bigg( \sum_{\substack{\me''\in t^{-1}(i(\me')) \\ P \in \mathcal P_{\me,\me''}(m)}} \mathbb T_P \left(\frac{2c(\me'')}{c(i(\me'))} - \delta_{\me'',-\me'}\right)\\ 
&\qquad f(\vert \me'' \vert+ \vert P \vert-\xi)+ \sum_{P \in \mathcal P_{\me,-\me'}(m)} \mathbb T_P f(\vert \me' \vert +\vert P \vert-\xi) \Bigg) \Bigg) \\
&=c(\me )^{-1} \Bigg( \delta_{\me, \me'} f(\xi)+ \mathbb T_{\me,\me'} f(\vert \me \vert-\xi) + \sum_{\substack{m \in \mathbb N \\ \me''\in t^{-1}(i(\me')) \\ P \in \mathcal P_{\me,\me''}(m)}} \mathbb T_P \frac{2c(\me'')}{c(i(\me'))} f(\vert \me'' \vert+ \vert P \vert-\xi) \Bigg).
\end{split}
\end{equation}
By an analogous computation, we find for the remaining part of the kernel in \eqref{eq:kernel}
\begin{equation}
\begin{split}
\label{eq:k2}
&K^{(2)}_f((\me,\xi),(\me',0)):=c(\me )^{-1} \Bigg( \mathbb T_{-\me,\me'}f(\xi)+ \sum_{\substack{m \in \mathbb N \\ P \in \mathcal P_{-\me,-\me'}(m) }} \mathbb T_P f(\vert \me' \vert+\vert P \vert-(\vert \me \vert-\xi)) \\
&\quad +\sum_{\substack{m \in \mathbb N \\ \me''\in t^{-1}(i(\me')) \\ P \in \mathcal P_{-\me,\me''}(m)}} \mathbb T_P \left(\frac{2c(\me'')}{c(i(\me'))} - \delta_{\me'',-\me'}\right) f( \vert \me'' \vert + \vert P \vert-(\vert \me \vert-\xi))\Bigg) \\
\quad &=c(\me )^{-1} \Bigg( \sum_{\substack{\me''\in t^{-1}(i(\me')) \\ P \in \mathcal P_{-\me,\me''}(m)}} \mathbb T_P \frac{2c(\me'')}{c(i(\me'))} f(\vert \me'' \vert + \vert P \vert-(\vert \me \vert-\xi)) + \mathbb T_{-\me, \me'}f(\xi) \Bigg).
\end{split}
\end{equation}
To see that this representation implies continuity at vertices, we introduce coefficients, that only depend on the terminal vertex $\mv \in \mathcal V$ but not(!) on the adjacent edge
\[\left(\alpha_0(\me,i(\me))f \right)(\xi) = \frac{2c(\me)}{c(i(\me))} f(\xi) \text{ and }\left(\alpha_0(\me,t(\me))f\right)(\xi) = \frac{2c(\me)}{c(t(\me))} f(\vert \me \vert-\xi) \] 
and otherwise for $m \in \mathbb N$
\begin{equation*}
\begin{split}
 \left(\alpha_m(\me,\mv)f\right)(\xi)&:=2 \sum_{\substack{\me'' \in t^{-1}\{v \} \\ P \in \mathcal P_{\me,\me''}(m)}} \mathbb T_P \frac{c(\me'')}{c(\mv)} f(\vert P \vert+ \vert \me'' \vert-\xi) \\
 & \quad + 2 \sum_{\substack{\me'' \in t^{-1}\{v \} \\ P \in \mathcal P_{-\me,\me''}(m)}} \mathbb T_P \frac{c(\me'')}{c(\mv)} f(\vert \me'' \vert+\vert P \vert- (\vert \me \vert-\xi)) .
 \end{split}
 \end{equation*}
 From \eqref{eq:k1} and \eqref{eq:k2} and an analogous calculation for terminal vertices, we find that 
 \[ K_f((\me,\xi),\mv) = c(\me)^{-1} \sum_{m \in \mathbb N_0} (\alpha_m(\me,\mv)f)(\xi).\]
 
To derive the natural condition at the vertices, we find for derivatives of $K^{(1)}_f$ and $K^{(2)}_f$ in \eqref{eq:kernel} and some fixed edge $\me$, using that since $f$ is an even function, its derivative is an odd function
\begin{equation*}
\begin{split}
&\frac{d}{d\xi'} \Big \vert_{\xi'=0} K^{(1)}_f((\me,\xi),(\me',\xi')):=c(\me )^{-1} \Bigg(- \delta_{\me, \me'} f'(\xi)\\
&\quad +\sum_{m \in \mathbb N}\Bigg( \sum_{P \in \mathcal P_{\me,\me'}(m)} \mathbb T_P f'(\vert P \vert-\xi)- \sum_{P \in \mathcal P_{\me,-\me'}(m)} \mathbb T_P f'(\vert \me' \vert +\vert P \vert-\xi) \Bigg) \Bigg) \\
&=c(\me )^{-1} \Bigg( -\delta_{\me, \me'} f'(\xi)+ \mathbb T_{\me,\me'} f'(\vert \me \vert-\xi) \\
&\quad + 2 \sum_{\substack{m \in \mathbb N \\ \me''\in t^{-1}(i(\me')) \\ P \in \mathcal P_{\me,\me''}(m)}} \mathbb T_P \left(\frac{c(\me'')}{c(i(\me'))} - \delta_{\me'',-\me'}\right)f'(\vert \me'' \vert+ \vert P \vert-\xi) \Bigg).
\end{split}
\end{equation*}
Similarly, we find that
\begin{equation*}
\begin{split}
&\frac{d}{d\xi'}\Big \vert_{\xi'=0} K^{(2)}_f((\me,\xi),(\me',\xi')):=c(\me )^{-1} \Bigg( \sum_{\substack{m \in \mathbb N \\ \me''\in t^{-1}(i(\me')) \\ P \in \mathcal P_{-\me,\me''}(m)}} \mathbb T_P \left(\frac{2c(\me'')}{c(i(\me'))} - \delta_{\me'',-\me'}\right) \times \\
&f'( \vert \me'' \vert + \vert P \vert-(\vert \me \vert-\xi)) + \mathbb T_{-\me,\me'}f'(\xi) +\sum_{\substack{m \in \mathbb N \\ P \in \mathcal P_{-\me,-\me'}(m) }} \mathbb T_P f'(\vert \me' \vert+\vert P \vert-(\vert \me \vert-\xi))\Bigg) \\
\quad &=c(\me )^{-1} \Bigg( 2 \sum_{\substack{m \in \mathbb N \\ \me''\in t^{-1}(i(\me')) \\ P \in \mathcal P_{-\me,\me''}(m)}} \mathbb T_P \left(\frac{c(\me'')}{c(i(\me'))}- \delta_{\me'',-\me'} \right) f'( \vert \me'' \vert +\vert P \vert-(\vert \me \vert-\xi)) + \mathbb T_{-\me, \me'}f'(\xi) \Bigg).
\end{split}
\end{equation*}
The study of directional derivatives at terminal ends is fully analogous.

\medskip

From the identity
\[ \sum_{\me \in \mathcal E_{\mv}} c(\me) \left( \frac{c(\me')}{c(\mv)} -\delta_{\me',-\me} \right) = c(\me')-c(\me')=0, \]
we infer that 
\begin{equation*}
\begin{split}
 \sum_{\me' \in \mathcal E_{i(\me')}} c(\me') \sum_{\substack{\me''\in t^{-1}(i(\me')) \\ P \in \mathcal P_{-\me,\me''}(m)}} 2\mathbb T_P \left(\frac{c(\me'')}{c(i(\me'))}- \delta_{\me'',-\me'} \right)=0. 
 \end{split}
\end{equation*}

We then define for edges $\me' \in \mathcal E(\mathcal G)$ with $i(\me)=i(\me')$ the quantity $(\beta_0(\me,\me')f)(\xi) := 2 \left( \frac{c(\me)}{c(t(\me))} -\delta_{\me,\me'} \right)f'(\xi)$ and for edges $\me' \in \mathcal E(\mathcal G)$ with $t(\me)=i(\me')$ the expressions 
\[
 (\beta_0(\me,\me')f)(\xi) := 2 \frac{c(\me)}{c(t(\me))} f'(\vert \me \vert-\xi).
 \]
Moreover, for $m \ge 1$
\begin{equation*}
\begin{split}
(\beta_m(\me,\me')f)(\xi) &:= 2 \sum_{\substack{m \in \mathbb N \\ \me''\in t^{-1}(i(\me')) \\ P \in \mathcal P_{\me,\me''}(m)}} \mathbb T_P \left(\frac{c(\me'')}{c(i(\me'))} - \delta_{\me'',-\me'}\right)f'(\vert \me'' \vert+ \vert P \vert-\xi) \\
&+2 \sum_{\substack{m \in \mathbb N \\ \me''\in t^{-1}(i(\me')) \\ P \in \mathcal P_{-\me,\me''}(m)}} \mathbb T_P \left(\frac{c(\me'')}{c(i(\me'))}- \delta_{\me'',-\me'} \right) f'( \vert \me'' \vert +\vert P \vert-(\vert \me \vert-\xi)).
 \end{split}
\end{equation*}
It then follows that the convolution satisfies the Kirchhoff condition \eqref{eq:kirchhoff} due to
\begin{equation*}
\begin{split}
 \sum_{\me' \in \mathcal E_{\mv}} c(\me') \dfrac{\partial K_f((\me,\xi),(\me',0))}{\partial
n_\me} = c(\me)^{-1} \sum_{\me' \in \mathcal E_{\mv}} c(\me') (\beta_m(\me,\me')f)(\xi)= 0. 
\end{split}
\end{equation*}
This concludes the proof.
\end{proof}

\subsection{Convolution semigroups}
\label{sec:convsem}
Let us continue by studying the heat equation (or more generally a convection-diffusion problem) as well as higher-order parabolic equations using our integral kernel framework on infinite metric graphs. 
For equations
\begin{equation*}\label{eq:diffadvpoly}
\partial_t u(t) = \Delta u(t,x) + \xi u(t,x) \text{ and }\partial_t v(t,x) =-(- \Delta)^m v(t,x)
\end{equation*} 
the solutions on $\RR$ are given by convolutions $u(t) = e^{\xi t} h_t*u_0$ and $v(t) = k_t*v_0.$
Here, $h_t$ is the standard heat kernel on $\RR$ given by $h_t(x) = \frac{ e^{-\frac{x^2}{4t} }}{(4\pi t)^{1/2}}$ 
and the kernel $k_t$, see~\cite[p.~326]{Dav95b}, satisfies estimates
\begin{equation}
\label{eq:otherkernels}
 \vert k_t(x,y) \vert \le c_1 t^{-1/(2m)} \operatorname{exp}\left(-c_2 \frac{\vert x-y \vert^{\frac{2m}{2m-1}}}{t^{1/(2m-1)}}+c_3t\right)
 \end{equation}
for some $c_1,c_2,c_3>0$, cf. \cite[Theorem $1$]{HOLE} and \cite{DAVIES1995141}.

\begin{theo}
\label{theo:theo1}
Let $\mathcal G$ be a metric graph and consider the convolution defined in \eqref{eq:deficonv} and the integral kernels $(h_t)_{t>0}$ and $(k_t)_{t>0}$ defined in Section \ref{sec:convsem}. Let also $\xi\in \CC$.
Then the following assertions hold.

\begin{enumerate}
\item The $C_0$-semigroup $(T_tf)(x)=e^{\xi t} (h_t*_{\mathcal G}f)(x)$ 
\[ \partial_t v(t,x) = \Delta v(t,x) + \xi v(t,x), \quad v(0,x)=f(x) \]
on each of the following spaces:
\begin{itemize}
\item $L^p(\mathcal G)$, $p \in [1,\infty)$; 
\item $C_0(\mathcal G)$;
\item $BUC(\mathcal G)$.
\end{itemize}

\item The $C_0$-semigroup $(T_tf)(x)=e^{\xi t}(k_t*_{\mathcal G}f)(x)$ satisfies the equation 
\[ \partial_t v(t,x) =-(- \Delta)^m v(t,x)+ \xi v(t,x), \quad v(0,x)=f(x)\]
on each of the following spaces:
\begin{itemize}
\item $L^p(\mathcal G)$, $p \in [1,\infty)$; 
\item $C_0(\mathcal G)$;
\item $BUC(\mathcal G)$.
\end{itemize} 
	
\end{enumerate} 
Additionally, at the vertices $\mv$ of $\mathcal G$ the function $v(t):=T_t f$ satisfies in either case
\begin{itemize}
\item continuity conditions on $(\Delta^{m}v)(t,\mv)$, $t>0$ and $m \in \mathbb N_0$;
\item Kirchhoff (i.e., zero sum) conditions on $(\frac{\partial }{\partial n}\Delta^{m}v)(t,\mv)$, $t>0$ and $m \in \mathbb N_0$.
\end{itemize}
\end{theo}

\begin{proof}
For the heat kernel it is obvious that $\Delta h_t \in \mathcal L^1.$ The kernel $k_t$ is explicitly given, for some normalization constant $\alpha_m$ by the formula
\[ k_t(x):=\frac{\alpha_{m}}{t^{1/2m}} \int_0^{\infty} e^{-s^{2m}} \cos(xs/t^{1/2m}) \ ds. \]
Differentiating this kernel in time shows that 
\begin{equation*}
\begin{split}
\Delta^{m} k_t(x):=\frac{\alpha_{m}}{t^{1/2m}} \int_0^{\infty} \frac{e^{-s^{2m}} s^{2m} \cos\left(\frac{xs}{t^{1/2m}}\right)}{t} \ ds.
\end{split}
\end{equation*}
It then follows since $k_t$ is the Fourier transform of an entire function that is integrable on any slice $\mathbb R+i\eta$ for $\eta \in \RR$, that $k_t$ decays faster than any exponential, i.e., for any $b>0$ it follows that $\left\vert \Delta^{m} k_t(x) \right\vert \le C_b e^{-b \vert x \vert}.$
This implies that also $\Delta^{m} k_t \in \mathcal L^1.$ In either case, this immediately implies that $v(t,x):=T_tf(x)$ solves the respective equation on the edges. That the solution satisfies the correct boundary conditions follows from the symmetry of the kernels $h_t,k_t$ (even functions) and Lemma \ref{eq:lemmbc}. More precisely, $\Delta^m h_t, \Delta^m k_t$ are even functions in $\mathcal L^1$ for all $m \in \mathbb N_0$ and thus the Lemma applies to every order.
\end{proof}

We will postpone the identification of the generator with \eqref{eq:domain} until Section \ref{sec:contraction} and, in particular, Proposition~\ref{prop:genercontr} (heat semigroups) and Theorem~\ref{theo:inc} (higher order parabolic problems).

\section{Spectral independence and Schrödinger semigroups}
\label{sec:Schrsemi}
We start by analyzing mapping properties of Schrödinger semigroups. The ultracontractivity estimates of the following subsection will also imply a crucial role in the proof of spectral independence in the subsequent subsection.
\subsection{Ultracontractivity and Schrödinger semigroups}

We consider Schrödinger operators $H=-\Delta + V$ with potentials $V \in L^{\infty}(\mathcal G)$. We also introduce the operator $H_{-}=-\Delta + V_{-}$ where $V_{-}$ is the negative part of $V.$
Consider the Hilbert space $L^2(\mathcal G)$ and the following closed symmetric form 
\[\tau(u,v) := \langle H u ,v \rangle_{L^2} + c \langle u,v \rangle_{L^2} \]
where $c$ is such that $\mathfrak{\tau}$ is a positive form. 
We then define another form 
\[\tau_{-}(u,v) := \langle H_{-} u ,v \rangle_{L^2} + c \langle u,v \rangle_{L^2}. \]
Using that $\nabla \vert u \vert = \sgn(u) \nabla u$, we find for $\tau [u]:=\tau(u,u)$ that $\tau$ satisfies the first Beurling-Deny criterion $\tau[ \vert u \vert] \le \tau [u] , \text{ for } u \in H^1(\mathcal G).$
Moreover, we have the domination of forms $\tau_{-}(u,v) \le \tau(u,v),$ for $0\le u,v \in H^1(\mathcal G).$
Thus, we have $\left\lvert e^{-tH}f \right\rvert \le e^{-tH_{-}}\vert f \vert$ for all $t\ge 0$ and all $f \in L^2(\mathcal G).$

This implies that to study $L^q$-$L^p$ smoothing of the Schrödinger semigroup $(e^{-tH_p})_{t \ge 0}$ it suffices to study $T_p(t):=e^{-tH_{p-}}$ where the index $p$ indicates the reference space $L^p(\mathcal G)$. To simplify the notation we will just write $H$ instead of $H_{-}$ and $V$ instead of $V_{-}$ in the sequel. We will also write $T_p^{\kappa}(t):=e^{-t (-\Delta_p+\kappa V_{-})}$ for the semigroup generated by $\Delta_p-\kappa V_{-}$ for some $\kappa \in \RR.$

\begin{theo}[Ultracontractivity] 
\label{theo:hypercontr}
Let $\mathcal G$ be a metric graph. Then the following assertions hold.

\begin{enumerate}
\item Let $V\in L^\infty(\mathcal G)$. Then the semigroup $(T(t))_{t\ge 0}=(e^{t(\Delta+V)})_{t\ge 0}$ consists of operators that are bounded from $L^q(\mathcal G)$ to $L^p(\mathcal G)$ for all $p \in [q,\infty]$ and $q \in [1,\infty)$; we denote their realization in $L^q(\mathcal G)$ by $T_q(t)$.

If we additionally impose Assumption \ref{ass2}, then there exists $M=M(q,p)>0$ such that
\begin{equation}\label{ultr1}
||T_q(t)f||_{L^p}\leq Mt^{-\frac{1}{2}(q^{-1}-p^{-1})}e^{t\|V^+\|_\infty}||f||_{L^q}\quad\textrm{for all}\ t>0.
\end{equation}

\item Let $m\ge 2$. Then the semigroup $(S(t))_{t\ge 0}=(e^{-t(-\Delta)^m})_{t\ge 0}$ consists of operators that are bounded from $L^q(\mathcal G)$ to $L^p(\mathcal G)$ for all $p \in [q,\infty]$ and $q \in [1,\infty)$; we denote their realization in $L^q(\mathcal G)$ by $S_q(t)$.

If we additionally impose Assumption \ref{ass2}, then there exists $M=M(q,p)>0$ such that
\begin{equation}\label{ultr2}
||S_q(t)f||_{L^p}\leq Mt^{-\frac{1}{2m}(q^{-1}-p^{-1})}||f||_{L^q}\quad\textrm{for all}\ t>0.
\end{equation}
\end{enumerate}
\end{theo}

\begin{proof}
We introduce the constant $C_t:=\left\lVert T^0_1(t) \right\rVert_{\mathcal L(L^1,L^{\infty})}$ for $t>0$. This quantity is finite as the integral kernel \eqref{eq:kernel} is uniformly bounded, which follows from \eqref{eq:expthree2} and the finiteness of the $\mathcal L^1$-norm. Let $\kappa>1$ such that $\kappa^{-1}+\eta^{-1}=1$ for some $\eta \in \mathbb N$.

Since the negative part of the potential is relatively $L^1$ zero-bounded, it follows that the Schrödinger operator $H^{\kappa}_{1}:=-\Delta_1+\kappa V$ generates a semigroup $T^{\kappa}_1(t)$ on $L^1(\mathcal G)$ such that for some $M>0$ and $\omega \in \mathbb R$ we have $\left\lVert T_1^{\kappa} \right\rVert_{\mathcal L(L^1)} \le M e^{\omega t}.$

Let $F_t(r):=T^{r \kappa}_2(t) \in \mathcal L(L^2),$ then by \cite[Theo.~$1.6.6$]{Ka10} and positivity of $T_2^{r \kappa}(t/n)$ we have for $r \in [0,1]$
\[ \left\lvert \left(e^{ i s \kappa V\frac{t}{n}} T_2^{r \kappa}(t/n) \right)^n f \right\rvert \le T_2^{r \kappa}(t) \vert f \vert. \]

By the Trotter--Kato formula we have from taking the limit $n \rightarrow \infty$ in the above equation $\left\lvert T_2^{(r+is)\kappa}(t) f \right\rvert \le T_2^{r\kappa}(t) \vert f \vert.$
This implies that 
\begin{equation*}
\begin{split}
\left\lVert F_t(is) f \right\rVert_{L^{\infty}} &\le \left\lVert F_t(0) \vert f \vert \right\rVert_{L^{\infty}} \le \left\lVert F_t(0) \right\rVert_{\mathcal L(L^1,L^{\infty})} \left\lVert f \right\rVert_{L^1} \text{ and } \\
\left\lVert F_t(1+is) f \right\rVert_{L^1} &\le \left\lVert F_t(1) \vert f \vert \right\rVert_{L^1} \le \left\lVert F_t(1) \right\rVert_{\mathcal L(L^1,L^1)} \left\lVert f \right\rVert_{L^1}
\end{split}
\end{equation*}
such that 
\begin{equation*}
\begin{split}
\left\Vert F_t(is) \right\Vert_{\mathcal L(L^1,L^{\infty})}&\le \left\lVert F_t(0) \right\rVert_{\mathcal L(L^1,L^\infty)}=C_t \text{ and }\\
\left\lVert F_t(1+is) \right\rVert_{\mathcal L(L^1,L^1)}&\le\left\lVert F_t(1) \right\rVert_{\mathcal L(L^1,L^1)} \le M e^{\omega t}.
\end{split}
\end{equation*}

Stein's interpolation theorem then implies that $\left\lVert F_t(s) \right\rVert_{\mathcal L(L^1,L^{s^{-1}})} \le C_t^{1-s}(Me^{\omega t})^s$
such that by choosing $s= \kappa^{-1}$ we find $\left\lVert F_t(s) \right\rVert_{\mathcal L(L^1,L^{\kappa})} \le C_t^{\eta^{-1}}(Me^{\omega t})^{\kappa^{-1}}.$
The Riesz--Thorin theorem implies then by interpolating $T_1(t): L^1(\mathcal G) \rightarrow L^{\kappa}(\mathcal G)$ and $T_1(t)^*: L^{\eta}(\mathcal G) \rightarrow L^{\infty}(\mathcal G)$ with $\eta \in \mathbb N$ that $T_1(t)$ induces bounded operators $T_q(t): L^q(\mathcal G) \rightarrow L^{p}(\mathcal G) $ with $q^{-1}-p^{-1}=\eta^{-1}.$

We then have that for $q_0=1 < q_1<...<q_\eta=\infty$ with $q_{j-1}^{-1}-q_j^{-1}=\eta^{-1}$
\[\left\lVert T_1(t) \right\rVert_{\mathcal L(L^1,L^{\infty})} \le \prod_{j=1}^\eta \left\lVert T_1(t/\eta) \right\rVert_{\mathcal L(L^{q_{j-1}},L^{q_j})} \le C_{t/\eta} M^{\eta-1}e^{t\omega(\eta-1)}. \]

For higher-order parabolic systems the $L^q \rightarrow L^{\infty}$ boundedness of $(S(t))$ follows from the $\mathcal L^1$-boundedness of the integral kernel. Thus, the Riesz--Thorin theorem implies the $L^q \rightarrow L^p$ boundedness of $(S_q(t)).$

In order to obtain estimates \eqref{ultr1} and~\eqref{ultr2} we observe that by Theorem \ref{theo:inc} we know that the semigroup $(T_2(t))$ is analytic on $L^2(\mathcal G)$. Its generator, that we denote here by $H_1=\Delta-V$, is associated with the quadratic form
\begin{align*}
a_1(u)&=\int_{\mathcal G} |u'(x)|^2\,dx+\int_{\mathcal G}V(x)|u(x)|^2\,dx,\qquad u\in H^1(\mathcal G).
\end{align*}
 Moreover, $(S_2(t))$ on $L^2(\mathcal G)$ is analytic and is generated by $H_m:=-(-\Delta)^m$ for $m\geq2$ with domain
\[D(H_m)=\left\{u\in H^{2m}(\mathcal G)\,:\,u^{(2n)}\in C(\mathcal G),\sum_{ \me \in \mathcal{E}_{\mv}(\mathcal G) } \dfrac{\partial }{\partial
n_\me}u^{(2n)}_\me(\mv) =0, 0\leq n<m\right\}.\]
It is easy to see (adapt the proof of \cite[Theorem 4.3]{GreMug19}) that this operator is associated with the quadratic form 
\begin{align*}
a_m(u)&=\int_{\mathcal G}|u^{(m)}(x)|^2 \,dx,\\
D(a_m)&=\left\{u\in H^m(\mathcal G)\,:\,u^{(2n)}\in C(\mathcal G), \sum_{ \me \in \mathcal{E}_{\mv}(\mathcal G) } \dfrac{\partial }{\partial
n_\me}u^{(2n)}_\me(\mv) =0, 0\leq n<\left\lfloor\frac{m}{2}\right\rfloor\right\}.
\end{align*}
Let $u\in H^1(\mathcal G)$, then an application of Gagliardo-Nirenberg inequality \eqref{gng} gives
\begin{align}\label{gn1}
||u||_{L^\infty}&\leq c_1||u'||_{L^2}^\frac{1}{2}||u||_{L^2}^\frac{1}{2}\\\nonumber&\leq c_1\left(a_1(u)+||V||_{L^\infty}||u||_{L^2}^2\right)^\frac{1}{4}||u||_{L^2}^\frac{1}{2}\\\nonumber&\leq \tilde{c}_1 a_1(u)^\frac14||u||_{L^2}^\frac12+\tilde{c}_2||u||_{L^2}.
\end{align}
Let now $u\in H^m(\mathcal G)$, $m\geq2$, again an application of \eqref{gng} gives
\begin{align}\label{gn2}
||u||_{L^\infty}&\leq c_1||u^{(m)}||_{L^2}^\frac{1}{2m}||u||_{L^2}^\frac{2m-1}{2m}\\\nonumber&=c_1a_m(u)^\frac{1}{4m}||u||_{L^2}^\frac{2m-1}{2m}.
\end{align}
Now we let $u=P(t)f$, where $P(t)=T_2(t)$ for $m=1$ and $P(t)=S(t)$ for $m\geq2$. By \eqref{gn1}, \eqref{gn2} and using contractivity (quasi contractivity for $(T_2(t))$) and analyticity on $L^2$ and Cauchy--Schwarz inequality, it follows that for all $t > 0$
\begin{align*}
||P(t)f||_{L^\infty}&\leq k_1\langle -H_mP(t)f,P(t)f\rangle^\frac{1}{4m}||P(t)f||_{L^2}^\frac{2m-1}{2m}+k_2||P(t)f||_{L^2}\\&\leq k_1e^{ct||V_+||_{L^\infty}}||H_mP(t)f||_{L^2}^\frac{1}{4m}||P(t)f||_{L^2}^\frac{1}{4m}||f||_{L^2}^\frac{2m-1}{2m}+k_2e^{ct||V_+||_{L^\infty}}||f||_{L^2}\\&\leq e^{\tilde{c}t}(k_1t^{-\frac{1}{4m}}+k_2)||f||_{L^2}\\&\leq Mt^{-\frac{1}{4m}}e^{\omega t}||f||_{L^2}.
\end{align*}
Observe that in the case $m\ge 2$ the constants $c,\tilde c,k_2,\omega$ vanish.
Now, applying this last inequality to the semigroup generated by $H_m^*$, by duality this yields the same bound to hold for the $L^1-L^2$ norm as well and then the claimed estimates follow for $q=1,p=\infty.$ Riesz--Thorin interpolation between $(1,\infty)$ and $(2,2)$ implies that the result holds also for $q,q'$ where $q^{-1}+q'^{-1}=1.$ Another application of Riesz--Thorin interpolation between $(q,q)$ and $(q,q')$ yields then the result for pairs $(q,p)$.
\end{proof}

\subsection{Spectral independence of Schrödinger operators and $-(-\Delta)^m$}
\label{sec:Spectralindep}

From the ultracontractivity proved in Theorem \ref{theo:hypercontr}, let us now prove the $p$-independence of the spectrum of generators.
\begin{theo}[$L^p$-independence]
\label{theo:specind}
Let $\mathcal G$ be a metric graph that satisfies Assumption \ref{ass2}. The spectrum of the (negative) Schrödinger operator $H_{p}:=\Delta_p-V$, with $V \in L^{\infty}(\mathcal G)$, on $L^p(\mathcal G)$ with $p \in [1,\infty)$ and Laplacians of higher order $H_p =- (-\Delta_p)^m$ on $L^p(\mathcal G)$ with $m \ge 1$, does not depend on $p$, i.e., $\sigma(H_2)=\sigma(H_p) \subset \RR$. The inclusion $\sigma(H_2)\subset \sigma(H_p)$ holds also without any growth assumption on the graph. Moreover, the resolvents of operators $H_p$ are consistent. 
\end{theo}
\begin{proof}
\medskip

``${{ \sigma(H_2) \subset \sigma(H_p)}}$''
Let $p \in [q,\infty]$, then $T_q(t) \in \mathcal L(L^q(\mathcal G),L^p(\mathcal G)\cap L^q(\mathcal G))$ for $t >0$, we have for $x \in D(H_q)$
\begin{equation}
\begin{split}
\label{eq:commuting}
H_pT_q(t)x &=\lim_{h \downarrow 0} \frac{T_p(h)-\operatorname{id}}{h} T_q(t)x = \lim_{h \downarrow 0} \frac{T_q(h)-\operatorname{id}}{h} T_q(t)x \\
&= \lim_{h \downarrow 0} T_q(t)\frac{T_q(h)-\operatorname{id}}{h}x = T_q(t)H_qx.
\end{split}
\end{equation}
Let $\lambda \in \rho(H_q) \cap \rho(H_p),$ then from \eqref{eq:commuting} we conclude that
\begin{equation*}
\begin{split}
(\lambda-H_p)^{-1}T_q(t) 
&= (\lambda-H_p)^{-1} T_q(t)(\lambda-H_q)(\lambda-H_q)^{-1} \\
&= (\lambda-H_p)^{-1} (\lambda-H_p)T_q(t)(\lambda-H_q)^{-1} = T_q(t)(\lambda-H_q)^{-1}.
\end{split}
\end{equation*}
Taking the limit $t \downarrow 0,$ we obtain that resolvents are consistent
\begin{equation}
\label{eq:consres}
(\lambda-H_p)^{-1} \vert_{L^q \cap L^p} =(\lambda-H_q)^{-1} \vert_{L^q \cap L^p}.
\end{equation}
By duality, the same is true for $2 \le p \le q \le \infty.$
Thus, for $q^{-1}+q'^{-1}=1$ and $\lambda \in \rho(H_q)=\rho(H_{q'})$
\[ (\lambda-H_q')^{-1}\vert_{L^q \cap L^{q'}} = (\lambda-H_q)^{-1} \vert_{L^q \cap L^{q'}}. \]

For $x \in L^p(\mathcal G) \cap L^q(\mathcal G)$, \eqref{eq:commuting} implies 
\begin{equation*}
\begin{split}
&\lim_{t \downarrow 0} T_q(t)(\lambda-H_q)^{-1} x = (\lambda-H_q)^{-1} x \text{ and } \\
&\lim_{t \downarrow 0} (\lambda-H_p)T_q(t)(\lambda-H_q)^{-1} x = \lim_{t \downarrow 0} T_q(t) x=x.
\end{split}
\end{equation*}
Using that $\lambda-H_p$ is closed implies then that $(\lambda-H_p)(\lambda-H_q)^{-1} x = x$
and thus, we conclude $(\lambda-H_q)^{-1} x = (\lambda-H_p)^{-1}x.$
It remains to notice then that by the Riesz--Thorin theorem, by interpolating between $(\lambda-H_{q'})^{-1}$ and $(\lambda-H_q)^{-1}$, it follows that $(\lambda-H_q)^{-1}$ is also continuous on $L^p$ which implies that $(\lambda-H_p)$ is surjective. It is also injective, since otherwise there is a non-zero $u \in D(H_p)$ such that $H_p u = \lambda u$ and by \eqref{eq:commuting}
\[ H_{q'} \ T_p(t)u = T_p(t)H_p u = \lambda \ T_p(t) u \]
which contradicts $\lambda \in \rho(H_{q'}).$

\medskip

``$\sigma(H_p) \subset \sigma(H_2)$'': 
To show the converse implication, we use that the resolvents are consistent for $z \in \rho(H_2) \cap \rho(H_p)$ by \eqref{eq:consres}.
We then notice that the resolvent $z\mapsto (H_2-z)^{-1} \vert_{L^p \cap L^2}$ extends on $\rho(H_2)$ to a bounded operator on $L^p$ which follows from the Combes--Thomas estimate in Proposition \ref{prop:CTE}. By unique continuation, this extension is the resolvent $(H_p-z)^{-1}$.
\end{proof}

\label{s:intr}

\medskip

We now continue by extending our results to spaces $C_0(\mathcal G)$ and $BUC(\mathcal G).$ 
We start with two far-reaching results from semigroup theory on the spectrum of the (higher-order) Laplacians $-(-\Delta)^m$ for $m \ge 1$ and thereby provide an explanation why $C_0(\mathcal G)$ and $BUC(\mathcal G)$ fit naturally in the framework of $L^p(\mathcal G)$-spaces.

Since duals of $C_0$-semigroups may not be strongly continuous, the concept of sun duals has been introduced to restore strong continuity. Given a $C_0$-semigroup $(T(t))$ on a Banach space $X_c$, its sun dual space is defined by 
\begin{equation*}
{X_{\operatorname{c}}}^{\odot}:=\left\{x' \in {X_{\operatorname{c}}}'; \lim_{t \downarrow 0} \left\lVert T(t)' x' -x' \right\rVert_{{X_{\operatorname{c}}}'}=0 \right\}.
\end{equation*}
The sun dual semigroup is then defined as $T(t)^{\odot}:=T(t)'\vert_{{X_{\operatorname{c}}}^{\odot}}.$
When it comes to spectral properties on different spaces, the relevance of this concept is due to the following result. 
\begin{prop}{\cite[Theo.~$2.1$]{Hempel1994}}
\label{semigroupapproach}
Let $(T(t))$ be a strongly continuous semigroup with generator $A$ and $Y \subset {X_{\operatorname{c}}}^{\odot}$ be a closed subspace that is invariant under $T(t)'.$ If $Y$ is \emph{equi-norming} for ${X_{\operatorname{c}}},$ i.e.~for all $x \in {X_{\operatorname{c}}}$ 
\[\left\lVert x \right\rVert_{X_{\operatorname{c}}}= \sup_{y \in Y; \left\lVert y\right\rVert=1} \left\langle y,x \right\rangle,
\]
 then $\rho_{\infty}(A)=\rho_{\infty}(A'\vert_Y)$, where $\rho_{\infty}$ is the connected component of the resolvent set containing a right half plane; and we have 
 \[
 (\lambda-A\vert_Y)^{-1} = (\lambda-A)\vert_{Y}^{-1}\qquad \hbox{ for all }\lambda \in \rho_{\infty}(A).
 \]

Moreover, the generator of $T(t)'\vert_Y$ is $A'\vert_{ Y}$. 
\end{prop}

For any $g \in (L^1)^*=L^{\infty}$ the convolution with a kernel in $\mathcal L^1$ becomes by Lemma \ref{techlemma} immediately a $BUC(\mathcal G)$ function. Thus, in particular for semigroups defined in Theorem \ref{theo:theo1}, $L^{1}(\mathcal G)^{\odot}\supset BUC(\mathcal G).$
Theorem \ref{semigroupapproach} however, also applies to the closed subspace $\operatorname{C}_0(\mathcal G)$ of bounded uniformly continuous functions, since there is always a normalized sequence of $\operatorname{C}_0(\mathcal G)$ functions in this space converging a.e.~ to $\sgn(f)$ for any $f \in L^1(\mathcal G).$

Thus, we have established at this point already the following Corollary:
\begin{corr}
\label{corr:specequiv}
The connected component of the resolvent set containing a right half plane, that we denote by $\rho_{\infty}$, satisfies for all $m \in \mathbb N_0$ and $V \in L^{\infty}$
 \begin{equation*}
 \begin{split}
 &\rho_{\infty}(-(-\Delta_{L^1})^m)=\rho_{\infty}(-(-\Delta_{C_0})^m)=\rho_{\infty}(-(-\Delta_{BUC})^m) \text{ and }\\
 &\rho_{\infty}(\Delta_{L^1}-V)=\rho_{\infty}(\Delta_{C_0}-V)=\rho_{\infty}(\Delta_{BUC}-V).
 \end{split}
 \end{equation*}
 In particular, let Assumption \ref{ass2} be satisfied. The spectrum of the Schrödinger operator $H_{p}:=-\Delta_p+V$, with $V \in L^{\infty}(\mathcal G)$, on $L^p(\mathcal G)$ with $p \in [1,\infty)$ and Laplacians of higher order $H_p =- (-\Delta_p)^m$ on $L^p(\mathcal G)$ with $m \ge 1$, does not depend on $p$, i.e., $\sigma(H_2)=\sigma(H_p) \subset \RR$ and coincides with the spectrum on $C_0(\mathcal G)$ and $BUC(\mathcal G).$
\end{corr}
\begin{proof}
The first part is immediate from Proposition \ref{semigroupapproach}. The second one follows by combining the first part and Theorem \ref{theo:specind}.
\end{proof}
 
\subsection{Counterexamples to spectral independence}

We now provide an example that shows that the spectrum on $BUC(\mathcal G)$ can also differ from the spectrum on $L^2(\mathcal G).$

\begin{ex}[Rapidly growing graph]
\begin{figure}[h]
 \includegraphics[width=10cm]{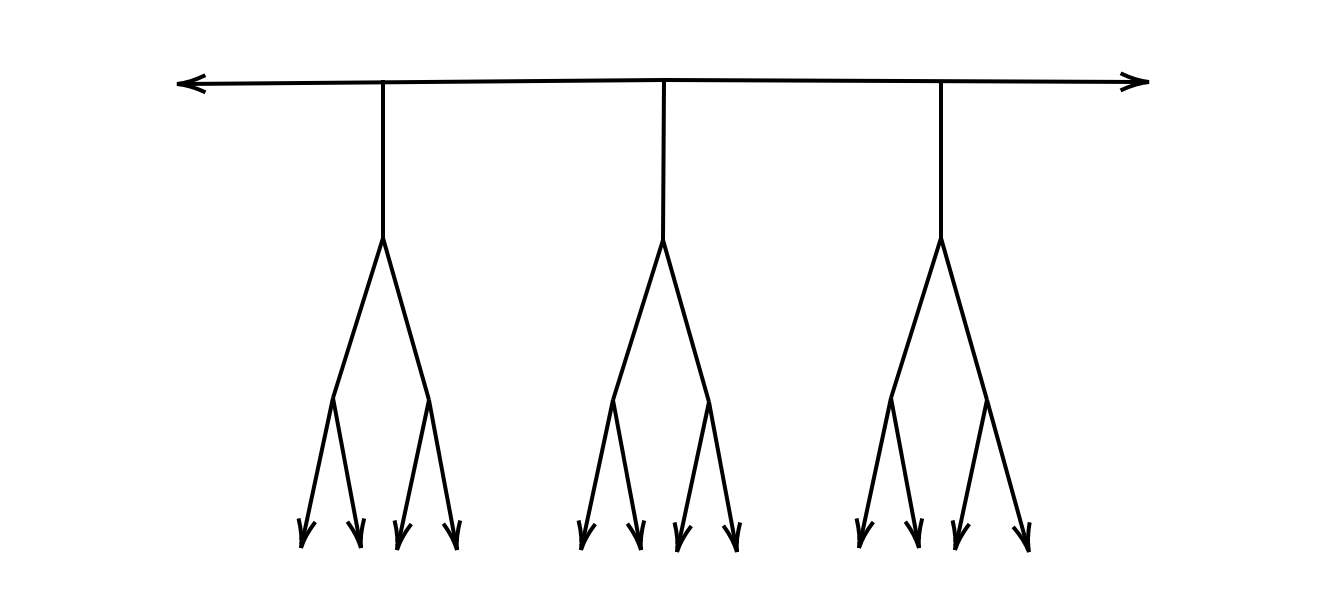}
 \caption{A counterexample to spectral independence}
 \label{fig:branch}
\end{figure}

Consider the Laplacian $\Delta$ on the infinite graph $\mathcal G$ illustrated in Figure \ref{fig:branch}.
Take an edge $e=(n,d)$ of the graph, where we parametrize all edges connected to the same node $n \in \ZZ$, on the upper horizontal line, with depth $d \in \NN_0$ (number of edges away from the horizontal line) by
\[ \psi_\me(t) = \frac{s_{\lambda}(t)e^{i\theta(n+d)} + e^{i\theta(n+d+1)}(s_{\lambda}(1)s_{\lambda}(t)-c_{\lambda}(1)s_{\lambda}(t))}{s_{\lambda}(1)}. \]
This defines a family of eigenfunctions $\Psi_{\theta}= (\psi_\me)$ in $BUC(\mathcal G)$ parametrized by a parameter $\theta \in [0,2\pi)$ such that 
\[ \Delta \Psi_{\theta} =\left( \arccos\left(\frac{2e^{i\theta}+e^{-i\theta}}{3} \right)\right)^2 \Psi_{\theta}.\]
However, $\left(\arccos\left(\frac{2e^{i\theta}+e^{-i\theta}}{3} \right)\right)^2$ takes also values with non-zero imaginary part. This is impossible for the self-adjoint Laplacian we defined on $L^2(\mathcal G)$. 
\end{ex}

While in the previous example the rapid growth rate of the graph $\mathcal G$ caused the spectral degeneracy between spaces $L^2(\mathcal G)$ and $BUC(\mathcal G)$, it is the unboundedness of the potential in the following example that causes the spectrum to depend on the underlying space:

\medskip
\begin{ex}[Harmonic oscillator]
Let us start by defining the operator $H = -\Delta + V$ for positive $V \in L^{\infty}_{\operatorname{loc}}$, defined weakly as the operator associated with the form $Q_H$ given by	
\[ Q_H(f)= \int_{\mathcal G} \vert f'(x) \vert^2+V(x) \vert f(x) \vert^2 \ dx, \qquad f\in D(Q_H),\] where $D(Q_H)=\{f\in L^2(\mathcal G): Q_H(f)\in L^2(\mathcal G)\}$. We then have that $H$ has compact resolvent if $V$ satisfies a suitable growth condition.
\begin{prop}
\label{prop:Schrgr}
If $V(x) \ge c \langle x \rangle^{\kappa}$, where $\langle x \rangle:=(1+ \vert x \vert^2)^{1/2}$ for some $\kappa>0$, $c>0$, and $x$ such that $d(x,0)>r$ for some fixed $r>0$, then the operator $H = -\Delta + V$-- defined weakly as the operator associated with the form $Q_H$ -- on $L^2(\mathcal G)$ has compact resolvent.
\end{prop}
\begin{proof}
To show that $H$ has compact resolvent, it suffices to show that the form domain $D(Q_H)$ is compactly embedded in $L^2(\mathcal G)$. To see this, we first notice that
\[ \lim_{r \rightarrow \infty} \int_{d(x,0)>r} \vert f(x)\vert^2 \ dx \le \lim_{r \rightarrow \infty} \frac{1}{\langle r \rangle^{\kappa}} \int_{d(x,0)>r} \langle x \rangle^{\kappa} \vert f(x) \vert^2 \ dx=0.\]
Then, since the first derivative is the generator of translations $(\tau_a f)=f(\bullet-a)$, this implies that $\lim_{a \rightarrow 0} \Vert \tau_a f - f \Vert_{L^2(\mathcal G)}=0.$ The Fréchet--Kolmogorov theorem implies therefore that the embedding $D(Q_H) \hookrightarrow L^2(\mathcal G)$ is compact.
\end{proof}

The following Corollary, which follows from interpolation theory \cite[Theo. $8.2.12$]{davies_2007}, shows that spectral equivalence holds, under the conditions of Proposition \ref{prop:Schrgr}, at least for $p \in (1,\infty):$
\begin{corr}
If $T_2(t)=e^{-H_2t}$ is compact for $0<t<\infty$, then $e^{-H_p t}$ is compact for $t \in (0,\infty)$ and $p \in (1,\infty)$ and analytic. Moreover, $\sigma(H_p)= \sigma(H_2).$
\end{corr}

We will now see that the spectral invariance can fail dramatically for $p=1$ as soon as the potential is unbounded:
Consider equilateral graphs $\mathcal G$ that are embedded into $\mathbb R^d.$ We then start by introducing the Schrödinger operator
\begin{equation}
\begin{split}
\label{eq:Hosc}
(H_{\operatorname{osc}}u)(x)=(-\Delta u)(x) + \omega^2 \langle x,\me-i(\me) \rangle^2 u(x)
\end{split}
\end{equation}
where $\langle \bullet, \bullet \rangle$ denotes the usual inner product.

Let us recall that the harmonic oscillator on $\RR^d$ given as $H_{\mathbb R^n} = -\Delta_{\RR^n} + \omega^2 \vert x \vert^2$
has a complete orthonormal basis induced by product states $\psi_j({\bf{x}}) := \prod_{i=1}^d \psi_{j_i}(x_i) \in L^2(\RR^d).$
It is then easy to check that the above eigenfunctions are also, for indices $j_k=j_i$, eigenfunctions to \eqref{eq:Hosc}.
More precisely, let $\me$ be an edge in direction $\me_i$: then we have $(H_{\operatorname{osc}}\psi)_\me(x_i)=\lambda_{j_i}\psi_\me(x_{ i }).$

In contrast to this, it is known \cite[Theo. $3$]{DAVIES1986126} that on $L^1(\RR)$ the operator
\[ H_{\mathbb R} = -\Delta_{\RR} + \omega^2 \vert x \vert^2 \] 
has spectrum $\sigma(H_{\mathbb R}) = \left\{ z \in \CC: \Re(z)>0 \right\}$ and the spectrum consists of eigenfunctions of multiplicity $2.$

Taking any of these eigenfunctions $\varphi_z \in L^1(\RR)$ we can define an eigenfunction $\varphi_z({\bf{x}}) := \prod_{i=1}^d \varphi_{z}(x_i)$ to $H_{\operatorname{osc}}$. This shows that that the $L^1$ spectrum of $H_{\operatorname{osc}}$ contains the right half plane. In particular, the semigroup generated by $H_{\operatorname{osc}}$ on $L^1(\mathcal G)$ is not analytic.
\end{ex}

\section{Properties of Schrödinger semigroups}

\label{sec:contraction}

We can now give a sufficient condition for Schrödinger operators $H=-\Delta+V$ to be accretive. Thus, by the Lumer--Phillips theorem, $-H$ generates a contractive semigroups on $L^p$-spaces as well as spaces $C_0$ and $BUC$. 
\begin{prop}\label{prop:genercontr}
The Schrödinger operator $H_{X_{\operatorname{c}}}$, with positive potential $V$ that is relatively $\Delta_{X_{\operatorname{c}}}$ bounded with relative $\Delta_{X_{\operatorname{c}}}$-bound $<1,$ is accretive on ${X_{\operatorname{c}}}=L^p(\mathcal G)$ for all $p \in [1,\infty)$ as well as spaces $X_{\operatorname{c}}=C_0(\mathcal G)$ and $BUC(\mathcal G)$. 
In particular, the heat semigroup constructed in Subsection \ref{sec:convsem} is contractive on all spaces $X_c= L^p(\mathcal G)$ with $p \in [1,\infty)$ and $C_0(\mathcal G)$, $BUC(\mathcal G)$ and generated by $\Delta$ with domain given in \eqref{eq:domain}.
\end{prop}
\begin{proof}
The statement is clear on $L^2(\mathcal G)$ by the spectral theorem.

Now, consider $H_{L^1}$. To show that this operator is accretive, we take $f \in D(H_{L^1})$ and define the functional
\[ \varphi_f(g):=\int_{\mathcal G}g(x) \ \varphi(f)(x) \ dx \text{ with }\varphi(f)(x)=\sgn(f(x)) \left\lVert f \right\rVert_{L^1} \]
from the \textit{duality set}, see \cite[\S~II.3]{EngNag00}. 
It suffices then to show that $\varphi_f(-H_{L^1}f) \le 0$ in order to conclude by \cite[Prop.~II.3.23]{EngNag00} that the generator $H_{L^1}$ is accretive, i.e.~for all $\lambda >0$ 
\begin{equation*}
\left\lVert (-H_{L^1}-\lambda)^{-1} \right\rVert_{\mathcal L(L^1(\mathcal G))} \le \lambda^{-1},
\end{equation*}
cf.\@ \cite[Prop.~II.3.14]{EngNag00}.
To see this, it suffices to look at one fixed edge $\me \in\mathcal{E}(\mathcal G).$
There, we assume that $f_\me$ has $x_1,...,x_n$ (possibly none) zeros such that $\me=\underbrace{(i(\me),x_1)}_{=:I_1}\cup...\cup \underbrace{(x_n,t(\me))}_{=:I_{n+1}}.$
Then, if $\sgn(f\vert_{I_i})=1$ we have $f'(x_i)\le 0$ and if $\sgn(f\vert_{I_i})=-1$ we obtain $f'(x_i)\ge 0.$ Thus, we have for $i=2,..,n$ that
\[ \varphi_f(-H_{L^1}f)= \int_{I_i} (\Delta f)_{\me}(x) \varphi(f)_{\me}(x) - V_{\me}(x) \vert f_{\me}(x) \vert \left\lVert f \right\rVert_{L^1} \ dx \le 0.\]
Hence, it follows that
\begin{equation*}
\begin{split}
 \int_{\me} (-H_{L^1}f)_{\me}(x) \varphi(f)_{\me}(x)dx &\le \int_{\me} (\Delta f)_{\me}(x) \varphi(f)_{\me}(x)dx = \left( \frac{\partial f_{\me}}{\partial n_{\me}}(t(e))+\frac{\partial f_{\me}}{\partial n_{\me}}(i(e)) \right).
\end{split}
\end{equation*}
The Kirchhoff condition implies then by summing over $\me \in \mathcal E(\mathcal G)$ that
\[ \int_{\mathcal G} - (H_{L^1}f)_{\me}(x) \varphi(f)_{\me}(x)dx \le 0.\]

\medskip

Clearly, the Laplacian with boundary conditions as defined in \eqref{eq:domain} does not have any eigenfunctions with positive eigenvalue on $L^1(\mathcal G)$ or $L^2(\mathcal G)$ (dissipativity/integration by parts). The heat semigroup $(T_p(t))$ introduced in Subsection \ref{sec:convsem} maps by Lemma \ref{eq:lemmbc} into $D(-\Delta_p)$, defined in \eqref{eq:domain}, and thus its generator $A_p$ is necessarily a restriction of $\Delta_p.$ 

Thus for $\lambda >0$ large we have that the invertible operator $A_1-\lambda$ is a restriction of the injective operator $\Delta_1-\lambda$ which implies that they both coincide $A_1=\Delta_1$. 
Thus, the heat semigroup is contractive on $L^1(\mathcal G)$ by the Lumer-Phillips theorem. This implies that $\Delta_1-\lambda$ is invertible for any $\lambda>0.$ 

Hence, by duality, there does not exist an eigenfunction $u \in D(\Delta_{BUC})$ satisfying $(\Delta_{BUC}-\lambda)u=0,$ as otherwise $\Delta_1-\lambda$ would not be surjective on $L^1(\mathcal G).$
We conclude from repeating the above calculation on spaces $L^p$ using the functional $\varphi_f(g) = \int_{\mathcal G} g(x) \operatorname{sgn}(f(x)) \vert f(x) \vert^{p-1}\Vert f \Vert^{2-p}_p \ dx$ that $\Delta_p$ is dissipative on all spaces $L^p(\mathcal G)$ with $p >1.$ 
Riesz--Thorin interpolation of the semigroup on $L^1(\mathcal G)$ and $L^2(\mathcal G)$ and duality implies then the contractivity of the semigroup on all spaces $L^p(\mathcal G)$ with $p \in [1,\infty)$ and spaces $C_0(\mathcal G), BUC(\mathcal G).$ 
We conclude, as above, that the generator of all these semigroups must be a restriction of the Laplacian we defined on these spaces in Section \ref{sec:DefNL}. 
The result for Schrödinger semigroups then follows from \cite[Theo.~III.2.7]{EngNag00}. \end{proof}

We remark that higher order parabolic semigroups studied in Theorem \ref{theo:theo1} do not necessarily have to be contractive which is obvious already for $\mathcal G=\RR$:

\begin{ex}[Non-contractivity of bi-Laplacian]
Consider the semigroup generated by the bi-Laplacian $-\Delta^2$ on $L^1(\mathbb R).$ In this case, it is well-known that the kernel $k_t$ associated with the semigroup $(e^{-\Delta^2 t})$ takes negative values on a set of positive measure whereas $\int_{\mathbb R} k_t =1.$ This implies that $\Vert k_t\Vert_{L^1}>1$ which implies $\sup_{\Vert u_0 \Vert_{L^1}=1}\lVert e^{-\Delta^2 t} u_0 \rVert_1 = \sup_{\Vert u_0 \Vert_{L^1}=1} \lVert k_t*u_0 \rVert_{1}= \lVert k_t \rVert_1>1$ that the associated semigroup on $L^1$ cannot be contractive.
\end{ex}

In fact, aside from this simple argument, it is known in much greater generality that derivatives of order higher than two do not yield in general contraction semigroups on spaces $L^p$ other than $p =2$, \cite[Theo $4.14$]{LANGER199973}. However, there is still a class of fourth order operators that generates contractive semigroups on $L^p$ for $p\in\left[\frac32,3\right] $ \cite[Theo $5.2$]{LANGER199973}; as also proved in \cite{GreMug19} in the case of the Friedrichs realizations of the bi-Laplacians on finite connected graphs through Nittka's characterization of $L^p$-contractivity.

We proceed with a study of positivity and the Feller property for semigroups:
%
A \textit{Banach lattice} is a Banach space $X$ such that for all $f,g\in X$ $\left\lvert f \right\rvert \le \left\lvert g \right\rvert$ implies $\left\lVert f \right\rVert \le \left\lVert g \right\rVert$.
Examples of such spaces include the spaces $X_c = L^p(\mathcal G)$, $C_0(\mathcal G),$ and $BUC(\mathcal G).$ A $C_0$-semigroup of operators $(T(t))$ on a Banach lattice is then called \textit{positive} if $f\ge 0$ implies that $T(t) f \ge 0$ for all $t \ge 0.$ 

\begin{prop}
\label{prop:pos}
The Schrödinger semigroup with $V \in L^{\infty}$ is positive on all spaces $L^p$ with $p \in [1, \infty)$ and spaces of continuous functions $C_0(\mathcal G)$ and $BUC(\mathcal G)$. In fact, the heat semigroup is even positive improving, i.e. a function $f \neq 0$ such that $f \ge 0$ gets mapped to a function $(T(t)f)(x)>0$ for all $t>0$ and $x \in \mathcal G.$
\end{prop}
\begin{proof}
We start with ${X_{\operatorname{c}}}=L^p(\mathcal G)$ and $p< \infty.$
Take $f \in L^p(\mathcal G)$ positive and some $f_n \in L^p(\mathcal G) \cap L^2(\mathcal G)$ positive such that $\lim_{n \rightarrow \infty} \left\lVert f_n -f \right\rVert_{L^p(\mathcal G)}=0.$ Then from standard operator theory $T_2(t)f_n \ge 0$ (Trotter--Kato theorem and Beurling-Deny criterion). On the other hand, $\lim_{n \rightarrow \infty} \left\lVert T_1(t)f_n- T_1(t)f \right\rVert_{L^1(\mathcal G)}=0$ by continuity of the semigroup. 
Standard measure theory yields the existence of a pointwise a.e. convergent subsequence such that $T_{1}(t)f \ge 0.$
Take $f \in BUC(\mathcal G)$ positive and any $g \in L^1(\mathcal G) $ positive, then the dual pairing is positive $(T_1(t)g,f)_{L^1(\mathcal G) \times L^{\infty}(\mathcal G)}\ge 0.$
This however also implies that for all $g \in L^1(\mathcal G)$ positive $(g,T_{BUC}(t)f)_{L^1(\mathcal G) \times L^{\infty}(\mathcal G)}\ge 0.$
If $T_{BUC}(t)f$ was not a positive function, then there was a small open interval $I$ on some edge, where this function is negative. Taking $g$ as the characteristic function of that interval $g=\chi_{I}$, we would obtain that $(g,T_{BUC}(t)f)_{L^1(\mathcal G) \times L^{\infty}(\mathcal G)}< 0$
which contradicts positivity above.

\medskip

Now, apart from this abstract argument for positivity, it can be directly seen from the heat kernel construction that the heat semigroup is positivity improving, as the heat kernel is strictly positive everywhere.
\end{proof} 

We immediately obtain the following corollary from Lemma \ref{techlemma} and Prop. \ref{prop:genercontr}.
\begin{corr}
\label{corr:Feller}
The free Schrödinger semigroup $(T_{\operatorname{C_0}}(t))$ on $C_0(\mathcal G)$ is a Feller semigroup. It also has the strong Feller property.
\end{corr}

Although it is well-known that the biharmonic heat kernel is non-positive, we can still obtain the following positivity result for the biharmonic equation if the conductivities are summable

\begin{prop}
We consider an equilateral graph that satisfies $c(\mathcal G)<\infty$. Then the semigroup $e^{-t(-\Delta)^m}$ is uniformly asymptotically positive on $L^2(\mathcal G)$, i.e., for every $\varepsilon>0$ there is $t_0 \ge 0$ such that 
\[d \left(E_{+},e^{t(-(-\Delta)^m- \inf (\sigma(-(-\Delta)^m)))}f\right) \le \varepsilon \Vert f \Vert \text{ for all }t \ge t_0\text{ and all }f \in E_{+},\] where $E_{+}:=\{u \in L^2(\mathcal G); u \ge 0 \}.$ 
\end{prop}
\begin{proof}
Using results from subsequent Sections (Prop. \ref{prop:compactness}, Cor. \ref{pointspec}, Theorem~\ref{theo:spec}) it follows that the spectral bound $\inf (\sigma(-(-\Delta)^m))$ is isolated in the spectrum, and corresponds to a finite eigenvalue. Thus, \cite[Theo $8.3$]{PosSem} implies then the claim.
\end{proof}

\begin{theo}\label{theo:inc}
Let $\mathcal G$ be a metric graph that satisfies Assumption \ref{ass2}. Schrödinger semigroups $T_q(t)=e^{-t(-\Delta_q+V)}$ with relatively ($-\Delta_q$)-bounded potential, with relative bound $<1$, are analytic on spaces $X_{\operatorname{c}}=L^p$ and spaces $C_0(\mathcal G)$, $BUC(\mathcal G).$ Moreover, the higher order Laplacians $-(-\Delta_q)^m$, with $m \ge 1,$ generate analytic semigroups on all these spaces as well. In the Hilbert space case $L^2(\mathcal G)$, the previous results hold also without imposing Assumption \ref{ass2}.
\end{theo}
\begin{proof}
On $L^2(\mathcal G)$, self-adjointness of the generator implies immediately the claim.

On all other spaces, it suffices to analyze operators $-(-\Delta_q)^m$, only. The analyticity of the Schrödinger operator then follows from perturbation theory, cf. \cite[Ch.2 Cor.$2.14$]{EngNag00}.

We start by showing that generators $A_q$, of semigroups $(T_q(t))$ are indeed given by $-(-\Delta_q)^m$.
Let $q\le 2$ first, then for integers $j < m,$ and $\alpha:=\frac{(j-1)q+2}{mq}$ satisfying $p^{-1}=1-q^{-1}=j - m\alpha + q^{-1}$, we find by Gagliardo--Nirenberg interpolation and the uniform boundedness of edge-lengths and conductivities from above and below that there exist universal constants $C_1,C_2>0$ such that 
\[ \Vert u^{(j)} \Vert_{p} \le C_1 \Vert u^{{(m)}} \Vert_{q}^{\alpha} \Vert u \Vert_{q}^{1-\alpha}+ C_2 \Vert u \Vert_{q}.\]
This implies that if there exists a solution $-(-\Delta_q)^{m} u = \lambda u$, then 
\[ \lambda \langle u , u \rangle_{L^p \times L^q}= -\langle (-\Delta)^{m}_q u , u \rangle_{L^p \times L^q} = -\langle u^{(m)} , u^{(m)} \rangle_{L^p \times L^q} \le 0, \]
which implies $\lambda \le 0.$ This implies that $-(-\Delta_q)^m-\lambda$ is injective for $\lambda>0.$ Since the semigroup maps into $D(-(-\Delta_q)^{m})$ it follows that the generator $A_q$ satisfies $A_q \subset -(-\Delta_q)^m.$ Hence, for $\lambda >0$ large enough we have $A_q-\lambda = -(-\Delta_q)^{m}-\lambda$ such that $-(-\Delta_q)^{m}-\lambda$ is invertible, too. This implies also that operators $-(-\Delta_p)^{m}-\lambda$ with $p \in [2,\infty)$ as well as $-(-\Delta_{C_0})^{m}-\lambda$ and $-(-\Delta_{BUC})^{m}-\lambda$ are injective (duality). Hence, we conclude that the generators are indeed given by $A_p= -(-\Delta_p)^{m}$, $A_{C_0}= -(-\Delta_{C_0})^{m}$, and $A_{BUC}=-(-\Delta_{BUC})^{m}.$

To show analyticity, recall that a closed, densely defined linear operator $A_q$ generates an analytic semigroup if and only if there exists a half-plane $\Re(\lambda)>\omega$ contained in $\rho(A_q)$ such that $\Vert (\lambda-A_q)^{-1} \Vert \lesssim \vert \lambda-\omega \vert^{-1}.$
All the above operators are closed and densely defined by Lemma \ref{lemm:closed}.

First, we can conclude from spectral equivalence, Corollary \ref{corr:specequiv}, that the resolvent set contains a sector and satisfies, by the Combes--Thomas estimate, Proposition \ref{prop:CTE}, and consistency of resolvents, the estimate for $A=-(-\Delta)^m$ on $L^p(\mathcal G)$ and $\Re(\lambda)>\omega=\varepsilon$
\[\Vert (\lambda-A_q)^{-1} \Vert \lesssim \vert \lambda-\omega \vert^{-1}.\]

For spaces $C_0(\mathcal G)$ and $BUC(\mathcal G)$ it suffices to note that, using Prop. \ref{semigroupapproach}
\begin{equation*}
\begin{split}
 \Vert (\lambda-A)^{-1} \Vert_{\mathcal L(C_0)} \le \Vert (\lambda-A)^{-1} \Vert_{\mathcal L(BUC)} &\le \Vert (\lambda-A)^{-1} \Vert_{\mathcal L(L^{\infty})}= \Vert (\lambda-A)^{-1} \Vert_{\mathcal L(L^1)}.
 \end{split}
 \end{equation*}
 This concludes the proof.
\end{proof}

Immediately from \cite[Cor.~III.3.12]{EngNag00}, we obtain the following.
\begin{corr}
\label{corr:SMT}
Schrödinger semigroups satisfy a spectral mapping theorem $e^{-t \sigma(H)} = \sigma(T(t)) \backslash \{0\}, \ t \ge 0$
and the spectral bound and growth bound of the semigroups $T(t)$ coincide.
\end{corr}

\subsection{Asymptotic properties}

Recall that a $C_0$-semigroup $(T(t))$ on ${X_{\operatorname{c}}}$ is called strongly stable if for all $x \in {X_{\operatorname{c}}}$ we have $\lim_{t \rightarrow \infty} \left\lVert T(t)x \right\rVert=0$, uniformly stable if $\lim_{t \rightarrow \infty} \left\lVert T(t) \right\rVert=0,$
and uniformly exponentially stable, if there is $\varepsilon>0$ such that $
\lim_{t \rightarrow \infty}e^{\varepsilon t} \left\lVert T(t) \right\rVert=0.$

From our previous discussions we immediately conclude that

\begin{corr}
\label{corr:asymptotics}
Let Assumption \ref{ass2} be satisfied. The semigroup generated by $-(-\Delta)^m$ is never uniformly stable on any of the spaces ${X_{\operatorname{c}}}=L^p(\mathcal G), C_0(\mathcal G), BUC(\mathcal G).$ 
\end{corr}
\begin{proof}
It follows from \cite[Prop.~III.1.7]{EngNag00} that uniform stability is equivalent to exponential stability. To have exponential stability the generator needs to possess a spectral gap. This is however not the case as $0 \in \sigma_p(\Delta_{BUC})$ which follows from $\Delta_{BUC} \indic=0.$
\end{proof}
Because $T(t)\indic=\indic$, on the space of bounded uniformly continuous functions the heat semigroup is not strongly stable; on all other spaces it may or may not be so. 

If we drop the assumption of sub-exponential growth on our graph, the heat semigroup can indeed become exponentially stable: 
\begin{ex}
Let $\mathcal G$ be the homogeneous equilateral tree of degree $q \ge 3$ with unitary conductivities. The spectrum of $\Delta$ on $L^2(\mathcal G)$ is given by \cite[Prop.$2$]{Cat97}
\begin{equation*}
\begin{split}
\sigma_p(\Delta) &= \{-k^2\pi^2; k \in \mathbb N\}\text{ and } \\
\sigma_c(\Delta)&=\left\{- \lambda \in (-\infty,0]: \cos(\sqrt{\lambda}) \in \left[ -\frac{2\sqrt{q-1}}{q},\frac{2\sqrt{q-1}}{q}\right] \right\}.
  \end{split}
\end{equation*} 
In particular, $\Delta$ has a spectral gap at zero. This implies that for some $M>0$ and $\omega>0$ we have exponential decay $\Vert T(t) \Vert \le M e^{-\omega t}.$ 
\end{ex}

\section{Equilateral graphs: Point spectra of Schrödinger operators}
In the following two sections, we assume the graph $\mathcal G$ to be equilateral.

Since we assume the graph $\mathcal G$ to be equilateral, this condition clearly implies that there is $\kappa \ge 1$ such that for all adjacent edges we have $c(\me') \le \kappa c(\me) \text{ if } \me'\cap \me \neq \emptyset.$
\label{sec:PS}
In this section, we characterize the point spectrum of Schrödinger operators 
\begin{equation*}
H\psi:=-\Delta \psi+V\psi, 
\end{equation*}
with relatively $-\Delta$ bounded potential $V=(V_{\me}) \in L^1_{\operatorname{loc}}$ that is the same on every edge $\me$ with relative bound $<1$, on \eqref{eq:domain}.

\medskip

Let $H^D$ be the Schrödinger operator with Dirichlet boundary conditions
\begin{equation*}
\begin{split}
D(H^D)&:=\left\{ f \in C^2([0,1]); f(0)=f(1)=0 \right\}, \quad H^D f:=Hf.
  \end{split}
\end{equation*}
The spectrum of this operator is discrete.
We also define the operator $H^D_{X_c}$ acting on $X_{\operatorname{c}}$ with Dirichlet boundary conditions on every edge.
There are always two linearly independent solutions $c_{\lambda},s_{\lambda} \in C^{\infty}[0,1]$ to the classical equation $-\Delta\psi+V_\me\psi= \lambda \psi$
such that $c_{\lambda}(0) =1, s_{\lambda}(0)=0$ and $c_{\lambda}'(0) =0, s_{\lambda}'(0)=1.$ \newline
\begin{ex}
For zero potential $V \equiv 0$, the solutions are explicitly given by 
\begin{equation}\label{eq:cs}
c_{\lambda}(t):=\cos(\sqrt{\lambda}t)\text{ and }s_{\lambda}(t):=\frac{\sin(\sqrt{\lambda}t)}{\sqrt{\lambda}}.
\end{equation}
\end{ex}

Let $\lambda \notin \sigma(H^D),$ a function satisfying $H \psi_{\lambda}= \lambda \psi_{\lambda}$ can then be expressed in terms of $s_{\lambda}$ and $c_{\lambda}$ on edges $\me$ as
\begin{equation*}
\psi_{\lambda} (t) = \frac{\psi_{\lambda}(1)-\psi_{\lambda}(0)c_{\lambda}(1)}{s_{\lambda}(1)}s_{\lambda}(t)+\psi_{\lambda}(0) c_{\lambda}(t).
\end{equation*}
Consider the operator $H^{\operatorname{ext}}_{X_{\operatorname{c}}}: D(H^{\operatorname{ext}}_{X_{\operatorname{c}}}) \subset {X_{\operatorname{c}}} \rightarrow {X_{\operatorname{c}}}$ defined by
\begin{equation*}
\begin{split}
D(H^{\operatorname{ext}}_{X_{\operatorname{c}}})&:=\left\{ 
f \in {X_{\operatorname{c}}} : \ f \in C(\mathcal G) \text{ and }H f \in {X_{\operatorname{c}}} \right\},\ H^{\operatorname{ext}}_{X_{\operatorname{c}}}\psi:=H\psi.
\end{split}
\end{equation*}
We stress that no Kirchhoff conditions at $0,1$ are imposed on the elements of $D(H^{\operatorname{ext}}_{X_{\operatorname{c}}})$.

\medskip 

As is well-known in the Hilbert space setting, there is a close connection between the discrete Laplacian on the underlying lattice of the graph and the continuous Laplacian on the metric graph. Our aim is to extend this idea to the function spaces we study in this article. 

To extend this idea to also other spaces, we want to introduce the following correspondences between functions spaces on combinatorial (the vertex set) and metric graphs, $\mathcal G$ that we use frequently throughout the article.
\begin{center}
\begin{table}[h]
\begin{tabular}{ r r} 
\hline
Space on vertex set ${X_{\operatorname{d}}}$& Space on metric graph ${X_{\operatorname{c}}}$ \\[1ex] \hline\hline
$c_0 (\mathcal V(\mathcal G))$ & $C_0(\mathcal G)$\\[1ex]
 $\ell^p(\mathcal V(\mathcal G)), p \in [1,\infty)$ &$ L^p(\mathcal G), p \in [1,\infty)$ \\[1ex] 
 $\ell^{\infty}(\mathcal V(\mathcal G))$ & $BUC(\mathcal G)$\\[1ex] 
\hline
\end{tabular}\\[5pt]
	\caption{Association of spaces on vertex set ${X_{\operatorname{d}}}$ and metric graphs ${X_{\operatorname{c}}}$.} \label{tab:XY}
\end{table}
\end{center}

We then have the following Lemma linking spaces on vertex sets to spaces on the metric graph:

\begin{lemm}
The vertex map $\pi : D(H^{\operatorname{ext}}_{X_{\operatorname{c}}}) \rightarrow {X_{\operatorname{d}}}$ defined by $\pi(\psi)(\mv):=\psi(\mv)$ restricts to an
isomorphism 
\begin{equation*}
\pi \vert_{\operatorname{ker}(H^{\operatorname{ext}}_{X_{\operatorname{c}}}-\lambda)}:\operatorname{ker}(H^{\operatorname{ext}}_{X_{\operatorname{c}}}-\lambda) \rightarrow {X_{\operatorname{d}}}
\end{equation*} for any $\lambda \in \rho(H^D).$

Its inverse is the gamma field $\gamma_{{X_{\operatorname{d}}}}: \rho(H^D) \rightarrow \mathcal L({X_{\operatorname{d}}},{X_{\operatorname{c}}}^{(2)})$, with ${X_{\operatorname{c}}}^{(2)}$ defined in \eqref{eq:spaces}, satisfying 
\begin{equation}
\begin{split}
\label{eq:algebraiceq}
&\pi\vert_{\operatorname{ker}(H^{\operatorname{ext}}_{X_{\operatorname{c}}}-\lambda)} \circ \gamma_{X_{\operatorname{d}}}(\lambda)=\operatorname{id}_{{X_{\operatorname{d}}}} \text{ and } \gamma_{X_{\operatorname{d}}}(\lambda) \circ \pi\vert_{\operatorname{ker}(H^{\operatorname{ext}}_{X_{\operatorname{c}}}-\lambda)}=\operatorname{id}_{\operatorname{ker}(H^{\operatorname{ext}}_{X_{\operatorname{c}}}-\lambda)}
\end{split}
\end{equation} on edges $\me \in \mathcal{E}(\mathcal G)$ and is given by
\begin{equation}
\label{eq:gamma}
(\gamma_{X_{\operatorname{d}}}(\lambda)z)_\me(t)= \frac{s_{\lambda}(t)z(t(\me))+z(i(\me))(s_{\lambda}(1)c_{\lambda}(t)-c_{\lambda}(1)s_{\lambda}(t))}{s_{\lambda}(1)}.\nonumber
\end{equation}
\end{lemm}
\begin{proof}
The algebraic relations \eqref{eq:algebraiceq} are immediate.
Also, the functions $(\gamma_{X_{\operatorname{d}}}(\lambda)z)$ and $H_{{X_{\operatorname{c}}}}(\mathcal \gamma_{X_{\operatorname{d}}}(\lambda)z)$ are in ${X_{\operatorname{c}}}$ since functions $c_{\lambda}$ and $s_{\lambda}$ are smooth. 
Using that $c_{\lambda}$ satisfies $H_{X_{\operatorname{c}}}c_{\lambda}=\lambda c_{\lambda}$, and similarly for $s_{\lambda}$, this implies 
\begin{equation}
\begin{split}
 \left\lVert \gamma_{X_{\operatorname{d}}}(\lambda)z \right\rVert_{{X_{\operatorname{c}}}^{(2)}} 
 &= \left\lVert H_{X_{\operatorname{c}}} \gamma_{X_{\operatorname{d}}}(\lambda)z \right\rVert_{X_{\operatorname{c}}} + \left\lVert \gamma_{X_{\operatorname{d}}}(\lambda)z \right\rVert_{X_{\operatorname{c}}} \\
 &\lesssim (\vert \lambda \vert+1) \left\lVert \gamma_{X_{\operatorname{d}}}(\lambda)z \right\rVert_{X_{\operatorname{c}}} \lesssim_{\lambda} \left\lVert z \right\rVert_{{X_{\operatorname{d}}}}
\end{split}
\end{equation}
since on every edge $\me$ there is a constant $C_{\lambda}>0$ independent of $\me$ such that 
\[ \left\lvert (\gamma_{X_{\operatorname{d}}}(\lambda)z)_\me(t) \right\rvert \le C_{\lambda} (\vert z(i(\me)) \vert+ \vert z(t(\me)) \vert).\]
Since the two maps $t \mapsto \tfrac{s_{\lambda}(t)}{s_{\lambda}(1)}$ and $t \mapsto \frac{s_{\lambda}(1)c_{\lambda}(t)-c_{\lambda}(1)s_{\lambda}(t)}{s_{\lambda}(1)}$ are continuous and vanish / attain the value one at $t=0$ and $t=1$, respectively, we have that for some sufficiently small $\varepsilon>0$, independent of $\me$,
\[ \vert z(i(\me)) \vert \le C_{\lambda,p,\varepsilon} \Vert (\gamma_{X_{\operatorname{d}}}(\lambda)z)_\me \Vert_{L^p(0,\varepsilon)} \text{ and } \vert z(t(\me)) \vert \le C_{\lambda,p,\varepsilon} \Vert (\gamma_{X_{\operatorname{d}}}(\lambda)z)_\me \Vert_{L^p(1-\varepsilon,1)}.\]
From this, we get that $\gamma_{X_{\operatorname{d}}}(\lambda) \in \mathcal L({X_{\operatorname{d}}},{X_{\operatorname{c}}}^{(2)})$.
\end{proof}
We also consider the map $\mu_{X_{\operatorname{d}}} \in \mathcal L({X_{\operatorname{c}}},{X_{\operatorname{d}}})$ defined as 
\begin{equation*}
\begin{split}
(\mu_{X_{\operatorname{d}}}(\lambda)f)(i(\me))&= \frac{ \langle (s_{\lambda}(1)c_{\lambda}-c_{\lambda}(1)s_{\lambda}),f_\me \rangle}{s_{\lambda}(1)} \text{ and } 
(\mu_{X_{\operatorname{d}}}(\lambda)f)(t(\me))= \frac{\langle s_{\lambda},f_\me \rangle}{s_{\lambda}(1)}.
\end{split}
\end{equation*}
In the Hilbert space case ${X_{\operatorname{c}}}=L^2(\mathcal G)$ and ${X_{\operatorname{d}}}=\ell^2(\mathcal V)$, we have $\mu_{\ell^2}(\lambda)=\gamma_{\ell^2}(\lambda)^*.$ 

The next Lemma provides us with an isomorphism, expressed in terms of the Floquet discriminant, \cite[279 pp.]{ReeSim78}, $D(\lambda):=\tfrac{1}{2} (c_{\lambda}(1)+s_{\lambda}'(1))$ on $\operatorname{ker}(H_{X_{\operatorname{c}}}-\lambda)$ rather than $\operatorname{ker}(H^{\operatorname{ext}}_{X_{\operatorname{c}}}-\lambda)\supset \operatorname{ker}(H_{X_{\operatorname{c}}}-\lambda)$.

\begin{lemm}
For $\lambda \in \rho(H^D_{X_{\operatorname{c}}})$ the operator $M(\lambda)_{X_{\operatorname{d}}} \in \mathcal L({X_{\operatorname{d}}})$, with Floquet discriminant $D(\lambda):=\tfrac{1}{2} (c_{\lambda}(1)+s_{\lambda}'(1))$, is given by
\begin{equation}
\label{eq:Mlambda}
(M(\lambda)_{X_{\operatorname{d}}}z)(\mv)= \sum_{\me \in \mathcal{E}_{\mv}(\mathcal G)}c(\me)\left(\frac{z\left(t(\me)1_{\mv \in i(\mathcal{E}(\mathcal G))}+i(\me)1_{\mv \in t(\mathcal{E}(\mathcal G))}\right)-D(\lambda)}{s_{\lambda}(1)} \right) \nonumber
\end{equation}
and has the property that $\gamma_{X_{\operatorname{d}}}(\lambda)\vert_{\operatorname{ker}(M(\lambda)_{X_{\operatorname{d}}})}: \operatorname{ker}(M(\lambda)_{X_{\operatorname{d}}}) \rightarrow \operatorname{ker}\left(H_{X_{\operatorname{c}}}-\lambda\right)$ 
is an isomorphism.

\end{lemm}

\begin{proof}
We only need to check the Kirchhoff condition on the derivatives. 
We first observe that the Wronskian $W(s_{\lambda},c_{\lambda})(t):=s_{\lambda}'(t)c_{\lambda}(t)-s_{\lambda}(t)c_{\lambda}'(t)$
satisfies $W(s_{\lambda},c_{\lambda})(0)=1$ from the initial conditions and since the Wronskian is constant we have $W(s_{\lambda},c_{\lambda})(t)=1 \text{ for all }t \in [0,1].$

This allows us to introduce the Dirichlet-to-Neumann map by differentiating \eqref{eq:gamma}
\begin{equation}
\label{DtN}
m(\lambda) := \frac{1}{s_{\lambda}(1)} \left( \begin{matrix} -c_{\lambda}(1) & 1 \\ 1 & -s_{\lambda}'(1) \end{matrix} \right)
\end{equation}
such that
\begin{equation*}
\left( \begin{matrix}(\gamma_{X_{\operatorname{d}}}(\lambda)z)'_\me(0) \\ -(\gamma_{X_{\operatorname{d}}}(\lambda)z)'_\me(1) \end{matrix} \right) = m(\lambda) \left( \begin{matrix} z_\me(0) \\ z_\me(1) \end{matrix} \right).
\end{equation*}
We therefore obtain for the derivatives at vertices $\mv \in \mathcal V(\mathcal G)$
\begin{equation*}
0=\sum_{\me \in \mathcal{E}_\mv(\mathcal G)} \frac{\partial (\gamma_{X_{\operatorname{d}}}(\lambda)z)_\me}{\partial n_\me}(\mv)=(M(\lambda)_{X_{\operatorname{d}}}z)(\mv).
\end{equation*}
This implies immediately that $0 \in \sigma_p(M(\lambda)_{X_{\operatorname{d}}})$ iff $\lambda \in \sigma_p(H_{X_{\operatorname{c}}}).$
\end{proof}
By rewriting the operator $M(\lambda)_{X_{\operatorname{d}}}$ we obtain the following corollary.
\begin{corr}
\label{pointspec}
For $\lambda \in \rho(H^D)$ we have that $\lambda \in \sigma_p(H_{{X_{\operatorname{c}}}})$ if and only if the Floquet discriminant $D(\lambda)$, satisfies $D(\lambda) \in \sigma_p(P_{{X_{\operatorname{d}}}})$ where 
\begin{equation}
\label{eq:PX}
(P_{{X_{\operatorname{d}}}}z)(\mv):= \sum_{\me \in \mathcal{E}_{\mv}(\mathcal G)} \frac{c(\me)}{c(\mv)} z\left(t(\me)1_{\mv \in i(\mathcal{E}(\mathcal G))}+i(\me) 1_{\mv \in t(\mathcal{E}(\mathcal G))}\right) \text{ on ${X_{\operatorname{d}}}$.}
\end{equation}
\end{corr}

We then have that 
\begin{prop}
\label{prop:compactness}
If $c(\mathcal G) < \infty$, then the discrete Laplacian is a compact operator on all spaces $\ell^p(\mathcal V(\mathcal G))$ with $p \in (1,\infty).$
\end{prop}
\begin{proof}
The following estimate shows the compactness on $\ell^p$ with $1/p+1/q=1$ and $p\in (1,\infty)$, for which we fix some vertex $o \in \mathcal V$,
\begin{equation*}
\begin{split}
&\Vert \indic_{B(o,r)^c} P z \Vert^p_p = \sum_{\mv \in \mathcal V(\mathcal G \backslash B(o,r))} c(\mv)^{1-p} \left\lvert \sum_{\me \in \mathcal{E}_{\mv}(\mathcal G)}c(\me) z\left(t(\me)1_{\mv \in i(\mathcal{E}(\mathcal G))}+i(\me) 1_{\mv \in t(\mathcal{E}(\mathcal G))}\right) \right\rvert^p \\
&\le \sum_{\mv \in \mathcal V(\mathcal G \backslash B(o,r))} c(\mv)^{1-p}\left( \sum_{\me;i(\me)=\mv}c(\me) \left\lvert z(t(\me)) \right\rvert^p +\sum_{\me;t(\me)=\mv}c(\me) \left\lvert z(i(\me)) \right\rvert^p\right)\left( \sum_{\me \in \mathcal{E}_{\mv}(\mathcal G)}c(\me)^q \right)^{p/q} \\
&\le \sum_{\mv \in \mathcal V(\mathcal G \backslash B(o,r))} c(\mv) \left( \sum_{\me;i(\me)=\mv}c(\me) \left\lvert z(t(\me)) \right\rvert^p +\sum_{\me;t(\me)=\mv}c(\me) \left\lvert z(i(\me)) \right\rvert^p\right) \xrightarrow[r \to \infty]{}0.
\end{split}
\end{equation*}
Hence, the operator $P$ is approximated by finite-rank operators $ \indic_{B(o,r)} P$ which implies the claim. 
This concludes the proof.
\end{proof}
This condition however, is not sufficient to obtain compactness on $c_0,$ $\ell^1,$ and $\ell^{\infty}$ as the following example shows:
\begin{ex}
To see that the operator is not compact on $\ell^{\infty}(\ZZ)$, and thus by Schauder's theorem also on $c_0$ and $\ell^1$, let us choose conductivities $c(n,n+1)=\frac{1}{\vert n \vert^2+1}$. Then 
\[ \frac{c(n,n+1)}{c(n)} = \frac{n^2-2n+2}{2n^2-2n+3}. \]
It is straightforward to verify that with this choice of conductivities $P$ is not compact, since the image of the standard basis $(e_n)_{n \in \mathbb N}$ under $P$ does not have a convergent subsequence.
\end{ex}

It follows from $\ell^p \subset \ell^2$, assuming still that $c(\mathcal G)< \infty,$ for $p \in (2,\infty)$ from standard properties, see e.g. \cite[Theo. $3$]{Symm}, that $\ker(P_{\ell^2}-\lambda)= \ker(P_{\ell^p}-\lambda)$ for all $\lambda \in \CC$ and $p\in (2,\infty).$

Let us now state a lower growth bound on solutions $\varphi$ to the equation $H_{X_{\operatorname{c}}}\varphi = \lambda \varphi$ for $\lambda$ such that $D(\lambda)$ is non-real under a suitable growth condition on the graph.
\begin{prop}
Fix an element $o \in \mathcal V(\mathcal G)$ and consider a nested partition $\mathcal G_n:=\left\{ x \in \mathcal G; \vert d(x,o) \vert \le n \right\}$ such that $\bigcup_{n} \mathcal G_n = \mathcal G.$ Moreover, assume that the graph satisfies the growth condition
\begin{equation*}
\limsup_{n \rightarrow \infty} \frac{\left\lvert \mathcal{V}\left(\mathcal G_n\right) \right\rvert}{ c^n} =0
\end{equation*}
for any $c>1.$ Then, the point spectrum of $H_{X_{\operatorname{c}}}$ satisfies $\sigma_p(H_{X_{\operatorname{c}}}) \subset \left\{ \lambda \in \mathbb C; D(\lambda) \in \mathbb R \right\}.$ In particular, the $\ell^2$ mass of any solution in ${X_{\operatorname{d}}}$ to the above equation satisfies for some $c>0$ that $ \langle \indic_{\mathcal G_n} z,z \rangle_{\ell^2(\mathcal V(\mathcal G))} \ge c r^n,$ where $r$ is the unique solution in $(1,\infty)$ to the equation $r^2-1= r \frac{\left\lvert \Im(D(\lambda)) \right\rvert}{d_v^*}.$
\end{prop}
\begin{proof}
Let $\lambda \in \mathbb{C}\backslash \mathbb{R}$ and assume there was an eigenfunction $\varphi \in D(\Delta_{X_{\operatorname{c}}})$ such that
\begin{equation*}
(-\Delta_{X_{\operatorname{c}}}+V) \varphi=\lambda \varphi.
\end{equation*} 
This implies by Corollary \eqref{pointspec} the existence of an eigenfunction $z \in {X_{\operatorname{d}}}$ such that
\begin{equation*}
\label{eq:EVpro}
P_{X_{\operatorname{d}}}z = D(\lambda) z.
\end{equation*}
We therefore conclude that
\begin{equation*}
\begin{split}
S_n&:=\langle \indic_{\mathcal G_n}P_{X_{\operatorname{d}}}z,z \rangle_{\ell^2(\mathcal V(\mathcal G))}
=\sum_{v \in \mathcal{V}\left(\mathcal G_{n}\right)}c(\mv)(P_{X_{\operatorname{d}}}z)(\mv) \overline{z(\mv)}\\
&=\langle \indic_{\mathcal G_n}D(\lambda)z,z \rangle_{\ell^2(\mathcal V(\mathcal G))}=D(\lambda) \sum_{v \in \mathcal{V}\left(\mathcal G_{n}\right)} c(\mv)\left\lvert z(\mv) \right\rvert^2.
\end{split}
\end{equation*}
We observe that any edge that is only partially contained in $\mathcal G_n$ is fully contained in $\Omega_{n+1}.$
Taking the absolute value of the imaginary part, we obtain the estimate 
\begin{equation}
\begin{split}
\left\lvert \Im S_n \right\rvert &= \left\lvert \Im(D(\lambda)) \right\rvert \sum_{v \in \mathcal{V}\left(\mathcal G_n\right) }c(\mv) \left\lvert z(\mv) \right\rvert^2 \nonumber\\
&=\left\lvert \Im\left(\sum_{\mv \in \mathcal{V}\left(\mathcal G_{n} \right)}c(\mv)(P_{X_{\operatorname{d}}}z)(\mv) \overline{z(\mv)} \right)\right\rvert \le d_\mv^* \sum_{\mv \in \mathcal{V}\left(\mathcal G_{n+1}\backslash \mathcal G_{n-1}\right)}c(\mv) \left\lvert z(\mv) \right\rvert^2 
\end{split}
\end{equation}
which by introducing $s_n:=\langle \indic_{\mathcal G_n} z,z \rangle_{\ell^2(\mathcal V(\mathcal G))}>0$ means that $s_{n+1}-s_{n-1}\ge \frac{\left\lvert \Im(D(\lambda)) \right\rvert}{d_v^*} s_n.$
Without loss of generality we may assume $s_{1} >0$ otherwise we translate the origin of the graph.
Then, since all $s_{n}>0$ we obtain a lower bound on $s_n$ by positive $\delta_n$ satisfying $\delta_{n+1}-\delta_{n-1}= \frac{\left\lvert \Im(D(\lambda)) \right\rvert}{d_v^*} \delta_n$ with $\delta_1:=s_1.$
Solutions to this three term recurrence relation are, in terms of two constants $\lambda_1, \lambda_2 \in \mathbb{R}$, given by $\delta_n = \lambda_1 r_1^n+ \lambda_2 r_2^n,$ where $r_1 \in (-1,0)$ and $r_2 \in (1, \infty)$ are solutions to the equation $x^2-1= \left\lvert \Im(\Delta(\lambda)) \right\rvert/d_v^* x.$ In particular $\lambda_2$ is strictly positive.
Thus, for large $n$ we see that $r_2^n\lesssim s_n$ where $r_2>1$.
This contradicts the boundedness of eigenfunctions in all spaces that we consider. 
\end{proof}
Let us record an immediate fact about the spectrum.
\begin{prop}
\label{zerospec}
$0 \in \sigma_p(\Delta_{BUC})$ is a simple eigenvalue for any infinite graph.
\end{prop}
\begin{proof}
$\Delta_{BUC} 1 = 0$. To prove the second assertion, observe that each element of the null space of $\Delta_{BUC}$ is necessarily a piecewise affine function, and indeed a constant function on each connected component, as one sees enforcing boundedness. The claim now follows in view of our standing assumption of connectedness of $\mathcal G$.
\end{proof}

Results on the Dirichlet spectrum highly depend on geometric properties of the graph results for various cases can be found in \cite[Theo. $2$]{Cat97}. In addition, we have the following immediate result:
\begin{lemm}
If the graph $\mathcal G$ is an equilateral tree, then the Dirichlet spectrum belongs to $\sigma_p(\Delta_{BUC}).$
\end{lemm}
\begin{proof}
Since functions $e^{\pm i \lambda x}$ are bounded, this just follows from iteratively adding new fragments of such functions, with the correct boundary conditions, along the tree.
\end{proof}

\section{Equilateral graphs: Continuous spectrum of the Kirchhoff-Laplacian}
\label{sec:CS}
In this section we are still assuming $\mathcal G$ to be equilateral and make the following growth assumption: 
\begin{Assumption}[Growth condition-Equilateral graphs]
\label{ass1}
We assume that $\mathcal G$ is an equilateral graph such that for all $\varepsilon>0$ there is $C_{\varepsilon}>0$ such that for all $\me \in \mathcal E(\mathcal G)$ 
\[c(B_r(x) \supset \mathcal E( \mathcal G)) \le C_{\varepsilon} e^{\varepsilon r} c(\me), \quad x=i(\me),\]
where $B_r(x) \supset \mathcal E( \mathcal G)$ is the set of edges fully contained in $B_r(x).$ Here, we write for $A \subset \mathcal E(\mathcal G)$ the conductivity $c(A) := \sum_{\me \in A} c(\me).$
\end{Assumption}

\medskip

Given $\lambda \in \rho(\Delta^D)$ and $g=(g_\me) \in {X_{\operatorname{c}}},$ we introduce the function $u \in D(\Delta_{{X_{\operatorname{c}}}})$ defined on edges $\me \in \mathcal E(\mathcal G)$ by
\begin{equation}
\begin{split}
\label{SturmLiouville}
u_\me(x):=&\frac{s_{\lambda}(x)}{s_{\lambda}(1)}\left(w(t(\me))- \int_x^1 g_\me(y) s_{\lambda}(1-y) \ dy\right) \nonumber \\
&+\frac{s_{\lambda}(1-x)}{s_{\lambda}(1)}\left(w(i(\me))-\int_0^x g_\me(y)s_{\lambda}(y) \ dy \right),
\end{split}
\end{equation}
where $s_\lambda$ is defined as in~\eqref{eq:cs};
$u$ satisfies the equation $(\Delta_{X_{\operatorname{c}}}-\lambda)u=g$
for any family of coefficients $(w(\mv))_{\mv \in \mathcal V(\mathcal G)}$ in ${X_{\operatorname{d}}}$.

Consider then functions $ f_{1,\lambda}(x):=s_{\lambda}(1-x)$ and $f_{2,\lambda}(x):=s_{\lambda}(x).$

These two functions are linearly independent on $(0,1)$ for $\lambda \notin \rho(\Delta^D)$ since 
\[f_{1,\lambda}(0)=s_{\lambda}(1)\neq 0=s_{\lambda}(0)=f_{2,\lambda}(0).\]

\begin{lemm}
\label{lem:auxi}
There exists a function $g \in C[0,1]$ such that $g(0)=g(1)=0$ and
\[\langle g_{\lambda},f_{1,\lambda} \rangle_{L^2(0,1)} =0\text{ and }\langle g_{\lambda},f_{2,\lambda} \rangle_{L^2(0,1)} =1.\]
\end{lemm}
\begin{proof}
Consider the measure $d\mu(x) = x(1-x) \ dx$ on $(0,1).$
There is by Gram--Schmidt orthonormalization a function $g_{\lambda}(x):=\widetilde{g_{\lambda}}(x)x(1-x)$ such that $\langle \widetilde{g_{\lambda}},f_{1,\lambda} \rangle_{L^2(d\mu)} =0$ and $\langle g_{\lambda},f_{2,\lambda} \rangle_{L^2(d\mu)} =1.$ If $g$ would not vanish at the end-points e.g.~at $x=0$, then $\vert \widetilde{g_{\lambda}}(x)\vert \ge c/x$ for some $c>0$ and thus $\widetilde{g_{\lambda}} \notin L^2((0,1),d\mu).$
\end{proof}

\begin{lemm}
\label{existg}
Let $W_{\lambda} \in {X_{\operatorname{d}}}$ be given, then there is a function $y \in {X_{\operatorname{c}}}$, that vanishes at the vertices, such that 
\begin{equation*}
\label{s-function}
W_{\lambda}(\mv)= \sum_{\me \in \mathcal{E}_\mv} \frac{c(\me)}{c(\mv)} \int_0^1 y_\me(x)f_{I(\me),\lambda}(x) \ dx
\end{equation*}
where $I(\me)=1$ if $i(\me)=\mv$ and $I(\me)=2$ otherwise.
\end{lemm}
\begin{proof}
Let $T$ be a spanning tree of $\mathcal G$, i.e., a new metric graph that shares all edges with $\mathcal G$ but such that some of the vertices can be cut through, cf.~\cite[Rem.~2.4(b)]{KenKurMal16}, and let it be oriented in such a way that for a given root $o$ of the tree an edge $\me$
\begin{itemize}
\item with a vertex $\mv$ of degree one, is oriented such that $i(\me):=\mv$; 
\item with vertices of degree $\ge 2$, is oriented such that we choose the initial vertex to be the vertex with smaller geodesic distance to $o.$
\end{itemize}
This way, there is for every $\mv \in \mathcal V(\mathcal G)$ some $\me \in \mathcal E(\mathcal G)$ such that $i(\me)=\mv.$
We then set $y_\me(x):=\alpha_{\me,\lambda} g_{\lambda}(x)$ with $g_{\lambda}$ as defined in Lemma \ref{lem:auxi}.
By Assumption \ref{ass1} we can define the positive quantity $c_0(\mv):=\sum_{\me \in \mathcal{E}(T): i(\me)=\mv}c(\me)$ such that $\frac{c(\mv)}{c_0(\mv)}$ is uniformly bounded. Then setting $\alpha_{\me,\lambda}:=\frac{c(\mv)}{c_0(\mv)} W_{\lambda}(\mv)$ if $\me \in \mathcal{E}(T)$ with $\mv=i(\me)$ and $\alpha_{\me,\lambda}=0$ if $\me \notin \mathcal{E }(T)$ provides the desired function $y$ in ${X_{\operatorname{c}}}.$ That $y$ is indeed a function in $X_c$ follows since for $p <\infty$
\[ \Vert y \Vert_{p} \le \left( \sum_{\me \in \mathcal E(\mathcal G)} \vert \alpha_{\me,\lambda} \vert^p \Vert g_{\lambda} \Vert^p_{p} \right)^{1/p} < \infty \]
and analogously for $p=\infty.$
\end{proof}

We have established the following correspondence between the spectra of the discrete Laplacian and continuous Laplacian on metric graphs:

\begin{theo}
\label{theo:spec}
Let $\mathcal G$ be an equilateral metric graph that satisfies Assumption \ref{ass1}. 
The spectrum of $\Delta_{X_{\operatorname{c}}}$ on ${X_{\operatorname{c}}} \in \{ L^p(\mathcal G), C_0(\mathcal G), BUC(\mathcal G)\}$ away from the point spectrum $\sigma_p(\Delta_{X_{\operatorname{c}}})$ and the Dirichlet spectrum $\sigma(\Delta^D)$ coincides with 
\begin{equation*}
\left\{ \lambda \notin (\sigma_p(\Delta_{X_{\operatorname{c}}})\cup \sigma(\Delta^D)); c_{\lambda}(1) \in \sigma(P_{X_{\operatorname{d}}}) \right\},
\end{equation*}
where the correspondence between $X_c$ and $X_d$ follows the scheme in Table~\ref{tab:XY}.
\end{theo}
\begin{proof}
We assume throughout the proof that $\lambda \notin (\sigma_p(\Delta_{X_{\operatorname{c}}})\cup \sigma(\Delta^D)).$

Given arbitrary $c_{\lambda}(1) \in \rho(P_{X_{\operatorname{d}}})$ and $g \in {X_{\operatorname{c}}}$, we have to construct $u \in D(\Delta_{X_{\operatorname{c}}})$ such that $(\Delta_{X_{\operatorname{c}}} -\lambda )u=g.$

Using $g,$ we obtain the sequence $(W_{\lambda}(\mv))_{\mv \in \mathcal{V}(\mathcal G)} \in {X_{\operatorname{d}}}$, as in Lemma \ref{existg}, by the formula
\begin{equation*}
W_{\lambda}(\mv)= \sum_{\me \in \mathcal{E}_\mv} \frac{c(\me)}{c(\mv)} \int_0^1 g_\me(x)f_{I(\me),\lambda}(x) \ dx.
\end{equation*}
By assumption there is $U \in {X_{\operatorname{d}}}$ such that 
\begin{equation}
\label{discretefulfill}
(P_{X_{\operatorname{d}}}-c_{\lambda}(1))U=W_{\lambda}.
\end{equation}
Then, by inserting the respective values of $U$ into
\begin{equation}
\begin{split}
u_\me(x):=&\frac{s_{\lambda}(x)}{s_{\lambda}(1)}\left(U(t(\me))- \int_x^1 g_a(y) s_{\lambda}(1-y) \ dy\right) \nonumber \\
&+\frac{s_{\lambda}(1-x)}{s_{\lambda}(1)}\left(U(i(\me))-\int_0^x g_a(y)s_{\lambda}(y) \ dy \right)
\end{split}
\end{equation}
we obtain the desired function that solves $(\Delta_{X_{\operatorname{c}}} -\lambda )u=g$. 
A direct computation and \eqref{discretefulfill} show that the Kirchhoff conditions are satisfied as well
\begin{equation}
\begin{split}
\label{condition}
(M(\lambda)_{X_{\operatorname{d}}}u)(\mv)
&=\sum_{\me \in \mathcal{E}_{\mv}(\mathcal G)} \frac{\partial u_\me}{\partial n_\me}(\mv) = \frac{c(\mv)}{s_{\lambda}(1)}\left(-(P_{X_{\operatorname{d}}}-c_{\lambda}(1))U+W_{\lambda}\right) = 0.
\end{split}
\end{equation}
For the converse implication, assume that $\lambda \in \rho(\Delta_{X_{\operatorname{c}}})$ and $W \in {X_{\operatorname{d}}}$. By the previous Lemma \ref{existg} there is $y \in {X_{\operatorname{c}}}$ such that 
\begin{equation*}
W_{\lambda}(\mv)= \sum_{e \in \mathcal{E}_{\mv}} \frac{c(\me)}{c(\mv)} \int_0^1 y_\me(x)f_{I(\me),\lambda}(x) \ dx.
\end{equation*}
Then, there is $u \in D({X_{\operatorname{c}}})$, i.e.~$u$ satisfies in addition the Kirchhoff condition, such that $(\Delta_{X_{\operatorname{c}}}-\lambda)u=y.$ Let $U \in {X_{\operatorname{d}}}$ be the vertex points associated with $u$, then the computation in \eqref{condition} shows that $(P_{X_{\operatorname{d}}}-c_{\lambda}(1))U=W_{\lambda}.$
This completes the proof.
\end{proof}

From our discussion so far, we conclude the following spectral independence result for equilateral graphs satisfying the sub-exponential growth condition in Assumption \ref{ass1}:
\begin{corr}[Spectral independence of Laplacian]
\label{corr}
The spectrum of the Laplacian $\Delta_{X_{\operatorname{c}}}$, for $\lambda \in \rho(H^D)$ on equilateral graphs satisfying Assumption \ref{ass1}, is independent of the space ${X_{\operatorname{c}}} \in \{ L^p(\mathcal G), C_0(\mathcal G), BUC(\mathcal G) \}$, i.e., $\sigma(\Delta_{X_{\operatorname{c}}}) \cap \rho(H^D)=\sigma(\Delta_{L^2}) \cap \rho(H^D).$
\end{corr}
\begin{proof}
It suffices to observe that the operator $P_{X_{\operatorname{d}}}$, defined in \eqref{eq:PX}, is of the form considered in \cite{BAUER2013717}.
This result follows by combining Theorem \ref{theo:spec} with \cite[Theo.~$2.1$]{BAUER2013717}. For the space $c_0$, which is not discussed in \cite{BAUER2013717}, we use that $c_0^{**}=\ell^{\infty}.$
\end{proof}

As a consequence of the consistency of the resolvents, cf. \eqref{eq:consres}, we show that the Kre\u{\i}n resolvent formula extends to spaces ${X_{\operatorname{c}}} = L^p(\mathcal G)$ with $p < \infty$:
\begin{corr}[Kre\u{\i}n resolvent formula] Let Assumption \ref{ass2} be satisfied and $\mathcal G$ be equilateral. We consider the Schrödinger operator $H_p=-\Delta+V$ with potential $V \in L^{\infty}(\mathcal G)$, that is the same on every edge, on $L^p(\mathcal G)$. Let $z \in \rho(H^D) \cap \rho(H_{p})$, where $H^D$ is the Schrödinger operator with Dirichlet boundary conditions on every edge. Then the operator $M_{\ell^p}$, defined in \eqref{eq:Mlambda}, is invertible and satisfies $(H_p-z)^{-1} = (H_p^D-z)^{-1}+ \gamma_{\ell^p}(z)M_{\ell^p}(z)^{-1} \mu_{\ell^p}(\overline{z}).$
Here $\gamma_{\ell^p}$ is the gamma field introduced in \eqref{eq:algebraiceq}.
\end{corr}
\begin{proof}
Since this formula is well-known to hold on $L^2(\mathcal G)$, see e.g.~\cite{Pan06} and references therein, and we have already established spectral independence in Theorem \ref{theo:specind} and consistency of resolvents in \eqref{eq:consres}, we can use an $L^2$-approximation argument: 
If we then take a sequence $u_n \in C_c(\mathcal G)$ approximating some $u \in {X_{\operatorname{c}}}$, the following limits exist:
\begin{equation*}
\begin{split}
&(H_p-z)^{-1}u = \lim_{n \rightarrow \infty, \Vert \bullet \Vert_{X_{\operatorname{c}}}} (H_{2}-z)^{-1}u_n, \\
&(H^D_{p}-z)^{-1}u = \lim_{n \rightarrow \infty, \Vert \bullet \Vert_{X_{\operatorname{c}}}} (H^D_{2}-z)^{-1}u_n , \text{ and }\\
&\left(\gamma_{{\ell^p}}(z)M_{{\ell^p}}(z)^{-1} \mu_{{\ell^p}}(\overline{z})\right) u = \lim_{n \rightarrow \infty, \Vert \bullet \Vert_{\ell^p}} \left(\gamma_{\ell^2}(z)M_{\ell^2}(z)^{-1} \gamma_{\ell^2}(\overline{z})^*\right) u_n.
\end{split}
\end{equation*}
This concludes the proof.
\end{proof}

\section{Markov semigroup and Brownian motion on metric graphs}
\label{sec:Markov}
Based on our preliminary work in previous sections, we are now in the position to introduce a canonical stochastic process which we will refer to as Brownian motion on infinite metric graphs, cf. \cite{KosPotSch12} for related results on finite graphs. 

\begin{defi}[Brownian motion]
We define the Brownian motion $(B_t)$ on a metric graph as the Feller process associated with the rescaled heat semigroup $A_{\operatorname{BM}} := \frac{1}{2} \Delta$ on ${\rm BUC}(\mathcal G)$.
\end{defi}

The theory of Feller processes gives immediately the \emph{Feynman--Kac} formula \cite[Theo.~$3.47$]{liggett2010continuous}. An analogous formula for \textit{discrete} graphs has been proved in~\cite{GueKelSch16}.
\begin{prop}[Feynman--Kac formula on metric graphs]
\label{prop:FK}
Consider the Feller process $B_t$ with generator $\Delta_{C_0}$ and a function $V \in C(\mathcal G).$
We then define for $f \in D(\Delta_{C_0})$
\[ u_f(t,x) =\mathbb E^x \left( f(X_t)\operatorname{exp}\left( \int_0^t -V(X_s) \ ds \right)\right).\]
It follows that $u_f(t,x) \in D(\Delta_{C_0})$ and $u_f$ solves the parabolic Schrödinger equation
\begin{equation*}
 \begin{split}
\partial_t u_f(t,x) &= -(-\Delta +V)u_f(t,x), \quad t>0 \\
u_f(0,x)&=f.
\end{split}
\end{equation*}
\end{prop}
We conclude with an immediate corollary on heat kernel estimates for Schrödinger semigroups and higher-order parabolic equations.

\begin{theo}[Heat kernel estimates]
\label{theo:kernel}
Let $\mathcal G$ be a metric graph and $V \in C^{\infty}(\mathcal E(\mathcal G))$ such that all its derivatives are bounded over the entire metric graph.
 Then the heat kernel of the Schrödinger semigroups and higher-order parabolic equations satisfies $k_t \in C^{\infty}(\mathcal E(\mathcal G) \times \mathcal E(\mathcal G)).$ 

For potentials $V \in C(\mathcal G) \cap L^{\infty}(\mathcal G)$ and all $x,y \in \mathcal G$, where $\mathcal G$ satisfies Assumption \ref{ass2}, we have for all $t >0$
\[ \vert e^{-tH}(x,y) \vert \lesssim \vert e^{t\Delta}(x,y) \vert^{1/2} t^{-1/4} e^{t \Vert V^+ \Vert_{\infty}} \]
where $V^+$ is the positive part of $V.$
Moreover, for semigroups $\vert e^{-t(-\Delta)^{m}}\vert$ there is $C_{t_0}>0$ such that for all $x,y \in \mathcal G$ and $t \ge t_0$ 
\[ \vert e^{-t(-\Delta)^{m}}(x,y)\vert \le C_{t_0} t^{-1/(2m)} \operatorname{exp}\left(-\tfrac{c_2}{2t^{1/(2m-1)}} \vert d(x,y) \vert^{\frac{2m}{2m-1}} + c_3 t\right)\]
where $c_2$ and $c_3$ are as in \eqref{eq:otherkernels}.
\end{theo}
\begin{proof}
The functional calculus implies that for $t>0$ the semigroup $T(t)=e^{-tH}$ is a map $T(t) \in \mathcal L( L^2(\mathcal G), H^m(\mathcal G))$ for all $m \ge 0.$ Duality implies then that $T(t) \in \mathcal L(H^{-n}(\mathcal G),L^2(\mathcal G))$ for $n \ge 0$ such that by the semigroup property
\[ T(t+s) \in \mathcal L(H^{-n}(\mathcal G),H^m(\mathcal G)) \text{ for all }m,n \in \mathbb N_0.\]
This implies that $T(t) \in \mathcal L(\mathscr E'(\mathcal E(\mathcal G)),\mathscr E(\mathcal E(\mathcal G)))$ for all $t>0.$

\medskip

The existence of the integral kernel representation for the Schrödinger semigroup follows by Schwartz's kernel theorem \cite[Theo.~$5.2.1$]{Hoermander}. Since the semigroup is smoothing for $t>0$ we conclude that the integral kernel must be given by a smooth function \cite[Theo.~$5.2.6$]{Hoermander}. 
The $p$-independence of the spectrum of the semigroup (away from zero) follows from the spectral mapping theorem in Corollary \ref{corr:SMT}.

\medskip

Hölder's inequality in the Feynman--Kac formula implies that 
\begin{equation*}
\begin{split}
\vert u_f(t,x) \vert 
&\le \mathbb E^x \left( \vert f(B_t) \vert^{p^{-1}+q^{-1}}\operatorname{exp}\left( \int_0^t -V(B_s) \ ds \right)\right) \\
&\le \mathbb E^x \left( \vert f(B_t) \vert \right)^{q^{-1}} \mathbb E^x \left( \vert f(B_t) \vert \operatorname{exp}\left( \int_0^t -pV(B_s) \ ds \right)\right)^{p^{-1}}\\
&= (e^{t\Delta}\vert f \vert(x))^{1/q}(e^{t(\Delta-pV)}\vert f \vert(x))^{1/p}.
\end{split}
\end{equation*}
If we now choose $f$ to be an approximate identity $0 \le f_n \rightarrow \delta_y$ we find using the ultracontractivity of the Schrödinger semigroup, i.e., Theorem \ref{theo:hypercontr}, that
\begin{equation*}
\begin{split}
\vert e^{-tH}(x,y) \vert 
&\le \vert e^{t\Delta}(x,y) \vert^{1/q} \vert e^{t(\Delta-pV)}(x,y) \vert^{1/p} \\
&\le \vert e^{t\Delta}(x,y) \vert^{1/q} \sup_{x',y'} \vert e^{t(\Delta-pV)}(x',y') \vert^{1/p} \\
&= \vert e^{t\Delta}(x,y) \vert^{1/q} \Vert e^{t(\Delta-pV)} \Vert_{\mathcal L(L^1,L^{\infty})}^{1/p} \\
&\lesssim \vert e^{t\Delta}(x,y) \vert^{1/q} t^{-1/(2p)} e^{t \Vert V^+ \Vert_{\infty}} 
\end{split}
\end{equation*}

The kernel estimates on the graph then follow by using that for $d(x,y) \ge n \ell_{\downarrow}$ and $k_t$ the integral kernel of the semigroup $(e^{-t(-\Delta)^{m}})_{t \ge 0}$ on $\RR,$ we find from kernel estimates \eqref{eq:otherkernels}, estimate \eqref{eq:expthree} on the transfer coefficients, and the definition of the kernel on the graph \eqref{eq:kernel}
\begin{equation*}
\begin{split}
 \vert K_{k_t}(x,y) \vert &\lesssim t^{-1/(2m)} \operatorname{exp}\left(-\tfrac{c_2 \vert d(x,y) \vert^{\frac{2m}{2m-1}} }{2t^{1/(2m-1)}}+ c_3 t\right) \sum_{k \ge n} \operatorname{exp}\left(-\tfrac{c_2}{2t^{1/(2m-1)}} \vert k \ell_{\downarrow} \vert^{\frac{2m}{2m-1}}\right)3^k \\
&\lesssim t^{-1/(2m)} \operatorname{exp}\left(-\tfrac{c_2}{2t^{1/(2m-1)}} \vert d(x,y) \vert^{\frac{2m}{2m-1}} + c_3 t\right).
\end{split}
\end{equation*}
Here, we used that $\sum_{k \ge n} \operatorname{exp}\left(-\tfrac{c_2}{2t^{1/(2m-1)}} \vert k \ell_{\downarrow} \vert^{\frac{2m}{2m-1}}\right)3^k $ is uniformly bounded for $t \ge t_0$ in $n.$
\end{proof}


\begin{prop}
The process $(B_t)$ has a.s.~continuous paths and has martingales $M_t:=B_t$ and $M_t:=\vert B_t \vert^2-t$ on the cubic metric graph $\mathcal G$ induced by the integer lattice $\mathcal V:=\ZZ^d$ with unit conductivities. 
\end{prop}
\begin{proof}
The continuity of paths follows by the locality of the graph Laplacian \cite[Theo.~$4.5.1$]{Dirichlet}. 

The martingales can be derived directly from the Markov generator using Dynkin's formula: 

Let $f \in D(\Delta_{C_0})$ then $\mathcal M_t^f:= f(B_t) - f(B_0) - \int_0^t \frac{1}{2} \Delta f(B_s) \ ds $
defines a martingale with respect to the filtration induced by the process $(B_t)_{t \ge 0}.$

The above martingales can be derived directly by approximating the (unbounded) functions $f:\mathcal G \rightarrow \mathcal G$ with $f(x) := x$ and $f:\mathcal G \rightarrow \RR$ with $f(x):=\vert x\vert^2$ on sufficiently large sets of growing size.

More precisely, let $f_n(x) = \chi_{n}(x) f(x) \in D(\Delta_{C_0}),$ where $\chi_n$ is a smooth cut-off function of modulus at most one, such that $\chi_n(x)=1$ for $\vert x \vert \le n$ and $\chi_n(x) = 0$ for $\vert x \vert \ge 2n.$

We then have that from the kernel estimates in Theorem \ref{theo:kernel}
\begin{equation*}
\begin{split}
\mathbb E^{x}( (f -f_n)(B_t) )&= T_t (f (1- \chi_n))(x) \\
&= \int_{B_{n}(0)^c} e^{t\Delta}(x,y)f(y) (1-\chi_{n}(y)) \ dy \xrightarrow [n \rightarrow \infty]{}0. 
\end{split}
\end{equation*}

This implies that $\lim_{n \rightarrow \infty} \mathbb E^x \vert M_t^{f_n}- M_t^{f} \vert =0 \text{ for all } t>0$ which shows that $M_t^t$ is a martingale as well, even though $f \notin D(\Delta_{C_0}).$
\end{proof}

\section{Parabolic Anderson model}
\label{sec:PAM}
We next turn to an application of our semigroup framework and introduce the parabolic Anderson model (PAM) on a periodic metric graph as the Cauchy problem
\begin{equation}
\begin{split}
\label{eq:evoleq}
\partial_t u(t,x) &=-(- \Delta + V_{\omega}) u(t,x), (t,x) \in (0,\infty) \times \mathcal G \text{ such that }u(0,x) = \indic.
\end{split}
\end{equation}

This model can be seen as a natural extension of the discrete parabolic Anderson model on combinatorial graphs \cite{gartner1990}. 
Here, we assume that $V_{\omega}$ are i.i.d.~ realizations of a potential on every edge $\me \in \mathcal E(\mathcal G).$ From our analysis so far, we know that there exists a random semigroup $(\omega \mapsto (T_{\omega}(t))_{t\ge 0})$ associated with the evolution equation \eqref{eq:evoleq} for $BUC(\mathcal G)$ initial data.

\medskip

In the following, we assume that $\mathcal G$ is a periodic equilateral metric graph with fundamental domain $D.$

We can now prove a version of intermittency for our random Schrödinger operator using the tools from semigroup theory we developed in the preceding chapters.
For this purpose, we define functions $\Lambda_p(t):= \log \mathbb E \int_D u(t,x)^p \ \frac{dx}{\vert D \vert} $ for $t \ge 0, p \in \mathbb N$.

Since the function $p \mapsto \Lambda_p(t)$ is convex, difference quotients $p \mapsto \Lambda_{p+1}(t) -\Lambda_{p}(t)$ are increasing in $p.$

Hence, we conclude that for $q \ge p$
\begin{equation*}
\begin{split}
q\Lambda_{q+1}(t) - (q+1)\Lambda_q(t) 
&= \sum_{k=0}^{q-1} \left((\Lambda_{q+1}(t)-\Lambda_q(t))-(\Lambda_{k+1}(t)-\Lambda_k(t)) \right) \\
&\ge \sum_{k=0}^{p-1} \left((\Lambda_{p+1}(t)-\Lambda_p(t))-(\Lambda_{k+1}(t)-\Lambda_k(t)) \right) \\
&=p\Lambda_{p+1}(t)-(p+1) \Lambda_p(t).
\end{split}
\end{equation*}

We then write $f \ll g$ to denote that $\lim_{t \rightarrow \infty} g(t)-f(t) = \infty.$
Thus, if $(p+1) \Lambda_p \ll p\Lambda_{p+1}$ then this implies also that $(q+1) \Lambda_q(t) \ll q\Lambda_{q+1} \ \text{ for all } q \ge p.$

We also define translations $(D_n)_{n \in \mathbb Z^d}$, with $D_0=D$, of the fundamental domain such that $\mathcal G = \sqcup_n D_n.$
\begin{defi}[Intermittency]
The solution to the parabolic Anderson model is called \emph{intermittent} if the Lyapunov exponents 
\[ \lambda_p:=\lim_{t \rightarrow \infty} \frac{\Lambda_p(t)}{t}, \text{ for } p \in \mathbb N \text{ exist and we have }\lambda_1 < \frac{\lambda_2}{2} < \frac{\lambda_3}{3} <...\quad. \]
\end{defi}

\begin{prop}
Suppose $H=-\Delta+V$ such that $V=(V_{\me})_{\me}$ is an i.i.d. family of bounded non-degenerate random potentials, then the parabolic Anderson model is intermittent.
\end{prop}
\begin{proof}
Let $k_{\omega}(t,x,y)$ be the Schwartz kernel of the PAM.
We then have that since the probability measure is translation invariant
\begin{equation}
 \begin{split}
 \label{eq:productformula}
 \mathbb E \int_D u_{\indic}(t+s,x) \ \frac{dx}{\vert D \vert} 
 &= \mathbb E \int_D \int_{\mathcal G} \int_{\mathcal G} k_{\omega}(s,x,y) k_{\omega}(t,y,z) \ dy \ dz \ \frac{dx}{\vert D \vert}\\
 &= \mathbb E \int_D \sum_{n,m \in \mathbb Z^d} \int_{D_n} \int_{D_m} k_{\omega}(s,y,x) k_{\omega}(t,y,z) \ dy \ dz \ \frac{dx}{\vert D \vert}\\
 &= \mathbb E \int_{D} \sum_{n,m \in \mathbb Z^d} \int_{D_{-m}} \int_{D_{n-m}} k_{\omega}(s,y,x) k_{\omega}(t,y,z) \ dz \ dx \ \frac{dy}{\vert D \vert} \\
 &= \mathbb E \int_D \left( u_{\indic}(s,x) u_{\indic}(t,x) \right) \ \frac{dx}{\vert D \vert}.
\end{split}
\end{equation}

Let $\lambda_0:=\sup \sigma(-H_{\omega})$ a.s.. Then, we conclude from \eqref{eq:productformula} that 
\[ \Lambda_1(2t)=\log \mathbb E \int_D u_{\indic}(2t,x) \ \frac{dx}{\vert D \vert} = \log \mathbb E \int_D u_{\indic}(t,x)^2 \ \frac{dx}{\vert D \vert}= \Lambda_2(t)\]
and using the Cauchy--Schwarz inequality 
\[ \left(\mathbb E \int_D u_{\indic}(s+t,x) \ \frac{dx}{\vert D \vert} \right)^2 \le \left(\mathbb E \int_D u_{\indic}(2s,x) \ \frac{dx}{\vert D \vert} \right) \left(\mathbb E \int_D u_{\indic}(2t,x) \ \frac{dx}{\vert D \vert} \right). \]
Taking the logarithm of the previous line we find that 
\[ 2\Lambda_1(s+t) \le \Lambda_1(2s)+ \Lambda_1(2t) \text{ for all } s,t \ge 0. \]
Since $\Lambda_1$ is also continuous and satisfies $\Lambda_1\left(\frac{s+t}{2}\right) \le \frac{ \Lambda_1(s)+ \Lambda_1(t)}{2},$
this implies the convexity of $\Lambda_1.$

\medskip

Positivity of the semigroup yields that using the spectral theorem with spectral measure $ \mathcal E_{H_{\omega}}$ for the operator $H_{\omega}$
\begin{equation}
\label{eq:zestimate}
 \mathbb E \int_D u_{\indic}(t,x) \ \frac{dx}{\vert D \vert} \ge \frac{ \mathbb E\langle T(t) \indic_D, \indic_D \rangle_{L^2}}{\vert D \vert}= \frac{1}{\vert D \vert}\mathbb E \int_{\sigma(H_{\omega})}e^{-\lambda t} \ d\langle \mathcal E_{H_{\omega}}(\lambda) \indic_D, \indic_D \rangle_{L^2}.
 \end{equation}
We have for any $\lambda>\lambda_0$ 
\begin{equation*}
 \begin{split}
\frac{\mathbb E\langle \mathcal E_{H_{\omega}}([\lambda_0,\lambda]) \indic_{D_n}, \indic_{D_n} \rangle_{L^2} }{\vert D \vert}
&=\frac{\mathbb E \langle \mathcal E_{H_{T^n \omega}}([\lambda_0,\lambda]) \indic_{D_0}, \indic_{D_0} \rangle_{L^2} }{\vert D \vert} \\
&=\frac{\mathbb E \langle \mathcal E_{H_{\omega}}([\lambda_0,\lambda])\indic_{D},\indic_{D} \rangle_{L^2}}{\vert D \vert} >0.
\end{split}
\end{equation*}
This implies, by applying the logarithm to \eqref{eq:zestimate}, that $\lim_{t \rightarrow \infty} \Lambda_1(t)/t\ge \lambda_0.$

\medskip

To show that $\lim_{t \rightarrow \infty} \Lambda_1(t)/t=\lambda_0$ we use that the spectral bound coincides with the growth bound (Corollary \ref{corr:SMT})
\[ \mathbb E \int_D u_{\indic}(t,x) \ \frac{dx}{\vert D \vert} = \mathbb E \int_D (T_{\omega}(t) \indic)(x) \ \frac{dx}{\vert D \vert} \lesssim e^{\lambda_0 t} \text{ a.s.} \]

\medskip

Convexity implies then that for $t_2 \ge t_1$
\[ \phi(t_2):=\frac{\Lambda_1(t_2)}{t_2}=\frac{\Lambda_1(t_2)-\Lambda_1(0)}{t_2} \ge \frac{\Lambda_1(t_1)-\Lambda_1(0)}{t_1}=\frac{\Lambda_1(t_1)}{t_1}=\phi(t_1). \]

Thus, if we assume that $\lambda=0$ such that $\lim_{t \rightarrow \infty} \Lambda_1(t)/t=0,$ then by the monotonicity of $\Lambda_1(t)/t$ we conclude that for all $t>0$ we have that $\Lambda_1(t)/t \le 0$. This implies $\Lambda_1(t) \le 0.$

Moreover, $\Lambda_1$ is non-increasing since from convexity we have for $t_0<t_1<t_2$
\[\Lambda_1(t_2) \ge \frac{\Lambda_1(t_1)-\Lambda_1(t_0)}{t_1-t_0} (t_2-t_1)+\Lambda_1(t_1). \]
Assuming that $\Lambda_1(t_1)>\Lambda_1(t_0)$ for some $0<t_0<t_1$ then the difference quotient in the previous line is positive.
Thus, we find that 
\[\limsup_{t_2 \rightarrow \infty} \Lambda_1(t_2) \ge \limsup_{t_2 \rightarrow \infty} \frac{\Lambda_1(t_1)-\Lambda_1(t_0)}{t_1-t_0} (t_2-t_1)+\Lambda_1(t_1)=\infty \]
which is impossible as $\Lambda_1(t) \le 0$. Therefore, $ \Lambda_1(t)$ is non-increasing.
Hence, since $ \Lambda_1(t)$ is non-negative, we have the existence of some $\rho$ such that $\lim_{t \rightarrow \infty} \Lambda_1(t)=\rho.$
Thus, if $\rho>-\infty$ we find that $\lim_{t \rightarrow \infty} \Lambda_2(t)-2\Lambda_1(t)= \lim_{t \rightarrow \infty} \Lambda_1(2t)-2\Lambda_1(t)=-\rho.$

On the other hand, if we assume that $\rho=-\infty$, we may note that the function $\Lambda_2(t)-2\Lambda_1(t) = \Lambda_1(2t)-2\Lambda_1(t)\ge 0$ by convexity and is non-decreasing, as the difference quotient of convex functions is non-decreasing in both parameters.
This implies that there is $c\in [0,\infty]$ such that $\lim_{t \rightarrow \infty} \Lambda_2(t)-2\Lambda_1(t) =c.$

To show that in this case $c=\infty$ we observe that 
\[ \Lambda_1(2^n t) -2^n\Lambda_1(t) = \sum_{i=1}^n 2^{n-i}\left(\Lambda_1(2^i t)-2 \Lambda_1(2^{i-1} t)\right) \le \sum_{i=1}^n 2^{n-i}c = (2^n-1)c.\]
Thus, we have from dividing by $2^n t$ and taking the limit $n \rightarrow \infty$
\[ 0= \lim_{n \rightarrow \infty} \frac{\Lambda_1(2^n t) }{2^n t} \le \frac{\Lambda_1(t)+c}{t}\quad \forall t\ge 0.\]
Now if $c$ was not infinite, then choosing $t$ large enough gives a contradiction since $\lim_{t \rightarrow \infty} \Lambda_1(t)=-\infty.$

From \eqref{eq:productformula} we then deduce that the random variable 
\[Z_{\omega}(x):=\mathbb E^x_{(X_t)} e^{\int_0^{\infty}-V_{\omega}(X_s)+\lambda_0 \ ds} \in [0,1)\]
which is strictly less than one, as $V_{\omega}$ is non-degenerate describes the $t \rightarrow \infty$ limit of the solution by the Feynman--Kac formula, cf. Prop. \ref{prop:FK}. 
Moreover, by the dominated convergence theorem, we find 
\begin{equation*}
 \begin{split}
 \mathbb E \int_D Z(x) \ \frac{dx}{\vert D \vert}&= \limsup_{t \rightarrow \infty }e^{\lambda_0 t} \mathbb E\int_D (u_{\indic}(t,x)) \ \frac{dx}{\vert D \vert} \\
 &= \limsup_{t \rightarrow \infty } \mathbb E_{V_{\omega}} \int_D \mathbb E^x_{(X_t)} e^{\int_0^t-V_{\omega}(X_s)+\lambda_0 \ ds} \frac{dx}{\vert D \vert}\\
 &= \limsup_{t \rightarrow \infty } \mathbb E_{V_{\omega}} \int_D \left(\mathbb E^x_{(X_t)} e^{\int_0^t-V_{\omega}(X_s)+\lambda_0 \ ds}\right)^2 \frac{dx}{\vert D \vert} =\mathbb E\int_D Z(x)^2 \frac{dx}{\vert D \vert}.
 \end{split}
\end{equation*}
Thus, since $Z_{\omega}(x) \in [0,1)$ it follows that $Z_{\omega}=0$ a.s., which implies $\rho=-\infty$ and the claim follows as $c=\infty$.
\end{proof}

\begin{appendix}

\section{Combes--Thomas estimate}
For the Combes--Thomas estimate we introduce on edges $\me$ functions for a fixed reference vertex $\mv$
\[ (f_{\mv})_{\me} (t):= d(i(\me)+t,\mv) + k_{i(\me),\me} \chi_{i(\me),\me}(t)+ k_{t(\me),\me} \chi_{t(\me),\me}(t)\] 
where $\chi_{i(\me),\me}(0)=\chi_{t(\me),\me}(\vert \me \vert)=0$ and $\supp(\chi_{i(\me),\me})$ is contained in a small neighborhood of $0$ while $\supp(\chi_{t(\me),\me})$ is contained in a small neighborhood of $\vert \me \vert$. This implies that 
\begin{equation*}
\begin{split}
(f_{\mv}')_{\me} (t) &=1+ k_{i(\me),\me} \chi'_{i(\me),\me}(t)+ k_{t(\me),\me} \chi'_{t(\me),\me}(t) \text{ and } \\ (f_{\mv}^{(n)})_{\me}(t) &= k_{i(\me),\me} \chi^{(n)}_{i(\me),\me}(t)+ k_{t(\me),\me} \chi^{(n)}_{t(\me),\me}(t) \text{ for } n \ge 2
\end{split}
\end{equation*}
such that for some fixed $\mv \in \mathcal V$ we have both $f_{\me}(0)=d(i(\me),\mv)$ and $f_{\me}(\vert \me \vert)=d(t(\me),\mv)$
and coefficients $k$ above are chosen such that the first derivative satisfies the natural boundary conditions at the vertices. Moreover, we assume that the first and second derivatives are uniformly bounded over the entire edge. This is possible as both the edge length and the conductivities are uniformly bounded from above and below.

For the sake of simplicity, in this article we have refrained from discussing semigroup generated by higher order Schrödinger operator $-(-\Delta)^m-V$. The careful reader will have noticed that most of our results easily carry over to this more general setting, but we avoid going into details. The main ingredient for this extension would be the following.

\begin{prop}[Combes--Thomas estimate]
\label{prop:CTE}
Let $V \in L^{\infty}$ and $m \in \mathbb N$ and consider $H:=(-\Delta)^m+V$. For graphs satisfying Assumption \ref{ass2}, and a set $K$ that is bounded away from $\sigma(H)$. Then, it follows for $z \in K$ and $p \in [1,\infty)$
\begin{equation*}
\begin{split}
&\left\lVert (H-z)^{-1} \right\rVert_{\mathcal L(L^p)} \lesssim \Vert (H_{\varepsilon}-z)^{-1} \Vert_{\mathcal L(L^2;H^2)} \lesssim_{\varepsilon,K} \Vert (H-z)^{-1} \Vert_{\mathcal L(L^2;H^2)} \lesssim_{\varepsilon,K} \vert \Im (z) \vert^{-1}.
\end{split}
\end{equation*}
for the operator $H_{\varepsilon}:= e^{\varepsilon f_{\mv'}(x)}H e^{-\varepsilon f_{\mv'}(x)}$ for a vertex $\mv' \in \mathcal V.$
\end{prop}
Fix a vertex $\mv' \in \mathcal V.$
\begin{proof}
\[H_{\varepsilon}:= e^{\varepsilon f_{\mv'}(x)}H e^{-\varepsilon f_{\mv'}(x)} \stackrel{m=0}{=} H -\varepsilon f_{\mv'}''(x)+\varepsilon^2 f_{\mv'}'(x)^2 -2\varepsilon f_{\mv'}'(x) \partial_x. \]
Thus, for $m=0$ we find that
\begin{equation*}
\label{eq:Ha}
(H_{\varepsilon}-z) = (1 + \underbrace{ ( -\varepsilon f_{\mv'}''(x)+\varepsilon^2 f_{\mv'}'(x)^2 -2\varepsilon f_{\mv'}'(x) \partial_x)(H-z)^{-1}}_{=\mathcal O_K(\varepsilon)})(H-z) 
\end{equation*}
for $ z \in \rho(H).$ 
Analogous computations show that for arbitrary $m \in \mathbb N_0$ and compact sets $K \subset \rho(H)$
\begin{equation*}
(H_{\varepsilon}-z) = (1 + \mathcal O_{K,m}(\varepsilon))(H-z).
\end{equation*}
This implies the inclusion $K \subset \rho(H_{\varepsilon})$ too if $ {\varepsilon} \le C_{m,K}$ for some constant $C_{m,K}$ sufficiently small.
Let $(\chi_{\mv})_{\mv \in \mathcal V}$, be a smooth partition of unity with support $ \supp(\chi_{\mv})$ being concentrated uniformly around $\mv.$
For such $ {\varepsilon} \le C_{m,K}$, it follows that
\begin{equation*}
\begin{split}
&\langle \chi_{\mv}h, (H-z)^{-1}\chi_{\mv'}g \rangle = \langle \chi_{\mv}he^{-\varepsilon f_{\mv'}}, (H_{\varepsilon}-z)^{-1}e^{\varepsilon f_{\mv'}} g\chi_{\mv'} \rangle \\
 &= e^{-\varepsilon f_{\mv'}(\mv)} \langle \chi_{\mv}h e^{-\varepsilon(f_{\mv'}-f_{\mv'}(\mv))}, (H_{\varepsilon}-z)^{-1}e^{\varepsilon f_{\mv'} } g\chi_{\mv'} \rangle.
\end{split}
\end{equation*}
Thus, we have shown, using Sobolev's embedding $H^2 \hookrightarrow L^p$ that 
\begin{equation*}
\begin{split}
& \vert \langle h\chi_{\mv}, (H-z)^{-1} g\chi_{\mv'} \rangle \vert \\
 &\lesssim e^{-\varepsilon d(\mv,\mv')} \Vert h\chi_{\mv}e^{-\varepsilon(d(\bullet,\mv')-d(\mv,\mv'))}\Vert_{L^q} \Vert (H_{\varepsilon}-z)^{-1} e^{{\varepsilon}f_{\mv'}} g \chi_{\mv'} \Vert_{L^{p}} \\
 &\lesssim e^{-\varepsilon d(\mv,\mv')} \Vert h\chi_{\mv}e^{-\varepsilon(d(\bullet,\mv')-d(\mv,\mv'))}\Vert_{L^q} \Vert (H_{\varepsilon}-z)^{-1} e^{\varepsilon d(\bullet,\mv')} g \chi_{\mv'} \Vert_{H^2} \\
 & \lesssim e^{-\varepsilon d(\mv,\mv')} \Vert h\chi_{\mv}e^{-\varepsilon(d(\bullet,\mv')-d(\mv,\mv'))}\Vert_{L^q} \Vert (H_{\varepsilon}-z)^{-1} \Vert_{\mathcal L(L^2,H^2)} \Vert e^{\varepsilon d(\bullet,\mv')} g \chi_{\mv'} \Vert_{L^2} \\
 & \lesssim e^{-\varepsilon d(\mv,\mv')} e^{2\varepsilon} \Vert h\chi_{\mv}\Vert_{L^q} \Vert (H_{\varepsilon}-z)^{-1} \Vert_{\mathcal L(L^2,H^2)} \Vert g \chi_{\mv'} \Vert_{L^2}.
\end{split}
\end{equation*}
Hence, we conclude from Young's inequality that
\begin{equation*}
\begin{split}
\vert \langle h, (H-z)^{-1} g \rangle \vert 
&\le \sum_{\mv,\mv' \in \mathcal V} \vert \langle h\chi_{\mv}, (H-z)^{-1} g\chi_{\mv'} \rangle \vert \\
&\lesssim \Vert(H_{\varepsilon}-z)^{-1} \Vert_{\mathcal L(L^2,H^2)} \sum_{\mv,\mv' \in \mathcal V} e^{-\varepsilon d(\mv,\mv')} \left(e^{2\varepsilon} \Vert h\chi_{\mv}\Vert^q_{q} + \Vert g \chi_{\mv'} \Vert^p_{p} \right) \\
&\lesssim_{\varepsilon} \Vert(H_{\varepsilon}-z)^{-1} \Vert_{\mathcal L(L^2,H^2)} \left( \Vert h\Vert^q_{q} + \Vert g \Vert^p_{p}\right).
\end{split}
\end{equation*}
This concludes the proof.
\end{proof}

\section{Properties of transfer coefficients}
In this section we record various useful properties of the transfer coefficients introduced in Definition \ref{def:TC}.
We start by collecting some properties of transfer coefficients $\mathbb T_{\me,\me'}$.
\begin{lemm}
\label{eq:lemm12}
For a metric graph $\mathcal G$ and $\me,\me' \in \mathcal E(\widetilde{\mathcal G})$ the following estimates on the transfer matrix hold
\begin{equation}
\begin{split}
\label{eq:firstones}
&\sum_{\me \in t^{-1}(i(\me'))} \left\vert \mathbb T_{\me,\me'} \right\vert \le 3, \quad \sum_{\me \in t^{-1}(i(\me'))} \mathbb T_{\me,\me'}=1, \text{ and } \\
&\sum_{\me' \in i^{-1}(t(\me))} c(\me') \left\vert \mathbb T_{\me,\me'} \right\vert \le 3c(\me), \quad \sum_{\me' \in i^{-1}(t(\me))} c(\me')\mathbb T_{\me,\me'} =c(\me).
\end{split}
\end{equation}
Moreover, let $P$ be a path of length $m$ starting at an edge $\me$ then 
\begin{equation}
\label{eq:expthree}
\sum_{P \in \mathcal P_{\me, \me'}(m)} \left\lvert \mathbb T_P \right\rvert \le 3^{m} 
\end{equation}
and also 
\begin{equation}
\label{eq:expthree2}
c(\me)^{-1} \sum_{\me' \in \mathcal E(\widetilde{\mathcal G})} \sum_{P \in \mathcal P_{\me, \me'}(m)} c(\me') \vert \mathbb T_P \vert \le 3^{m }.
\end{equation}
\end{lemm}
\begin{proof}
The first set of bounds \eqref{eq:firstones} follow right from the definition of transfer coefficients. 
To see \eqref{eq:expthree}, we use that for a path $P$ of length $m$ connecting edges $\me$ and $\me'$ we have by induction using \eqref{eq:firstones}, assuming it to holds for paths of length $m$ 
\[\sum_{P \in \mathcal P_{\me, \me'}(m+1)} \vert \mathbb T_P \vert= \sum_{\me'' \in t^{-1}(i(\me')) }\sum_{P \in \mathcal P_{\me, \me''}(m)} \vert \mathbb T_P \vert \vert \mathbb T_{\me'',\me'} \vert \le 3^{m+1}.\]
The estimate \eqref{eq:expthree2} follows from a similar inductive argument. Assuming it to hold for paths of length $m$, we find 
\begin{equation*}
\begin{split}
c(\me)^{-1}\sum_{\me' \in \mathcal E(\widetilde{\mathcal G})} \sum_{P \in \mathcal P_{\me, \me'}(m+1)} c(\me') \vert \mathbb T_P \vert &\le c(\me)^{-1}\sum_{\me' \in \mathcal E(\widetilde{\mathcal G})} \sum_{\me'' \in t^{-1}(i(\me))} \sum_{P \in \mathcal P_{\me, \me''}(m)} \vert \mathbb T_P \vert c(\me')\vert \mathbb T_{\me'',\me'} \vert \\
&\le c(\me)^{-1}\sum_{\me'' \in \mathcal E(\widetilde{\mathcal G})}\sum_{P \in \mathcal P_{\me, \me''}(m)} \vert \mathbb T_P \vert \sum_{\me' \in i^{-1}(t(\me''))} c(\me')\vert \mathbb T_{\me'',\me'} \vert \\
&\le 3 \sum_{\me'' \in \mathcal E(\widetilde{\mathcal G})}\sum_{P \in \mathcal P_{\me, \me''}(m)} c(\me'') \vert \mathbb T_P \vert \le 3^{m+1}
\end{split}
\end{equation*}
where we used \eqref{eq:firstones} in the second-to-last row and the induction hypothesis in the last one.
This concludes the proof.
\end{proof}

\section{Convolution properties}
\begin{lemm}
\label{lemm:convlemm}
The identity $K_f *_{\mathcal G}K_g = K_{f*g}$
holds for all $f,g\in \mathcal L^1$, where $*_{\mathcal G}$ is the convolution introduced in~\eqref{eq:deficonv}.
\end{lemm}
\begin{proof}
This property follows by showing that 
\[ \int_{\mathcal G} K_f((\xi,\me),y)K_g(y,(\xi'',\me'')) \ dy = K_{g*f}((\xi,\me),(\xi'',\me'')).\]
Using \eqref{eq:kernel} it suffices to consider in the above integral the products of four different types of terms 
\begin{equation*}
\begin{split}
K_{h}^{(1a)}((\xi,\me),(\xi',\me'))&:=\sum_{P \in \mathcal P_{\me,\me'}(m)} \mathbb T_P h(\xi'+\vert P \vert-\xi) \\
K_{h}^{(1b)}(((\xi,\me),(\xi',\me'))&:= \sum_{P \in \mathcal P_{\me,-\me'}(m)} \mathbb T_P h(\vert \me' \vert-\xi'+\vert P \vert-\xi) \\
K_{h}^{(2a)}(((\xi,\me),(\xi',\me'))&:=\sum_{P \in \mathcal P_{-\me,\me'}(m)} \mathbb T_P h(\xi'+\vert P \vert-(\vert \me \vert-\xi))\text{ and }\\
K_{h}^{(2b)}((\xi,\me),(\xi',\me'))&:=\sum_{P \in \mathcal P_{-\me,-\me'}(m)} \mathbb T_P h(\vert \me' \vert-\xi'+\vert P \vert-(\vert \me \vert-\xi)).
\end{split}
\end{equation*}
It then follows that the individual terms of $K_{g*f}((\xi,\me),(\xi'',\me''))$, i.e.,
\begin{equation*}
\begin{split}
&K_{f*g}^{(1a)}\text{ is obtained from integrands }K_{f}^{(1a)}K_{g}^{(1a)} \text{ and } K_{f}^{(1b)}K_{g}^{(1b)}\\
&K_{f*g}^{(1b)}\text{ is obtained from integrands }K_{f}^{(1a)}K_{g}^{(1b)} \text{ and } K_{f}^{(1b)}K_{g}^{(2a)}\\
&K_{f*g}^{(2a)}\text{ is obtained from integrands }K_{f}^{(2a)}K_{g}^{(1a)} \text{ and } K_{f}^{(2b)}K_{g}^{(2a)}\\
&K_{f*g}^{(2b)}\text{ is obtained from integrands }K_{f}^{(2b)}K_{g}^{(2b)} \text{ and } K_{f}^{(2b)}K_{g}^{(1b)},
\end{split}
\end{equation*}
whereas integrals involving terms
\begin{equation*}
\begin{split}
&K_{f}^{(2b)}K_{g}^{(1b)}+K_{f}^{(2a)}K_{g}^{(2b)}, \quad K_{f}^{(1a)}K_{g}^{(2b)}+K_{f}^{(1b)}K_{g}^{(1a)} \text{ and }\\
&K_{f}^{(1a)}K_{g}^{(2a)}+ K_{f}^{(1b)}K_{g}^{(1a)}, \quad K_{f}^{(2a)}K_{g}^{(2a)}+ K_{f}^{(2b)}K_{g}^{(1a)}
\end{split}
\end{equation*}
vanish by symmetry.
\end{proof}

\section{A Gagliardo--Nirenberg inequality on metric graphs}
Reasoning as in \cite[Theorem 1]{Pro13}, it is straightforward that the Gagliardo-Nirenberg inequality holds on the infinite graph $\mathcal{G}$. We give the proof here for the sake of completeness.
\begin{prop}
Let $\mathcal G$ be a connected locally finite metric graph with edge lengths and conductivities uniformly bounded from below and from above. Let $u\in L^2(\mathcal G)\cap C(\mathcal G)$ such that $u^{(m)}\in L^2(\mathcal G)$, then the Gagliardo-Nirenberg inequality 
\begin{equation}\label{gng}||u||_{L^\infty}\leq C||u^{(m)}||_{L^2}^\frac{1}{2m}||u||_{L^2}^\frac{2m-1}{2m}\end{equation}
holds on $\mathcal G$.
\end{prop}
\begin{proof}
By density we can consider functions $u\in C_c^\infty(\mathcal G)$. Fix a vertex $\mv_0\in\mathcal V$ and consider the subgraph $G$ of $\mathcal G$ induced by the vertices in a neighborhood of $\mv_0$, more precisely: all vertices belonging to a ball $B_r(\mv_0)$ and all vertices that are adjacent to them (even if they are not in the ball), where $r$ is chosen in such a way that supp $u\subset B_r(\mv_0)$. We denote the first set of vertices by $V_{in}$ and the latter one by $V_{out}$, clearly $u(\mv)=0$ for $\mv\in V_{out}$. Now, it is possible to identify all vertices in $V_{out}$ with one sole vertex of this set; in such a way we obtain a new metric graph $\overline{G}$ having multiple edges between $V_{out}$ and some vertices in $V_{in}$ and some loops at $V_{out}$ that we can discard since $u\equiv0$ there.
 Now, for every finite graph $\overline{G}$ it is possible to construct a new graph $\widetilde{G}$ out of it by doubling each edge of $\overline{G}$, such that every vertex in the resulting graph $\widetilde{G}$ has even degree. If $\tilde{\me}$ is the new edge obtained doubling $\me$ we extend $u$ on $\widetilde{G}$ by $\widetilde{u}(\tilde{\me})=u(\me)$. Since $\widetilde{G}$ is connected and every vertex has even degree, it admits an Euler tour, i.e., a closed walk that traverses each edge exactly once. 
Therefore, $\widetilde{G}$ admits an Euler tour starting and terminating in $V_{out}$, and the function $\widetilde{u}:\widetilde{G}\to\mathbb R$ can be regarded as a function on the interval $(0, l)$, where the interval $(0, l)$ is obtained by concatenating the edges in their order of appearance during the Euler tour. Therefore, $\widetilde{u}\in L^2 (0, l)\cap C_0(0,l)$ and hence we can apply the classic Gagliardo-Nirenberg inequality for functions in $L^2(\mathbb R)$ with $m$-th derivative in $L^2(\mathbb R)$ to $\widetilde{u}$, c.f. \cite[page 12]{Nir59}. Then,
\begin{align*}||u||_{L^\infty(\mathcal G)}\leq||\widetilde{u}||_{L^\infty(\mathbb R)}&\leq c||\widetilde{u}^{(m)}||_{L^2(\mathbb R)}^\frac{1}{2m}||\widetilde{u}||_{L^2(\mathbb R)}^\frac{2m-1}{2m}\\&\leq 2^{|\widetilde{G}|} c||u^{(m)}||_{L^2(\widetilde{G})}^\frac{1}{2m}||u||_{L^2(\widetilde{G})}^\frac{2m-1}{2m}\\&\leq \tilde{c}||u^{(m)}||_{L^2(\mathcal G)}^\frac{1}{2m}||u||_{L^2(\mathcal G)}^\frac{2m-1}{2m}\end{align*}
since every edge of $\widetilde{G}$ occurs twice in the Euler tour and conductivities are uniformly bounded from above and below.
\end{proof}
\end{appendix}

\bibliographystyle{alpha}
\bibliography{literatur}

\end{document}